\documentclass{amsart}

\usepackage{amssymb}
\usepackage{amsmath}
\usepackage{amsfonts}
\usepackage{latexsym}
\usepackage{amsthm}
\usepackage{amscd}
\usepackage{enumerate}
\usepackage{graphics}
\usepackage{pstricks,pstricks-add,pst-math,pst-xkey}
\usepackage{subfigure}

\usepackage[margin=1in]{geometry}

 \newtheorem{thm}{Theorem}[section]
 \newtheorem{thm*}{Main Theorem}
 \newtheorem{cor*}[thm*]{Main Corollary}
 \newtheorem{cor}[thm]{Corollary}
 \newtheorem{lem}[thm]{Lemma}
 \newtheorem{prop}[thm]{Proposition}
 
 \theoremstyle{definition}
 \newtheorem{defn}[thm]{Definition}
 \theoremstyle{remark}
 \newtheorem{rem}[thm]{Remark}
 \newtheorem{exam}[thm]{Example}
 \numberwithin{equation}{section}

 \newcommand{\A}{\mathcal{A}}

 \newcommand{\C}{\ensuremath{\mathbb{C}}}

 \newcommand{\pp}{\ensuremath{\mathbb{P}}}

 \newcommand{\derp}[2]{\ensuremath{\frac{\partial #1}{\partial #2}}}

 \newcommand{\bir}{\ensuremath{\mathsf{\mathbf{Bir}^3}}}
 \newcommand{\birq}{\ensuremath{\mathsf{\mathbf{Bir}^3_2}}}
 \newcommand{\birr}[1]{\ensuremath{\mathsf{\mathbf{Bir}^3_{2,#1}}}}
 \newcommand{\lin}{\ensuremath{\mathsf{\mathbf{lin}}}}
 \newcommand{\jac}{\ensuremath{\mathrm{Jac}}}
 \newcommand{\gen}[1]{\ensuremath{\mathsf{\mathbf{gen}^{[2]}(#1)}}}
 \newcommand{\tang}[1]{\ensuremath{\mathsf{\mathbf{tan}^{[2]}(#1)}}}
 \newcommand{\osc}{\ensuremath{\mathsf{\mathbf{osc}^{[2]}(\times)}}}
 \newcommand{\ind}{\ensuremath{\mathrm{Ind}}}
 \newcommand{\exc}{\ensuremath{\mathrm{Exc}}}

\begin{document}

\title[Flows of birational quadratic transformations in $\pp^3$]
 {Flows of birational quadratic transformations in $\pp^3$}

\author{ M\`{o}nica Manjar\'{\i}n}

\address{IRMAR, Campus de Beaulieu, 35042 Rennes, France}

\email{monica.manjarin-arcas@univ-rennes1.fr}



\date{18th February 2010}

\dedicatory{}

\commby{}


\begin{abstract}
We study and classify up to a linear conjugation germs of flows in the space of quadratic birational transformations
of the complex projective space of dimension 3. As a consequence we show that every quadratic flow preserves a pencil
of planes through a line.
\end{abstract}

\maketitle

\section{Introduction}

Let $\pp^m$ denote the complex projective space of dimension $m$.
Consider $m+1$ polynomial functions $\varphi_i(x_0,...,x_m)$ of $m+1$ variables with complex coefficients, homogeneous
of the same degree, without common factors and not all identically zero. They define a \emph{rational transformation} $\varphi\in \mathsf{\mathbf{Rat}}^m$ by
$$\varphi:\pp^m\dashrightarrow\pp^m; \quad [x_0,...,x_m]\dashrightarrow [\varphi_0(x_0,...,x_m),\ldots, \varphi_m (x_0,...,x_m)].$$

We define the degree of $\varphi$, denoted by $\deg(\varphi)$, as the degree of the polynomials $\varphi_i$.
We say that $\varphi$ is a \emph{birational transformation} or a \emph{Cremona transformation} if there exists a rational transformation 
$\psi:\pp^m\dashrightarrow \pp^m$ such that $\psi\circ \varphi$ is the identity. The study of these transformations, introduced by L. Cremona, was initiated
in the 19th century by Cremona, Noether and Jonqui\`eres among others. The group of birational transformations of $\pp^m$
or \emph{Cremona group} will be denoted by $\mathsf{\mathbf{Bir}}^m$.

We denote by $\ind(\varphi)$ the \emph{indeterminacy locus} of $\varphi$, i.e. the points where all the polynomials $\varphi_0,\ldots,\varphi_m$ vanish. 
We define the \emph{Jacobian} of $\varphi$ as the determinant of the matrix $\big(\derp{\varphi_i}{x_j}\big)$ and it is denoted by $\jac(\varphi)$. 
Finally $\exc(\varphi)$ denotes the \emph{exceptional locus} of $\varphi$, i.e. zero set of $\jac(\varphi)$. 

\begin{exam}
The \emph{Cremona involution} $\sigma$ in $\pp^2$ is defined by $\sigma[x,y,z]=[yz,xz,xy]$. Its indeterminacy locus consists of three distinct points
$P_0=[1,0,0]$, $P_1=[0,1,0]$ and $P_2=[0,0,1]$ and $\sigma$ blows down the three coordinate lines $x=0$, $y=0$ and $z=0$, i.e. $\ind(\varphi)=\{P_0,P_1,P_2\}$ 
and $\exc(\varphi)=\{[x,y,z]\in \pp^2:\, x\cdot y \cdot z=0\}$.
\end{exam}

In $\pp^2$ birational transformations have been extensively studied, for monographs on the subject we refer the reader
to \cite{Hudson}, \cite{God}, \cite{Albe} and \cite{Des1}. Among the most significant results we find the following classical theorem 
proved by Noether in 1871:

\begin{thm}[Noether]
$\mathsf{\mathbf{Bir}^2}$ is generated by $\mathrm{Aut}_{\C}(\pp^2)=\mathrm{PGL}(3,\C)$ and the Cremona involution $\sigma$.
\end{thm}

Alternatively, Noether's theorem can be stated by saying that up to an element of $\mathrm{PGL}(3,\C)$ every birational transformation of
$\pp^2$ is a composition of quadratic transformations. 

In sharp contrast, in $\pp^m$ for $m\geq 3$, I. Pan proved the following in \cite{Pan} (a similar result had also been previously 
proved by H. Hudson for $m=3$, see \cite{Hudson}):

\begin{thm}[Pan]
 Let $m\geq 3$. Every set of generators of  $\mathsf{\mathbf{Bir}}^m$ must contain an uncountable number of transformations of
degree $>1$. 
\end{thm}

This result gives us a hint of the reason why so little is known on Cremona transformations in higher dimensions. 

There is a certain number of classical problems regarding the Cremona group, although they have been addressed mostly in dimension 2. 
We mention briefly some of them, the bibliography included is by no means exhaustive, see the monographs mentioned above for more
references. On the one hand we find the study of special families of transformations, such as involutions (see \cite{Bert} or \cite{BayBeau}), 
De Jonqui\`eres maps, i.e. elements of the plane Cremona group which preserve a fixed rational fibration (see \cite{Isk} for
a presentation of the Cremona group with generators the De Jonqui\`eres maps and $\mathrm{Aut}_{\C}(\pp^1\times \pp^1)$, for instance) or transformations
that preserve a fixed curve. Castelnuovo proved in 1892 (c.f. \cite{Cast}) that an element of $\mathsf{\mathbf{Bir}}^2$ preserving an
irreductible curve of genus $>1$ must be either conjugated to a De Jonqui\`eres map or of degree 2, 3 or 4 (see also \cite{BPV} for
a more recent and precise version). On the other hand, Kantor began the classification of finite subgroups up to a birrational conjugation 
(c.f. \cite{Kan}), a history of the classification of these groups can be found in the recent paper \cite{DolgIsk}.
Demazure and Umemura studied in the 70's and the 80's (see \cite{Dema} and \cite{Ume}, \cite{Ume1}, and \cite{Ume2} respectively) algebraic subgroups of 
$\mathsf{\mathbf{Bir}}^m$ of maximal rank (for $m=2$ and $m=3$ the problem goes back again to the end of the 19th century with the
work of Enriques and Fano, see \cite{EnrFan} and \cite{Fan}). More recently, Cerveau-D\'eserti have introduced a new approach to study maximal uncountable 
Abelian algebraic subgroups of $\mathsf{\mathbf{Bir}}^2$ (cf. \cite{CerDes}) by means of germs of flows. 

\begin{defn}
A \emph{germ of flow} in $\mathsf{\mathbf{Bir}}^m$ is a germ of holomorphic map $t\mapsto \varphi_t\in\mathsf{\mathbf{Bir}}^m$ such that
$$\varphi_{t+s}=\varphi_t \circ \varphi_s \quad \mathrm{and} \quad \varphi_0=\mathrm{id}.$$
From now on $\varphi_t$ will denote a germ of flow and we will talk of \emph{flow} for shortness.
\end{defn}

We denote by $\langle \varphi_t \rangle \subset \mathsf{\mathbf{Bir}}^m$ the subgroup generated by $\{ \varphi_t \}$ and by 
$\overline{\langle \varphi_t\rangle}^{Z}$ its Zariski closure in $\mathsf{\mathbf{Bir}}^m$. Cerveau-D\'eserti have studied the maximal 
Abelian algebraic subgroup of $\mathsf{\mathbf{Bir}}^2$ which contains $\overline{\langle \varphi_t\rangle}^{Z}$. Note that if that $G$ 
is a subgroup of $\mathsf{\mathbf{Bir}}^m$, then its Zariski closure in $\mathsf{\mathbf{Bir}}^m$ is algebraic and therefore has a finite 
number of connected components. Thus, if $G$ is uncountable the identity component of its closure $\overline{G}^{Z}$ must have dimension 
at least 1, which implies that $G$ contains flows. 

They essentially tackle the cases where the transformations are of small degree, that is quadratic or cubic. 
Although birational transformations of a given degree of $\mathsf{\mathbf{Bir}}^m$ do not form a subgroup (note that the degree 
of the composition might be bigger) they contain groups.  
In \cite{CerDes} they classify the quadratic flows up to linear conjugation and prove 
that every germ of flow (of arbitrary degree) preserves a fibration by lines. Their notation of flow (drooping the germ) is justified
by the classification, that shows that the time coordinate $t$ can actually be taken in $\C$.

The main tool consists roughly speaking in understanding the behavior of the familes of points obtained taking the indeterminacy
points of each quadratic transformation in the flow $\varphi_t$ and the famlies of curves obtained taking the elements of the exceptional locus of each
transformation $\varphi_t$. Another crucial ingredient, which allows to derive consequences for a flow of arbitrary degree, is a result
by Cantat-Favre (see \cite{CanFav}. Theorem 1.2) on the group of birational transformations which preserve a foliation on $\pp^2$.

In this paper we carry out a similar (although more involved) study on $\pp^3$, namely we classify up to linear conjugation germs of flows of quadratic 
birational transformations in $\pp^3$. We recall that in $\pp^2$ it is long been known that, 
up to linear automorphisms (acting on the right and on the left), i.e. up to changes of coordinates in the target and in the origin space,
there are only three quadratic transformations, depending on the number of points of $\ind(\varphi)$ (which might be 3, as for the Cremona involution,
2 or 1). Recently, an analogous classification was obtained for quadratic birational transformations in $\pp^3$ by 
Pan-Ronga-Vust (c.f. \cite{PRV}). We shall use that result as a departing point to achieve a classification of flows of 
quadratic birational transformations in $\pp^3$. As in $\pp^2$, the type of a quadratic transformations in $\pp^3$ up to linear automorphisms
is detemined by its indeterminacy locus. Note that we shall not be able to make use of any result regarding foliations on $\pp^3$, for no such
result is available, and hence we will restrict ourselves to the case of degree 2.

Recall that the degree of the inverse of a birational transformation $\varphi$ of $\pp^m$ is lower or equal than $(\deg \varphi)^{m-1}$.
The pair $ (\deg(\varphi), \deg (\varphi^{-1}))$ is called \emph{bidegree of a birational transformation} $\varphi$ and is denoted by
$\textrm{bideg}(\varphi)$. Note that inverses of quadratic transformations in $\pp^2$ are always quadratic, while in $\pp^3$ if we exclude the 
linear transformations we obtain three possibilities for the bidegree: $(2,2)$, $(2,3)$ and $(2,4)$. We denote by $\birq$ the set of quadratic 
birational transformations of $\pp^3$ and by $\lin$ the linear transformations in $\birq$, which can be seen in the following form: choose 
$\ell_0,\ell_1,\ell_2,\ell_3,\ell\in A_1$, where $A_1$ denotes the vector space of linear forms on $\C^4$, $\ell$ is not zero and 
$\ell_0,\ell_1,\ell_2,\ell_3$ are linearly independent, then $(\ell \ell_0,\ell \ell_1, \ell \ell_2, \ell \ell_3)\subset\birq$. The subset of 
of $\birq$ of transformations of bidegree $(2,2)$, which also contains $\lin$, is denoted by $\birr{2}$.

Note that if $\varphi_t$ is a germ of flow in $\birq$ then $\varphi_t \subset \birr{2}$. Indeed, given $\varphi_t\in \birq$ its inverse 
$\varphi_{-t}$ must be quadratic. 
It follows that for our purposes we can focus our attention on the transformations of $\birr{2}$, rather than on $\birq$.

\begin{defn}
Let $\varphi_t$ be a flow in $\bir$. We say that $\varphi_t$ is \emph{quadratic} if $\varphi_t\in \birr{2}$ and it is not contained in $\lin$ for
a generic $t$.
\end{defn}

The indeterminacy locus of a Cremona transformation $\varphi\in \birr{2}\backslash \lin$ is either $C_I \cup P_I$ or $C_I$, where $C_I$ 
is a plane conic and $P_I$ a point. In the first case we say that $\varphi$ is \emph{generic} (and the point $P_I$ is not contained in the plane 
of the conic) and in the second case that it is \emph{non-generic}. Alternatively 
we can define $P_I$ as the only point blown up to an hyperplane by $\varphi_t$, which determines the point $P_I$ even in the non-generic case, 
in which $P_I\in C_I$. When $\varphi_t$ is non-generic we distinguish two cases depending on whether the multiplicity of the conic $C_I$ at $P_I$ 
is 1 or 2. These three possibilities (generic, non-generic with multiplicity 1, and non-generic with multiplicity 2) determine three different 
behaviours of the germs of flows whose elements are of the corresponding types. Pan-Ronga-Vust prove that there are 7 types of transformations of 
$\birr{2}\backslash \lin$ up to changes of coordinates in the origin and target spaces that they denote by $\gen{O}$, $\gen{\times}$, $\gen{/\!/}$, 
$\tang{O}$, $\tang{\times}$, $\tang{/\!/}$ and $\osc$. The symbols $O$, $\times$ and $/\!/$ make reference to the type of the conic $C_I$. 
Briefly, the first three cases correspond to the generic case, $\tang{O}$ and $\tang{\times}$ are non-generic with conic $C_I$ of multiplicity 1 at $P_I$ and $\tang{/\!/}$ 
and $\osc$ are non-generic with conic $C_I$ of multiplicity 2 at $P_I$. We denote by $H$ the plane containing $C_I$ which is the only hyperplane contracted 
to a point and by $S$, if it exists, the only surface contracted to a curve by $\varphi$. 

We devote section 2 to studying the geometrical properties of the elements of $\birr{2}$ which are necessary for our purposes, namely the indeterminacy
and the exceptional loci. In section 3 we consider quadratic flows $\{\varphi_t\}$. Given a germ of flow we will consider the family
$\{ {P_I}_t \}$ of points, where ${P_I}_t$ is the only point blown up to a plane by $\varphi_t$, and analogously for the families $\{{C_I}_t\}$, 
$\{H_t\}$ and also, when it exists, $\{ S_t\}$. The behaviour of these families (whose elements belong either to the indeterminacy locus or 
to the exceptional locus of the corresponding $\varphi_t$) will play a central role in our proof of the classification of germs of quadratic transformations
in $\mathsf{\mathbf{Bir}}^3$. We prove the following:

\begin{thm*}
Let $\varphi_t$ be a quadratic flow in $\mathsf{\mathbf{Bir}}^3$. Then one of the following possibilities hold:
\begin{enumerate}[\bf a)]
\item $\varphi_t\in \gen{/\!/}\cup \lin$, $P_I$, $S$, $C_I$ are fix and $H_t$ is mobile,
\item $\varphi_t \in \tang{O}\cup \tang{\times}\cup \lin$, $P_I$, $S$ are fix and $H_t$, ${C_I}_t$ are mobile,
\item $\varphi_t \in \tang{/\!/}\cup \lin$, $H$, $C_I$ are fix and ${P_I}_t$ can be either fix or mobile,
\item $\varphi_t \in \osc \cup \lin$, $H$, $P_I$ are fix and $C_I$ can be either fix or mobile.
\end{enumerate}
Moreover, if $H$ is fix then $\varphi_t$ is a polynomial flow, i.e. ${\varphi_t}_{|\pp^3\backslash H}: \pp^3\backslash H \rightarrow \pp^3 \backslash H$ is
polynomial for each $t$. In particular there are no flows in $\gen{O}\cup \gen{\times}\cup \lin$.
\end{thm*}

This result gives us the possible configurations and will be the first step to achieve a classification. 

We also prove the following higher dimensional analogous in $\pp^3$ of the result in $\pp^2$ that states that any flow
preserves a rational fibration (see \cite{CerDes}):

\begin{thm*}
Let $\varphi_t$ be a quadratic flow in $\mathsf{\mathbf{Bir}}^3$. Then:
\begin{enumerate}[\bf i)]
\item There exists a line $L$ such that $\varphi_t$ preserves the family of hyperplanes through $L$. Moreover in cases {\bf a)} and {\bf c)} 
we can choose $L=C_I$.
\item If $P_I$ is fix then $\varphi_t$ preserves the family of hyperplanes through $P_I$ (in particular the family of lines through $P_I$).
\end{enumerate}
\end{thm*}

In the generic case we deduce a geometrical description of the behaviour of a flow:

\begin{cor*}
Let $\varphi_t$ a quadratic generic flow in $\mathsf{\mathbf{Bir}}^3$. Then $\varphi_t$ is determined by a linear flow $\eta_t$ on the $\pp^2$ of the net of lines 
through $P_I$ and a linear flow $\chi_t$ on the $\pp^1$ of the pencil of planes through the line $C_I$. Namely
$$\varphi_t(P)=\eta_t(P\vee P_I) \cap \chi_t(P\vee C_I).$$ 
\end{cor*}

Finally,

\begin{thm*}
Let $\varphi_t$ be a quadratic flow in $\mathsf{\mathbf{Bir}}^3$. Then, up to linear conjugation, $\varphi_t$ is in one (and only one) of the lists of theorems 
\ref{thm:NFgen} (generic case), \ref{thm:NFtangox1} (non-generic with conic $C_I$ of multiplicity 1 at $P_I$) and 
\ref{thm:NFtangllosc} (non-generic with conic $C_I$ of multiplicity 2 at $P_I$).  
\end{thm*}

In particular one concludes that if $\varphi_t$ is a germ of quadratic flow then $\varphi_t$ is defined for every $t\in \C$, 
which justifies the use of the term flow.  

\medskip
The author is indebeted to D. Cerveau for suggesting the problem and to both him and J. D\'eserti for explaining their work, for helpful 
discussions and hints and also for comments on this manuscript. This paper was written during a postdoctoral research stay at the IRMAR, 
the author would also like to thank the members of Analytic Geometry team for their warm hospitality. 

\section{$\birr{2}$}

In this section  we discuss results on $\birr{2}$, mainly from \cite{PRV}.
We are particularly interested in which points, curves and surfaces are contracted or exploded by a given
$\varphi\in \birr{2}$. We begin by recalling an elementary result:

\begin{lem}\label{lem:stop}
Let $S=\{s=0\}$ be a surface contracted to a point $P$ by $\varphi\in \bir\backslash \lin$, where $s$ is a
homogeneous polynomial. Then $\deg s=1$, i.e. $S$ is a hyperplane. Moreover there is at most one hyperplane
contracted to a point by $\varphi\in\birr{2}$.
\end{lem}

\begin{proof}
Up to a change of coordinates we can assume that $P=[0,0,0,1]$. Then if $S$ is contracted to $P$ we have
$$ \varphi_0=a_0\cdot s,\quad \varphi_1=a_1 \cdot s, \quad \varphi_2=a_2 \cdot s,$$
where $a_0,a_1,a_2$ are homogeneous polynomials. It follows that $\deg s\leq 2$. If $\deg s=2$ then
$a_0,a_1,a_2\in \C$ and $\varphi_0,\varphi_1,\varphi_2$ are linearly dependent, which contradicts
the hypothesis that $\varphi$ is birational. Assume now that $H=(h=0)$ and $H'=(h'=0)$ are two different hyperplanes contracted
to a point $P$. As before we can assume that $P=[0,0,0,1]$, then
$$\varphi_0=a_0\cdot h \cdot h', \quad \varphi_1=a_1 \cdot h \cdot h',\quad \varphi_2=a_2 \cdot h \cdot h'$$
where $a_0,a_1,a_2\in \C$, so again $\varphi_0,\varphi_1,\varphi_2$ would be linearly dependent, which is a
contradiction. If $\varphi(H)\neq \varphi(H')$ we would proceed in a similar way.
\end{proof}

\medskip

We recall next a result by Pan-Ronga-Vust (c.f. \cite{PRV}) and we derive some consequences. 
We denote by $A_2$ the vector space of quadratic forms on $\C^4$. Given $\varphi\in \birq$ such that in a given set of coordinates
it is written as $\varphi[x_0,x_1,x_2,x_3]=[\varphi_0,\varphi_1,\varphi_1,\varphi_2]$ we define $M_{\varphi}$ as the vector subspace
of $A_2$ of dimension 4 generated by $(\varphi_i)_{i=1}^3$. Then classifying elements of $\birr{2}$ up to linear changes of coordinates at
the origin and at the target space is equivalent to studying the action of GL$(4)$ on the subspaces $M$ of dimension 4 of $A_2$ such
that
$$\varphi_M:\pp^3\dashrightarrow \mathbb{P}(M^{\vee}),\qquad x\dashrightarrow \{f\in M: f(x)=0\}$$
is birational, where $M^{\vee}$ denotes the dual space of $M$.

\begin{prop}[Pan-Ronga-Vust]
Let $\varphi \in \birr{2}\backslash \lin$. Then there exist a quadric $Q=\{q=0\}$ of rank lower or equal than 3 and a plane $H=\{h=0\}$ such 
that if $C_I$ denotes the conic $Q\cap H$ then 
$$M_{\varphi}\subseteq \{f\in A_2:\, C_I\subset (f=0)\}.$$
\end{prop}

\begin{proof}
Let $q,q'$ two generic elements of $M_{\varphi}$. Then the intersection of the quadrics $(q=0)$ and $(q'=0)$ is the strict transform
$B$ by $\varphi^{-1}$ of a general line and an effective 1-cycle which depends only on $M$ and whose support is contained
in $\mathrm{Ind}(\varphi)$. The curve $B$ is rational and irreductible and as $\varphi^{-1}$ is quadratic $B$ is a smooth conic and 
we can assume that $B=(q(x_0,x_1,x_2)=x_3=0)$. Therefore there exist $h\in A_{1}$ and $\alpha\in \C$ such that $q'=\alpha q + x_3 h$. 
Then $B_0=(q=h=0)$ and $M_{\varphi}\subseteq \{ f\in A_2 :\, B_0\subset (f=0)\}$.
\end{proof}

\begin{prop} \label{cor:basicbir223}
Let $\varphi \in \birr{2}\backslash \lin$. With the same notation as above:
\begin{enumerate}[\bf (i)]
\item We can choose coordinates such that $H=(x_3=0)$, $Q=(q(x_0,x_1,x_2)=0)$ and $\ell_0,\ell_1,\ell_2,\ell_3\in A_1$ such that
$$\varphi[x_0,x_1,x_2,x_3]=[\ell_0 x_3, \ell_1 x_3, \ell_2 x_3, q+\ell_3 x_3].$$
\item The hyperplane $H$ is the only surface contracted to a point. We denote $P=\varphi(H)$.
\item There is exactly one point $P_I\in \ind (\varphi)$ which is blown up to a surface, which is a hyperplane and will be denoted by $H_I$.
\item We denote by $H^{-1}$ the hyperplane contracted to a point $P^{-1}$ by $\varphi^{-1}$ and by $P_I^{-1}\in \ind (\varphi^{-1})$ the
point blown up to a hyperplane $H_I^{-1}$. Then $$H=H_I^{-1},\quad H_I=H^{-1},\quad P=P_I^{-1}, \quad P_I=P^{-1}.$$
\label{prop:4iv}
\item Either $\ind(\varphi)=C_I$ or $\ind(\varphi)=C_I\cup P_I$. In the second case we will say that $\varphi$ is \emph{generic}.
\item $\varphi$ is generic if and only if we can choose coordinates such that
$$\varphi[x_0,x_1,x_2,x_3]=[x_0x_3,x_1x_3,x_2x_3,q(x_0,x_1,x_2)]$$
and $\varphi$ is non generic if and only if we can choose coordinates such that
$$\varphi[x_0,x_1,x_2,x_3]=[x_0x_3,x_1x_3+q(x_0,x_1,x_2),x_2x_3,x_3^2].$$
\item With the above coordinates, if $\varphi$ is generic then $$\jac(\varphi)=-2 x_3^2 q(x_0,x_1,x_2)$$ and if $\varphi$ is non-generic then 
$$\jac(\varphi)=2 x_3^3 \bigg(x_3+\derp{q}{x_1}\bigg).$$
\item If $\varphi$ is generic then the quadric $S=(q=0)$ is contracted to a conic $C$. Moreover $S$ is the only surface contracted to a curve
and $q\in M_{\varphi}$.
\item If $\varphi$ is non-generic and $\derp{q}{x_1}\neq 0$ then the hyperplane $\Pi=(x_3+\derp{q}{x_1}=0)$ is contracted to a conic $C$. Moreover 
$\Pi$ is the tangent plane to the quadric $x_1x_3+q\in M_{\varphi}$ at the point $(0,1,0,0)=P_I$ and it is the only surface contracted to a curve.
\item If $\varphi$ is non-generic and $\derp{q}{x_1}=0$ there exist no surfaces contracted to curves.
\item The image by $\varphi\in \birr{2} \backslash \lin$ of a hyperplane $L$ is a hyperplane if and only if $P_I\in L$ or $L\cap H\subset C_I$ 
(which implies $\mathrm{rank}\, C_I \leq 2$).\label{prop:4xi}
\item If $C_I$ is blowed to a quadric $S_I \not\supset H_I$ then the image by $\varphi$ of a quadric containing $C_I$ is a quadric.
\item The strict transform by $\varphi^{-1}$ of a line through $P_I$ is a conic through $P_I$.
\end{enumerate}
\end{prop}

\begin{proof}
It is clear that we can choose coordinates as in $\textbf{(i)}$. Using this expression we see that the hyperplane $H=(x_3=0)$ is
contracted to the point $P=[0,0,0,1]$, by the previous lemma, it must be the only surface contracted to a point, which proves
$\textbf{(ii)}$. To prove $\textbf{(iii)}$ note that a point is blown up to a surface if and only if its inverse $\varphi^{-1}\in \birr{2}$ contracts
the surface to the point. This implies that there is exactly one point with this property, denoted $P_I$, and that
it is blown up to an hyperplane, which we will denote by $H_I$. Moreover, $P_I=P^{-1}$ and $H_I=H^{-1}$ and analogously for the
other equalities in $\textbf{(iv)}$. We have $\varphi(H)=P$, therefore $H=\varphi^{-1}(P)$, which implies $P^{-1}_I=P$ and
$H_I^{-1}=H$. From $\varphi^{-1}(H^{-1})=P^{-1}$ we would prove the other equalities.
The assertion that $\ind(\varphi)$ is either the conic $C_I$ or the conic $C_I$ plus a point is clear in view of the previous
expression in coordinates, it is enough to prove that the extra point of $\ind (\varphi)$ in the generic case is exactly the
point $P_I$ blown up to a hyperplane. As this is an easy consequence of $\textbf{(ix)}$, we proceed to prove $\textbf{(vi)}$. 
Assume that $$\varphi[x_0,x_1,x_2,x_3]=[\ell_0 x_3, \ell_1 x_3, \ell_2 x_3, q+ \ell_3 x_3].$$ A point $R$ belongs to $\ind (\varphi) \backslash C_I$ 
if it verifies the equations:
$$\ell_0 x_3 = \ell_1 x_3 = \ell_2 x_3 = q + \ell_3 x_3 = 0$$
and  $x_3(R)\neq 0$. We now claim that if such an $R$ exists then
$$x_0'=\ell_0,\quad x_1'=\ell_1, \quad x_2'=\ell_2, \quad x_3'=x_3$$
is a change of coordinates. Indeed, if $\ell_0,\ell_1,\ell_2$ were not linearly independent then $\varphi$ would not be birational.
Moreover $\ell_0 (R)= \ell_1 (R)= \ell_2 (R)=0$. Therefore
$$\varphi[x'_0,x'_1,x'_2,x'_3]=[x'_0 x'_3, x'_1 x'_3, x'_2 x'_3, q'(x'_0,x'_1,x'_2,x'_3)]$$
and $R=[0,0,0,1]$. As $R\in \ind(\varphi)$ we conclude that $q'$ has no quadratic terms in $x'_3$ so we can
make a change of coordinates at the target space so that $q'$ does not depend on $x'_3$.
On the other hand, $\varphi$ is non generic if and only if $\ell_0$, $\ell_1$, $\ell_2$ and $x_3$ are linearly independent.
Indeed, for $\varphi$ to be birational $\ell_0$, $\ell_1$ and $\ell_2$ must be linearly independent and the equation
$\ell_0=\ell_1=\ell_2=0$ has a single point $R$ as solution. Then $R\in C_I$ if and only if $x_3(R)=q(R)=0$. Up to a change
of coordinates we can assume that $R=[0,1,0,0]$ and that
$$\varphi[x'_0,x'_1,x'_2,x'_3]=[x'_0x'_3, \varphi'_1(x'_0,x'_1,x'_2,x'_3), x'_2x'_3,x'^2_3].$$
As $\varphi'_1(R)=0$ the polynomial $\varphi'_1$ cannot have terms in $x'^2_1$ and we can also assume that the
only terms in $x'_3$ are of the form $x'_1x'_3$. This allows us to conclude.
The proof of $\textbf{(vi)}$ is a computation. Then $\textbf{(vii)}$ is consequence of the expressions $\textbf{(v)}$ and $\textbf{(vi)}$. 
For $\textbf{(ix)}$ note only that $q(x_0,x_1,x_2)$ does not contain the monomial $x_1^2$ because $P_I=(0,1,0,0)\in \ind(\varphi)$. The 
last statements are clear.
\end{proof}

Let $\varphi\in \birr{2}\backslash \lin$ generic, there are three possibilities for $q(x_0,x_1,x_2)$ which correspond
to rank of $C_I$ equal to 3, 2 and 1. Then one can choose:
$$q_O=x_0^2-x_1x_2, \qquad q_{\times}=x_1x_2, \qquad q_{/\!/}=x_2^2.$$
One checks readily that with these choices of coordinates at the origin and the target space one has $\varphi=\varphi^{-1}$. Similar arguments 
can be applied in the rest of the cases, yielding the following theorem:

\begin{thm}[Pan-Ronga-Vust] Let $\varphi \in \birr{2}\backslash \lin$. Up to a linear automorphism on the right and on the left one has 7 
possibilities for $\varphi$:
\begin{eqnarray*}
\gen{\alpha}&:=& \langle x_0x_3, x_1x_3,x_2x_3,q_{\alpha} \rangle \\
\tang{\alpha}&:=& \langle x_0x_3,x_2x_3, x_3^2,x_1x_3-q_{\alpha}\rangle \\
\osc&:=& \langle x_0x_3-x_1x_2, x_1x_3, x_2x_3,x_3^2 \rangle
\end{eqnarray*}
where $\alpha=O,\times$ or $/\!/$ with $q_O=x_0^2-x_1x_2$, $q_{\times}=x_1x_2$, $q_{/\!/}=x_2^2$.
\end{thm}

\begin{cor}
Let $\varphi \in \birr{2}\backslash \lin$. Then $\varphi$ and $\varphi^{-1}$ are in the same class with respect to composition by linear 
automorphisms on the right and on the left.
\end{cor}

\begin{cor}
Let $\varphi \in \birr{2}\backslash \lin$. Then:
\begin{enumerate}[\bf i)]
\item Assume $\varphi$ is generic. Then $C_I$ is blown up to a quadric $S_I$. Moreover, if $\varphi^{-1}$ refers to the corresponding 
element for $\varphi^{-1}$, then
 $$S=S_I^{-1}, \quad C=C_I^{-1}.$$ 
\item Assume $\varphi$ is non generic and $\derp{q}{x_1}\neq 0$, i.e. $\varphi\in \tang{O}\cap\tang{\times}$. 
Then $C_I$ is blown up to a plane ${\Pi_I}$. Moreover, if $^{-1}$ refers to the corresponding element for $\varphi^{-1}$ then 
$$\Pi=\Pi_I^{-1}, \quad C=C_I^{-1}.$$
\end{enumerate}
\end{cor}

\begin{proof}
Assume that $\varphi$ is generic (for the other cases similar arguments apply). As $\varphi$ and $\varphi^{-1}$ are in the same class 
there is a surface $S_I$ contracted to a conic $\widetilde{C}$ by $\varphi^{-1}$. As $\widetilde{C}\subset \ind(\varphi)$ we must have 
$\widetilde{C}=C_I$. Thus, $C_I$ is a conic blown up to a surface $S$. Moreover, $S=S_I^{-1}$ and $C=C_I^{-1}$.
\end{proof}

In the figures of the next page we represent the main geometrical objects of each type of transformations.
See the table in page \pageref{thetable} for a study of the seven cases. We will write $\varphi$ instead of 
$\varphi[x_0,x_1,x_2,x_3]$ and analogously for $\varphi^{-1}$. 

It will also be useful to have a geometric description of the elements of the \emph{linear system $\Gamma_{\varphi}$} associated to 
$\varphi$, i.e. the linear system of inverse images of hyperplanes by $\varphi$. Note that if $\varphi=(\varphi_0,\varphi_1,\varphi_2,\varphi_3)$ 
for some choices of coordinates then $(\varphi_i=0)=\varphi^{-1}(x_i=0)\in \Gamma_{\varphi}$. Recall that we say that
a quadric \emph{osculates} at a point $P$ along a germ of curve $C$ if it has a contact of order $2$ with $C$ at $P$, i.e.
if the multiplicity of intersection of the surface with $C$ at $p$ is 3).

\begin{center}
\begin{tabular}{|l|l|}
  \hline
  \textbf{Type} & \textbf{Definition of $\Gamma_{\varphi}$} \\ \hline \hline
  \gen{O} & Quadrics containing a smooth conic $C_I$ and a point $P_I\not\in C_I$ \\ \hline
  \gen{\times} & Quadrics containing a conic $C_I$ of rank 2 and a point $P_I\not\in C_I$ \\ \hline
  \gen{/\!/} & Cones containing a line $C_I$ and a point $P_I\not\in C_I$ and \\
  & tangents to a plane $S_I$ along $C_I$ \\ \hline
  \tang{O} & Quadrics containing a smooth conic $C_I$ and \\
  & tangents to a plane $S$ at a point $P_I\in \C_I$\\ \hline
  \tang{\times} &  Quadrics containing a conic $C_I$ of rank 2 and \\
  &tangents to a plane $S$ at a point $P_I\in \C_I$ \\ \hline
  \tang{/\!/} & Cones containing a line $C_I$, tangents to a plane $H$ \\
  & along $C_I$ and osculating at a point $P_I\in C_I$ \\
  & along a curve $\alpha$ tangent to $H$ at $P_I$\\ \hline
  \osc &  Quadrics containing a conic $C_I$ of rank 2, $C_I=L_1 \cup L_2$, \\
  & and osculating at the point $P_I=L_1\cap L_2$ along a curve \\
  & $\alpha$ tangent to the plane $H=L_1\vee L_2$ at $P_I$ \\  \hline
 \end{tabular}
\end{center}

\medskip

\renewcommand{\thesubfigure}{}

\begin{figure}[h]

\psset{xunit=0.5cm,yunit=0.5cm,algebraic=true,dotstyle=*,dotsize=3pt
0,linewidth=0.4pt,arrowsize=3pt 2,arrowinset=0.25} \scriptsize

\subfigure[$\gen{O}$]{\begin{pspicture*}(2.46,-3.8)(12.54,3.32)
 \psline(3.24,-3.04)(10.22,-3.12)
\psline(10.22,-3.12)(12.26,0.1) \psline[linestyle=dashed,dash=1pt
1pt](5.28,0.18)(12.26,0.1) \psline(5.28,0.18)(3.24,-3.04)
\rput{0.79}(7.81,-1.46){\psellipse[linestyle=dashed,dash=1pt
1pt](0,0)(2.17,0.93)} \psline(5.67,-1.35)(7.8,2.62)
\psline(9.97,-1.34)(7.8,2.62) \psline(5.28,0.18)(6.48,0.17)
\psline(9.17,0.12)(12.26,0.1) \rput[tl](4.16,-2.18){$H=H_I$}
\rput[tl](7.76,3.12){$P_I$} \rput[bl](7.74,-2.12){$C_I$}
\rput[bl](8.92,0.96){$S$} \psdots(7.8,2.62)
\end{pspicture*}}
\subfigure[$\gen{\times}$]{\begin{pspicture*}(2.49,-3.8)(12.55,3.38)
\psline(3.24,-3.04)(10.22,-3.12) \psline(10.22,-3.12)(12.26,0.1)
\psline[linestyle=dashed,dash=1pt 1pt](5.28,0.18)(12.26,0.1)
\psline(5.28,0.18)(3.24,-3.04) \rput[tl](4.16,-2.18){$H=H_I$}
\rput[tl](8.14,2.8){$P_I$} \psline(9.65,2.39)(6.56,1.57)
\psline(6.52,-1.97)(6.56,1.57) \psline[linestyle=dashed,dash=1pt
1pt](6.52,-1.97)(9.69,-0.86) \psline(9.69,-0.86)(9.65,2.39)
\psline(6.91,2.46)(9.36,1.63) \psline(9.36,1.63)(9.36,-1.9)
\psline(9.36,-1.9)(8.32,-1.34) \psline[linestyle=dashed,dash=1pt
1pt](8.32,-1.34)(6.85,-0.59) \psline[linestyle=dashed,dash=1pt
1pt](6.85,-0.59)(6.9,1.66) \psline(6.9,1.66)(6.89,2.45)
\psline[linestyle=dashed,dash=1pt 1pt](8.24,2.01)(8.32,-1.34)
\psline(5.28,0.18)(6.54,0.16) \psline(9.68,0.13)(12.18,0.1)
\psline(9.36,-0.98)(9.69,-0.86) \psline(6.52,-1.97)(8.32,-1.34)
\rput[lt](7.71,-1.75){\parbox{1.15 cm}{$C_I$}}
\rput[bl](9.81,0.77){$S$} \psdots(8.24,2.01)
\end{pspicture*}}
\subfigure[$\gen{/\!/}$]{\begin{pspicture*}(2.9,-3.8)(12.59,2.9)
\psline(3.24,-3.04)(10.22,-3.12) \psline(10.22,-3.12)(12.26,0.1)
\psline[linestyle=dashed,dash=1pt 1pt](5.28,0.18)(12.26,0.1)
\psline(5.28,0.18)(3.24,-3.04) \rput[tl](4.16,-2.18){$H=H_I$}
\rput[tl](8.16,1.0){$P_I$} \psline(9.65,2.39)(6.56,1.57)
\psline(6.52,-1.97)(6.56,1.57) \psline(9.69,-0.86)(9.65,2.39)
\psline(5.28,0.18)(6.54,0.16) \psline(9.68,0.13)(12.18,0.1)
\rput[lt](7.71,-1.75){\parbox{1.14 cm}{$C_I$}}
\psline(9.69,-0.83)(6.52,-1.95) \psline(9.83,-0.9)(9.79,2.35)
\psline(9.79,2.35)(6.7,1.53) \psline(6.66,-2.01)(6.7,1.53)
\psline(9.83,-0.86)(6.66,-1.99) \rput[bl](9.95,0.67){$S$}
\psdots(8,0.8)
\end{pspicture*}}

\end{figure}

\begin{figure}[h]

\psset{xunit=0.6cm,yunit=0.6cm,algebraic=true,dotstyle=*,dotsize=3pt
0,linewidth=0.4pt,arrowsize=3pt 2,arrowinset=0.25} \scriptsize
\newrgbcolor{cccccc}{0.8 0.8 0.8}

\subfigure[$\tang{O}$]{\begin{pspicture*}(2.76,-3.8)(13,3.2)
\psline(3.24,-3.04)(10.22,-3.12)
\psline(10.22,-3.12)(11.32,-0.02)
\psline[linestyle=dashed,dash=1pt 1pt](4.34,0.06)(11.32,-0.02)
\psline(4.34,0.06)(3.24,-3.04)
\rput{0.79}(7.81,-1.46){\psellipse[linestyle=dashed,dash=1pt 1pt](0,0)(2.17,0.93)}
\psline[linecolor=cccccc](5.67,-1.35)(7.8,2.62)
\psline[linecolor=cccccc](9.97,-1.34)(7.8,2.62)
\psline(4.34,0.06)(6.41,0.04)
\rput[tl](10.00,-0.9){$P_I$}
\psline(10.92,2.54)(10.9,0.16)
\psline(10.92,2.54)(9.96,-0.24)
\psline(10.9,0.16)(10.08,-2.66)
\psline(9.96,-0.24)(10.08,-2.66)
\psline[linestyle=dashed,dash=1pt 1pt](9.96,-0.24)(9.6,-1.74)
\psline[linestyle=dashed,dash=1pt 1pt](9.6,-1.74)(9.78,-3.88)
\psline[linestyle=dashed,dash=1pt 1pt](9.78,-3.88)(10.08,-2.66)
\psline(9.23,0)(10.04,-0.01)
\psline(11.32,-0.02)(10.85,-0.01)
\rput[bl](6.2,-2.78){$H$}
\rput[bl](7.74,-2.12){$C_I$}
\psdots(9.97,-1.34)
\rput[bl](11,1.88){$H_I=\Pi$}
\rput[tl](7.76,3.12){$Q$}
\end{pspicture*}}
\subfigure[$\tang{\times}$]{\begin{pspicture*}(2.08,-3.8)(13,2.87)
\psline(3.24,-3.04)(10.22,-3.12) \psline(10.22,-3.12)(12.26,0.1)
\psline[linestyle=dashed,dash=1pt 1pt](5.28,0.18)(12.26,0.1)
\psline(5.28,0.18)(3.24,-3.04) \rput[lt](8.17,-1.8){\parbox{1.16
cm}{$C_I$}} \psline[linestyle=dashed,dash=1pt
1pt](6.4,-0.77)(9.28,-2.05)
\psline[linecolor=cccccc](9.28,-2.05)(9.21,1.36)
\psline[linecolor=cccccc](6.4,-0.77)(6.44,2.37)
\psline[linestyle=dashed,dash=1pt 1pt](10.12,-0.79)(6.95,-2.38)
\psline[linecolor=cccccc](6.95,-2.38)(6.95,1.22)
\psline[linecolor=cccccc](6.95,1.22)(10.08,2.21)
\psline[linestyle=dashed,dash=1pt
1pt,linecolor=cccccc](10.08,2.21)(10.12,-0.79)
\psline[linecolor=cccccc](6.44,2.37)(9.21,1.36)
\psline[linestyle=dashed,dash=1pt
1pt,linecolor=cccccc](8.36,1.67)(8.39,-1.66)
\rput[tl](9.54,-0.26){$P_I$} \psline(10.12,-0.79)(11.8,1.51)
\psline(6.95,-2.38)(8.72,0.36) \psline(8.72,0.36)(11.8,1.51)
\psline(5.28,0.18)(6.41,0.17) \psline(10.78,0.12)(12.26,0.1)
\psline(6.4,-0.77)(6.95,-1.01) \psline(6.95,-2.38)(8.39,-1.66)
\psline(9.28,-2.05)(8.39,-1.66) \psline(9.27,-1.22)(10.12,-0.79)
\psline[linecolor=cccccc](8.36,1.66)(8.38,-0.17)
\psline[linestyle=dashed,dash=1pt
1pt,linecolor=cccccc](9.22,0.55)(8.39,-1.66)
\psline[linecolor=cccccc](10.08,2.21)(10.1,0.88)
\rput[bl](5.29,-2.73){$H$} \psdots(9.76,-0.97)
\rput[bl](11.43,0.58){$H_I=\Pi$}
\rput[tl](8.14,2.5){$Q$}
\end{pspicture*}}
\end{figure}

\begin{figure}[h]

\psset{xunit=0.6cm,yunit=0.6cm,algebraic=true,dotstyle=*,dotsize=3pt
0,linewidth=0.4pt,arrowsize=3pt 2,arrowinset=0.25} \scriptsize
\newrgbcolor{cccccc}{0.8 0.8 0.8}

\subfigure[$\tang{/\!/}$]{\begin{pspicture*}(2.67,-3.5)(12.52,2.5)
\psline(3.24,-3.04)(10.22,-3.12)
\psline(10.22,-3.12)(12.26,0.1)
\psline[linestyle=dashed,dash=1pt 1pt](5.28,0.18)(12.26,0.1)
\psline(5.28,0.18)(3.24,-3.04)
\rput[tl](4.2,-2.3){$H=H_I$}
\rput[tl](9.01,-1.21){$P_I$}
\psline[linecolor=cccccc](9.65,2.39)(6.56,1.57)
\psline[linecolor=cccccc](6.52,-1.97)(6.56,1.57)
\psline[linecolor=cccccc](9.69,-0.86)(9.65,2.39)
\psline(5.28,0.18)(6.54,0.16)
\psline(9.68,0.13)(12.18,0.1)
\rput[lt](7.71,-1.85){\parbox{1.16 cm}{$C_I$}}
\psline(9.69,-0.83)(6.52,-1.95)
\psline[linecolor=cccccc](9.83,-0.9)(9.79,2.35)
\psline[linecolor=cccccc](9.79,2.35)(6.7,1.53)
\psline[linecolor=cccccc](6.66,-2.01)(6.7,1.53)
\rput[bl](9.95,0.67){$Q$}
\psline(9.83,-0.86)(6.66,-1.99)
\psdots(8.91,-1.18)
\end{pspicture*}}
\subfigure[$\osc$]{\begin{pspicture*}(3.11,-3.5)(12.46,2.9)
\psline(3.24,-3.04)(10.22,-3.12)
\psline(10.22,-3.12)(12.26,0.1)
\psline[linestyle=dashed,dash=1pt 1pt](5.28,0.18)(12.26,0.1)
\psline(5.28,0.18)(3.24,-3.04)
\rput[tl](4.2,-2.3){$H=H_I$}
\rput[tl](8.19,-1.55){$P_I$}
\psline[linecolor=cccccc](9.65,2.39)(6.56,1.57)
\psline[linecolor=cccccc](6.52,-1.97)(6.56,1.57)
\psline[linestyle=dashed,dash=1pt 1pt](6.52,-1.97)(9.69,-0.86)
\psline[linecolor=cccccc](9.69,-0.86)(9.65,2.39)
\psline[linecolor=cccccc](6.91,2.46)(9.36,1.63)
\psline[linecolor=cccccc](9.36,1.63)(9.36,-1.9)
\psline(9.36,-1.9)(8.32,-1.34)
\psline[linestyle=dashed,dash=1pt 1pt](8.32,-1.34)(6.85,-0.59)
\psline[linestyle=dashed,dash=1pt 1pt,linecolor=cccccc](6.85,-0.59)(6.9,1.66)
\psline[linecolor=cccccc](6.9,1.66)(6.89,2.45)
\psline[linestyle=dashed,dash=1pt 1pt,linecolor=cccccc](8.24,2.01)(8.32,-1.34)
\psline(5.28,0.18)(6.54,0.16)
\psline(9.68,0.13)(12.18,0.1)
\psline(9.36,-0.98)(9.69,-0.86)
\psline(6.52,-1.97)(8.32,-1.34)
\rput[lt](7.2,-1.8){\parbox{1.15 cm}{$C_I$}}
\psdots(8.32,-1.34)
\rput[bl](9.81,0.77){$Q$}
\end{pspicture*}}
\end{figure}

\label{thetable}
\rotatebox{90}{\begin{tabular}{|c|c|c|c|c|}
  \hline \textbf{Type, expression} & $\mathbf{\ind(\varphi)}$ & $\mathbf{\jac(\varphi)},$ & $P_I, \, H_I$ & \textbf{$H$, $S$, $\Pi$, $C$}   \\
  \textbf{and inverse} & & $\mathbf{\jac(\varphi^{-1})}$ &   & \\  \hline\hline

   \gen{O}  & $(x_3=x_0^2-x_1x_2=0)$ & $-2x_3^2(x_0^2-x_1x_2),$   & $P_I$,   &$\varphi(H)=P=P_I,$\\
   \small{$\varphi=[x_0x_3,x_1x_3,x_2x_3,x_0^2-x_1x_2]$} & $\cup P_I$ & Idem & $H_I=(x_3=0)=H$ & $S=(x_0^2-x_1x_2=0)$, \\
   \small{$\varphi^{-1}=\varphi$} & with $P_I=[0,0,0,1]$ &&& $\varphi(S)=S\cap H=C$ \\ \hline

   \gen{\times} & $(x_3=x_1x_2=0)$ & $-2x_3^2x_1x_2,$  & $P_I$, &  $\varphi(H)=P=P_I,$\\
   \small{$\varphi=[x_0x_3,x_1x_3,x_2x_3,x_1x_2]$} & $\cup P_I$ & Idem &  $H_I=(x_3=0)=H$ & $S=(x_1x_2=0)=0$, \\
    \small{$\varphi^{-1}=\varphi$} & with $P_I=[0,0,0,1]$ &&& $\varphi(S)=S \cap H=C$ \\ \hline

    \gen{/\!/} & $(x_3=x_2^2=0)$ & $-2x_3^2x_2^2,$  & $P_I$, & $\varphi(H)=P=P_I,$\\
   \small{$\varphi=[x_0x_3,x_1x_3,x_2x_3,x_2^2]$} & $\cup P_I$ & Idem &  $H_I=(x_3=0)=H$ & $S=(x^2_2=0)$, \\
    \small{$\varphi^{-1}=\varphi$} & with $P_I=[0,0,0,1]$ &&& $\varphi(S)=S\cap H=C$ \\ \hline \hline

  \tang{O}  & $(x_3=x_0^2-x_1x_2=0)$ & $2x_3^3(x_2+x_3),$  &  $P_I=[0,1,0,0]$,  & $H=(x_3=0), \, \varphi(H)=P=P_I,$\\
  \small{$\varphi=[x_0x_3,x_1x_3+x_1x_2-x_0^2,x_2x_3,x_3^2]$} & & $2(x_2+x_3)^2x_3$ &  $H_I=(x_2+x_3=0)$  & $\Pi=(x_2+x_3=0)$, \\
   \small{$\varphi^{-1}=[x_0(x_2+x_3),x_1x_3+x_0^2,x^2_2+x_2x_3,x_2x_3+x_3^2]$} &&&& $\varphi(\Pi)=(x_2+x_3=x_0^2-x_1x_3=0)=C$ \\ \hline

  \tang{\times} & $(x_3=x_1x_2=0)$ & $2x_3^3(x_3-x_2),$  & $P_I=[0,1,0,0]$,  &$H=(x_3=0), \, \varphi(H)=P=P_I,$\\
  \small{$\varphi=[x_0x_3,x_1x_3-x_1x_2,x_2x_3,x_3^2]$} && $2(x_3-x_2)^3x_3$ & $H_I=(x_3-x_2=0)$ &  $\Pi=(x_2-x_3=0),$ \\
  \small{$\varphi^{-1}=[x_0(x_3-x_2),x_1x_3,x_2x_3-x^2_2,x^2_3-x_2x_3]$} &&&& $\varphi(\Pi)=(x_3-x_2=x_1=0)=C$ \\ \hline \hline

  \tang{/\!/} & $(x_3=x^ 2_2=0)$ & $2x_3^4,$  &  $P_I=[0,1,0,0]$, & $\varphi(H)=P=P_I$\\
  \small{$\varphi=[x_0x_3,x_1x_3-x_2^2,x_2x_3,x_3^2]$} && Idem &  $H_I=(x_3=0)=H$ & \\
  \small{$\varphi^{-1}=[x_0x_3,x_1x_3+x_2^2,x_2x_3,x^2_3]$} && && \\ \hline

   \osc & $(x_3=x_1x_2=0)$ & $2x_3^4,$   & $P_I=[1,0,0,0]$, & $\varphi(H)=P=P_I$ \\
   \small{$\varphi=[x_0x_3-x_1x_2,x_1x_3,x_2x_3,x_3^2]$} && Idem & $H_I=(x_3=0)=H$ &  \\
   \small{$\varphi^{-1}=[x_0x_3+x_1x_2,x_1x_3,x_2x_3,x^2_3]$} &&  && \\  \hline
\end{tabular}}

\begin{rem}
In the case $\varphi\in \tang{/\!/}$ we can take $\alpha$ as $(x_1 x_3 -x_2^2=x_0=0)$ (with the previous
expression) and for $\varphi\in \osc$ we can choose $\alpha=(x_0x_3-x_2^2=x_1-x_2=0)$. Note however that these
choices are not unique, for instance for $\varphi\in \tang{/\!/}$ we could also choose $\alpha$ as $(x_1 x_3 -x_2^2=x_0+x_2=0)$.
\end{rem}

\begin{prop}\label{thm:3types}
Let $\varphi$ be a quadratic rational map of $\pp^3$. Then $\varphi$ is birational of bidegree $(2,2)$ if and only if one of the following 
possibilities holds:
\begin{enumerate}[\bf i)]
\item $\jac(\varphi)=h^2 \cdot q$ with $h\in A_1$, $q\in A_2$ of rank $\leq 3$, $(h,q)=1$, $Sing\,(q=0)\not\subseteq (h=0)$ and such that 
$\varphi(h=0)=P$ is a point and $\varphi(q=0)=C$ a plane conic of the same rank as $C_I$ such that $P\not\in C$. Then $\varphi\in \gen{\alpha}$ 
and $\alpha=O,\times, /\!/$ with rank $\alpha=$ rank $q$.
\item $\jac(\varphi)=h^3 \cdot l$ with $h,l\in A_1$, $(h,l)=1$ and such that $\varphi(h=0)$ is a point $P$ and $\varphi(l=0)$ is a conic 
$C$ such that $P\in C$. Moreover there exists a point $P_I\in \ind(\varphi)$ such that for every $q\in \Gamma_{\varphi}$ the plane $\Pi=(\ell=0)$ 
is the tangent plane of the quadric $(q=0)$ at $P_I$. Then $\varphi\in \tang{O}\cup \tang{\times}$.
\item $\jac(\varphi)=h^4$ with $h \in A_1$ and such that $\varphi(h=0)$ is a point $P$ and there exists a point $P_I\in \ind(\varphi)$ such 
that the strict transform of a line in $\pp^3$ by $\varphi^{-1}$ is a curve by the point $P_I$. Then $\varphi\in \tang{/\!/}\cup \osc \cup \lin$.
\end{enumerate}
\end{prop}

\begin{proof}
It is enough to prove the converse implications.
\begin{enumerate}[\bf i)]
\item We can assume that $h=x_3$, $q=q(x_0,x_1,x_2)$, $C=(x_3=q=0)$ and $P=[0,0,0,1]$. Note that $(q=x_3=0)\subset C_I$. If $(h=0)$ is contracted 
to a point then we can assume that there exist $\ell_0,\ell_1,\ell_2\in A_1$ such that $\varphi_0=\ell_0 x_3$, $\varphi_1=\ell_1 x_3$ and 
$\varphi_2=\ell_2 x_3$. On the other hand as $(q=0)$ is contracted to the plane conic $C$ we have $\varphi_3=\alpha q$ with $\alpha\in \C^{\ast}$. 
Now, the condition $\jac(\varphi)=x_3^2 q$ implies that $\ell_0,\ell_1,\ell_2\in A_1(x_0,x_1,x_2)$ and that they are linearly independent. 
Therefore, we can choose coordinates $x'_0=\ell_0, x'_1=\ell_1, x'_2=\ell_2, x_3'=x_3$ such that
$$\varphi[x'_0,x'_1,x'_2,x'_3]=[x'_0x'_3, x'_1x'_3, x'_2x'_3, q'(x'_0,x'_1,x'_2)].$$
Then $\varphi$ is clearly birational and generic.
\item Without loss of generality we can assume that $h=x_3$, $\varphi(x_3=0)=[0,1,0,0]=P$, $l=x_2+x_3$, $\varphi(l)\subset (x_2+x_3=0)$ and 
$P_I=P=[0,1,0,0]$. As $\Pi$ is the tangent plane at the point $P_I$ we have that $\varphi_i=\alpha_i x_1(x_2+x_3)+q_i(x_0,x_2,x_3)$, with 
$\alpha_i\in \C$ and $q_i\in A_2(x_0,x_2,x_3)$. Using that $(x_3=0)$ is contracted to the point $P$  we conclude that there exist
    $\ell_0,\ell_2,\ell_3\in A_1(x_0,x_2,x_3)$ such that $\varphi_0=\ell_0 x_3$, $\varphi_2=\ell_2 x_3$ and $\varphi_3=\ell_3 x_3$. On the
    other hand, $\jac(\varphi)=(x_2+x_3)x_3^3$ implies that $\ell_0,\ell_2,\ell_3$ are linearly independent. Finally, $\Pi=(x_2+x_3=0)$ is 
contracted to a plane conic. Up to a linear change of coordinates at the origin and target space one can assume that
    $$\varphi[x'_0,x'_1,x'_2,x'_3]=[x'_0x'_3,x'_1(x'_2+x'_3)+\epsilon {x'}_0^2+(1-\epsilon) x'_0x'_2, x'_2x'_3, {x'}_3^2],$$
    with $\epsilon=0,1$. 
\item We can assume that $h=x_3$ and $P=P_I=[0,1,0,0]$. Then $\varphi_0=\ell_0 x_3$, $\varphi_1=\ell_1 x_3+ q$, $\varphi_2=\ell_2 x_3$, 
$\varphi_3=\ell_3 x_3$ with $\ell_i\in A_1(x_0,x_1,x_2,x_3)$ for $i=0,1,2,3$ and $q\in A_2(x_0,x_1,x_2)$. On the other hand
$$(\varphi_i=0) \cap (\varphi_j=0)=( x_3=q=0) \cup (\ell_i=\ell_1 x_3+q=0)$$
for $i=0,2,3$. As the strict transform by $\varphi^{-1}$ of the line $(x_1=x_i=0)$ is the curve 
$(\ell_i=\ell_1 x_3 + q=0)=C_i$ for $i=0,2,3$, imposing $P_I \in C_i$ yields $\ell_i \in A_1(x_0,x_2,x_3)$ for 
$i=0,2,3$ and that $q$ has no term in $x_1^2$. Therefore
$$\jac(\varphi)=\left|\begin{array}{cccc} \derp{\ell_0}{x_0} x_3 & \derp{\ell_1}{x_0} x_3+\derp{q}{x_0} & \derp{\ell_2}{x_0} x_3 & \derp{\ell_3}{x_0} x_3 \\ 0 & \derp{\ell_1}{x_1} x_3+\derp{q}{x_1} & 0 & 0\\
\derp{\ell_0}{x_2} x_3 & \derp{\ell_1}{x_2} x_3+\derp{q}{x_2} & \derp{\ell_2}{x_2} x_3 & \derp{\ell_3}{x_2} x_3 \\ \derp{\ell_0}{x_3} x_3 & \derp{\ell_1}{x_3} x_3+\derp{q}{x_3} & \derp{\ell_2}{x_3} x_3 & \derp{\ell_3}{x_3} x_3 \end{array}\right|=2 x_3^3 \bigg(\derp{\ell_1}{x_1}x_3+\derp{q}{x_1}\bigg)\left|\begin{array}{ccc} 
\derp{\ell_0}{x_0} & \derp{\ell_2}{x_0} & \derp{\ell_3}{x_0} \\ 
\derp{\ell_0}{x_2} & \derp{\ell_2}{x_2} & \derp{\ell_3}{x_2} \\ 
\derp{\ell_0}{x_3} & \derp{\ell_2}{x_3} & \derp{\ell_3}{x_3}
\end{array}\right|.$$
It follows that $\derp{q}{x_1}\neq 0$ and that $\{ \ell_i \}_{i=0}^3$ are linearly independent. Up to a linear change of coordinates on $x_0,x_2$ 
at the origin we can assume that $q$ is either $-x_2^2$ or $-x_0x_2$ (or $0$, which yields a linear transformation). Now, a change of coordinates 
at the target space allows to assume that $\ell_i=x_i$ for $i=0,1,2,3$. Therefore, $\varphi$ is birational and $\varphi\in \tang{/\!/}\cup \osc \cup \lin$.
\end{enumerate}
\end{proof}

\section{Flows in $\birr{2}$}

Let $\varphi_t$ be a quadratic flow.

We denote by $H_t$ the hyperplane contracted by $\varphi_t$ to a point $P_t:=\varphi_t(H_t)$, 
by ${C_I}_t$ the plane conic contained in $\ind (\varphi_t)$ (recall that ${C_I}_t\subset H_t$) 
and by ${P_I}_t$ the point blown up by $\varphi_t$ to a hyperplane ${H_I}_t$.

We denote, if it exists, the surface $S_t \neq H_t$ contracted to a plane conic $C_t$, and by 
${S_I}_t \neq H_t$, if it exists i.e if $\varphi_t\in \gen{\cdot}\cup \tang{O}\cup \tang{\times}\cup \lin$, 
the surface to which is blown up the conic ${C_I}_t$. Note that for $S_t \neq H_t$ (resp. ${S_I}_t\neq {H_I}_t$) 
we mean that there is at least one $t_0$ such that $S_{t_0} \neq H_{t_0}$ (resp. ${S_I}_{t_0}\neq {H_I}_{t_0}$). 
Note also that there exists a surface $S_t\neq H$ contracted to a curve $C_t$ if and only if the curve ${C_I}_t$ 
is blown up to a surface ${S_I}_t\neq {H_I}_t$, for a map $\varphi\in \birr{2}$ and its inverse $\varphi^{-1}$ 
are of the same type.

Note that the set of hyperplanes contracted to a point by $\varphi_t$, i.e. $\{H_t\}$, are a germ of analytic 
set in ${\pp^3}^\vee$ (the dual space of $\pp^3$). We can therefore consider the family of hyperplanes contracted 
to a point, and since there is a unique $H_t$ for each $t$ we have an analytic germ of map 
$$t\mapsto H_t \in {\pp^3}^\vee.$$

Moreover $H_t$ is well defined for $t=0$ for one can easily see that if 
$\varphi_0[x_0,x_1,x_2,x_3]=[\alpha x_0, \alpha x_1, \alpha x_2, \alpha x_3]$ then $H_t\rightarrow (\alpha=0)$ when
$t\rightarrow 0$. 

Analogously, we can consider the germ of map $t\rightarrow S_t$ where $S_t$ are the surfaces contracted to a conic by $\varphi_t$ 
(which might be quadrics, hyperplanes or not exist) in the set of quadrics of $\pp^3$ and the family of $S_t$. 
Similarly, we can consider the families of $P_t$, ${P_I}_t$, $H_t$, ${S_I}_t$ and ${C_I}_t$.

\begin{defn}
Let $\varphi_t$ be a quadratic flow and $H_t$ as above. We say that the family $H_t$ is \emph{fix} if it does not depend on 
$t$ and that it is \emph{mobile} otherwise. Analogously for the rest of the elements above.
\end{defn}

Nevertheless, ${P_I}_t$ and ${C_I}_t$ might not be defined for $t=0$ when the family is mobile, as we will see in the
classification of flows that we shall obtain (for $P_I$ see Theorem \ref{thm:NFtangllosc}, case \textbf{I} and for ${C_I}_t$ see 
Theorem \ref{thm:NFtangox1}, case \textbf{I}). 

Regarding ${S_I}_t$ we will see, also as a consequence of the classification, that if the family exists it is fix (see Theorem \ref{thm:fixmb}),
in particular, it will be defined for $t=0$.

\begin{exam}\label{exam:genll}
From the affine flow $(x,y,z)\mapsto \big(\frac{x}{1-tx},y+t,z+t,1 \big)$ we obtain
$$\varphi_t[x_0,x_1,x_2,x_3]=[x_0x_3, (x_1+tx_3)(x_3-tx_0), (x_2+tx_3)(x_3-tx_0), x_3 (x_3-tx_0)]$$
with
$\ind (\varphi_t)=(x_0=x_3=0) \cup \{[1,0,0,0]=: P_I\}, \,
C_I=(x_0=x_3=0),\, P_t=\varphi_t (H_t)=P_I, \, H_t=(x_3+tx_0=0),\,
S_t=(x_3=0)$ and $C_t=\varphi_t (S_t)=C_I.$
In particular $\varphi_t \in \gen{/\!/}$ for $t\neq 0$.
\end{exam}

\begin{exam}\label{exam:tango}
Let
$$\varphi_t[x_0,x_1,x_2,x_3]=[x_0(x_3+tx_0),e^tx_1x_3+(e^t-1)x^2_2,x_2(x_3+tx_0),(x_3+tx_0)^2].$$
Then
$\ind(\varphi_t)=(x_3+tx_0=e^tx_1x_3+(e^t-1)x_2^2=0)={C_I}_t, \,
P_I=[0,1,0,0],\, H_t=(x_3+tx_0=0)$ and $S_t=(x_3=0)$.
In particular $\varphi_t\in \tang{O}$ for $t\neq 0$.
\end{exam}

\begin{exam}\label{exam:tangx}
Let
$$\varphi_t[x_0,x_1,x_2,x_3]=[x_0(x_3+tx_0),e^tx_1x_3,x_2(x_3+tx_0),(x_3+tx_0)^2].$$
Then $\ind(\varphi_t)={C_I}_t=(x_1=x_3+tx_0=0)\cup (x_3=x_0=0),\, 
H_t=(x_3+tx_0=0),\, P_I=[0,1,0,0]$ and $S_t=(x_3=0)$.
In particular $\varphi_t\in \tang{\times}$ for $t\neq 0$.
\end{exam}

\begin{exam}\label{exam:tangll}
Let
$$\varphi_t[x_0,x_1,x_2,x_3]=[tx_3^2+x_0x_2, (x_1+tx_2)x_2, x_2^2, x_2x_3].$$
Then $\ind (\varphi_t)=(x_2=x_3^2=0),\, H_t=(x_2=0),\, P_t=\varphi_t(H_t)=[1,0,0,0]=P_I$ and 
$C_I=(x_2=x_3^2=0)$.
In particular $\varphi_t\in \tang{/\!/}$ for $t\neq 0$.
\end{exam}

\begin{exam}\label{exam:osc}
Let $$\varphi_t[x_0,x_1,x_2,x_3]=[x_0x_3-tx_1x_2, x_1x_3, x_2x_3, x_3^2].$$
Then
$\ind (\varphi_t)=(x_3=x_1x_2=0)=C_I,\, H=H_I=(x_3=0)$ and $P_I=[1,0,0,0]$.
In particular $\varphi_t\in \osc$ for $t\neq 0$.
\end{exam}

In fact we will show that there are no flows in $\gen{O}\cup \gen{\times}\cup \lin$.

\begin{rem}\label{rem:fixmb}
Let $\varphi_t$ be a quadratic flow. Then:
\begin{enumerate}[\bf i)]
\item $P_t$ is fix if and only if ${P_I}_t$ is fix. Moreover in this case $P=P_I$.
\item $H_t$ is fix if and only if ${H_I}_t$ is fix. Moreover in this case $H=H_I$.
\item $C_t$ is fix if and only if ${C_I}_t$ is fix. Moreover in this case $C=C_I$.
\item $S_t$ is fix if and only if ${S_I}_t$ is fix. Moreover in this case $S=S_I$.
\end{enumerate}
\end{rem}

Note that the condition that $H_t$ is fix, i.e. $H_t=H$, does clearly not imply that $P_t=\varphi_t(H)$ is fix.

\begin{proof}
It is a consequence of corollary \ref{cor:basicbir223}, since ${P_I}_t=P^{-1}_t=P_{-t}$.
One proceeds analogously to prove the other cases.
\end{proof}

\begin{lem}\label{lem:PmobHfix}
Let $\varphi_t$ be a quadratic flow. If ${P_I}_t$ is a mobile point blown up by $\varphi_t$ into ${H_I}_t$, then
${H_I}_t$ is fix.
\end{lem}

\begin{proof}
We have
$$\varphi_t(\varphi_s({P_I}_t))=\varphi_s(\varphi_t({P_I}_t))=\varphi_s({H_I}_t).$$
Note that as ${P_I}_t$ is mobile we can assume that $\varphi_s({P_I}_t)$ is a point.
We will discuss the different possibilities:
\begin{enumerate}[\bf 1)]
\item $\varphi_t(\varphi_s({P_I}_t))$ is a point and ${H_I}_t=H_s$, then $H_{-t}=H_s$ and we conclude that $H_t$
is fix.
\item $\varphi_s({P_I}_t)={P_I}_t$ and $\varphi_s({H_I}_t)={H_I}_t$. However if ${H_I}_t$ is mobile then
$\varphi_s({H_I}_t)$ must be generically a quadric (note that $\varphi_s({P_I}_t)={H_I}_t$  is a contradiction
for ${P_I}_t$ is mobile and this would imply ${P_I}_t={P_I}_s$) unless ${P_I}_s\in {H_I}_t$ or ${C_I}_s\subset {H_I}_t$
for every $t,s$. All is left to do is to exclude these two cases. Note that it is equivalent to assume that
${P_I}_s\in H_t$ or that ${C_I}_s\subset H_t$ respectively for every $t,s$ for $H_t={H_I}_{-t}$.
\begin{itemize}
\item Let us begin by showing that there are no quadratic flows $\varphi_t$ such that ${P_I}_t$, $H_t$ are mobile,
${P_I}_t\in {H_s}$, $\varphi_s ({P_I}_t)={P_I}_t$ and $\varphi_s(H_t)=H_t$.

Note that the flow $\varphi_t$ must be non-generic. Moreover, we must have ${P_I}_t\subset L$ line and $H_t$ contained
in the family of hyperplanes through $L$. The points of the segment of $L$ described by ${P_I}_t$ are fixed by $\varphi_s$
and the plans through $L$ too. Without loss of generality we can assume that $L=(x_0=x_1=0)$. Then
$H_t=( \alpha_t x_0 + \beta_t x_1 =0 )$ with $\alpha_t,\beta_t\in \C$ and
$$\varphi_t[x_0,x_1,x_2,x_3]=[x_0 \cdot \ell_t, x_1 \cdot \ell_t, x_2 \cdot \ell'_t, x_3\cdot \ell'_t]$$
with $\ell_t,\ell'_t\in A_1(x_0,x_1,x_2,x_3)$. Moreover $\ind (\varphi_t)=(x_0=x_1=\ell'_t=0)\cup (x_2=x_3=\ell_t=0)$.
If $x_0,x_1,\ell'_t$ are linearly independent their intersection is a point which does not belong to
$(x_2=x_3=\ell_t=0)$. We conclude that
$$\ell'_t=a_t x_0+b_t x_1; \qquad \ell_t=c_t x_2+d_t x_3.$$
Therefore this would imply that $\ind(\varphi_t)$ are two lines in $\pp^3$ that do not intersect, which is a contradiction.

\item Let now see that there are no quadratic flows $\varphi_t$ such that $C_I$ is a fix line, $C_I\subset H_t$ for every $t$, 
$H_t$ and ${P_I}_t$ are mobile, $\varphi_s ({P_I}_t)={P_I}_t$ and $\varphi_s(H_t)=H_t$.

We can assume that $\varphi_t\in \gen{/\!/}\cup \lin$ and that ${P_I}_t\notin C_I$ for any $t$, otherwise we are in the previous case. 
The hyperplane $S_t=C_I \vee {P_I}_t$ is contracted by $\varphi_t$ into $C=C_I$. Moreover, as 
$\varphi_s({P_I}_t)={P_I}_t$ and $\varphi_s(S_t)$ must be a hyperplane through $C_I$ (by proposition \ref{cor:basicbir223} 
\ref{prop:4xi}) we conclude that $\varphi_s(S_t)=S_t$ for every $t,s$. As 
$${S_I}_t=\varphi_s({S_I}_t)=\varphi_s \circ \varphi_t (C_I)= \varphi_t \circ \varphi_s(C_I)=\varphi_t({S_I}_s)={S_I}_s$$
we conclude that the hyperplanes $S$ and $S_I$ are fix, which contradicts $\varphi_s(S_t)=C_I$.
\end{itemize}
\item $\varphi_s({P_I}_t)\subset \ind (\varphi_t) \backslash {P_I}_t$. In particular, $\varphi_t$ is generic. 
Then $\varphi_t(\varphi_s({P_I}_t))={P_I}_t\vee \varphi_s({P_I}_t)$. Therefore, ${H_I}_t$ is contracted to the line 
${P_I}_t \vee \varphi_s({P_I}_t)$, which implies that $\varphi_t \in \gen{/\!/}\cup \lin$ and $H_t=S_s$ for every $t,s$. Again we have 
reached a contradiction for ${P_I}_t\in S_t=H_t\not\ni {P_I}_t$. 
\end{enumerate}
\end{proof}

\begin{cor}
Let $\varphi_t$ be a quadratic flow. If $H_t$ is mobile then ${P_I}_t$ is fix.
\end{cor}

\begin{cor}
Let $\varphi_t$ be a quadratic flow. Then either $H_t$ or ${P_I}_t$ are fix.
\end{cor}

\begin{lem}
Let $\varphi_t$ be a quadratic flow in $\gen{\cdot}\cup \lin$. If ${C_I}_t$ is fix then $\mathrm{rank}\, C_I\leq 1$.
\end{lem}

\begin{proof}
Let $\varphi_t=[{\varphi_0}_t,{\varphi_1}_t,{\varphi_2}_t,{\varphi_3}_t]$, if ${C_I}$ has rank greater or equal than 2
one has ${C_I}=(q=\ell=0)$ where $q$ is a homogeneous quadratic polynomial of rank lower or equal than 3, $\ell\in A_1$, 
$(\ell, q)=1$ and ${\varphi_i}_t={\alpha_i}_t q +{\ell_i}_t \ell$ where ${\alpha_i}_t\in \C$, ${\ell_i}_t\in A_1$. 
As $\varphi_0=Id$ the quadratic polynomials ${\alpha_i}_0 q + {\ell_i}_0 \ell$ must
have a common factor. There are two possibilities. If ${\alpha_i}_0=0$ for $i=0,1,2,3$ then for $t$ small enough ${\ell_i}_t$ 
are linearly independent, which implies that $\ind (\varphi)=(\ell=q=0)$, i.e., $\varphi$ is not generic. Otherwise, we use that
$${\alpha_j}_t {\varphi_i}_t - {\alpha_i}_t {\varphi_j}_t = ({\alpha_j}_t {\ell_i}_t - {\alpha_i}_t {\ell_j}_t )\cdot \ell$$
and not all the coefficients ${\alpha_j}_t {\ell_i}_t - {\alpha_i}_t {\ell_j}_t$ can be zero (recall that
$\dim\langle {\ell_i}_t \rangle =3$). As $\varphi_0=\mathrm{Id}$ we deduce that $q=\ell \cdot \ell'$ with
$\ell'\in A_1$. Indeed,
$$({\alpha_j}_0 x_i - {\alpha_i}_0 x_j) \cdot \alpha={\alpha_j}_0 {\varphi_i}_0 - {\alpha_i}_0 {\varphi_j}_0 = ({\alpha_j}_0 {\ell_i}_0 - {\alpha_i}_0 {\ell_j}_0 )\cdot \ell$$
with $\alpha\in A_1$ such that $\varphi_0=(\alpha x_0, \alpha x_1, \alpha x_2, \alpha x_3)$. It follows that
$\alpha=a \cdot \ell$ with $a\in \C^{\ast}$. Therefore, $\ell$ must divide $q$, i.e. $q=\ell\cdot \ell'$ with $\ell'\in A_1$, which
is a contradiction.
\end{proof}


\begin{lem}
Let $\varphi_t$ be a quadratic flow in $\gen{\cdot}\cup \lin$. Then either $H_t$ or ${P_I}_t$ are mobile.
\end{lem}

\begin{proof}
We assume that $P_I=[1,0,0,0]$ is fix and that $H_t=(x_0=0)$ is fix. Using the expressions of the previous section we can
assume that
$$\varphi_t[x_0,x_1,x_2,x_3]=
\bigg[ \frac{{g_2}_t(x_1,x_2,x_3)+x_0 {g_1}_t(x_1,x_2,x_3)}{x_0},
{\ell_1}_t, {\ell_2}_t,{\ell_3}_t \bigg],$$
where $[x_1,x_2,x_3]\mapsto [{\ell_1}_t, {\ell_2}_t,{\ell_3}_t]$ is a plane Cremona transformation, so clearly it
is impossible to have $\varphi_0=\mathrm{Id}$.
\end{proof}

\begin{lem}
Let $\varphi_t$ be a quadratic flow. If ${C_I}_t$ is a mobile plane conic blown up by $\varphi_t$ into a surface ${S_I}_t \neq {H_I}_t$, then ${S_I}_t$ is fix.
\end{lem}

\begin{proof}
We have
$$\varphi_t( \varphi_s ({C_I}_t))=\varphi_s (\varphi_t ({C_I}_t))=\varphi_s ({S_I}_t).$$
If $C_t$ is mobile we can assume that $\varphi_s({C_I}_t)$ is a mobile curve. There are two
possibilities:
\begin{enumerate}[\bf 1)]
\item $\varphi_s ({C_I}_t)={C_I}_t$ and then either ${C_I}_t$ is a curve blown up by $\varphi_s$ into
a surface $\varphi_s ({S_I}_t)$, which implies ${C_I}_t={C_I}_s$ (a contradiction for we are assuming
that ${C_I}_t$ is mobile) or ${S_I}_t=S_s={S_I}_{-s}$ and ${S_I}_t$ and $S_t$ are fix.
\item $\varphi_t( \varphi_s ({C_I}_t))$ is a curve and then ${S_I}_t=S_s={S_I}_{-s}$ and ${S_I}_t$ and $S_t$ are fix.
\end{enumerate}
\end{proof}

\begin{cor}
Let $\varphi_t$ be a quadratic flow. If there exists a mobile family of surfaces $S_t$ contracted to
a curve then ${C_I}_t$ is fix.
\end{cor}

\begin{cor}
Let $\varphi_t$ be a quadratic flow in $\gen{\cdot}\cup \lin$. Then either $S_t$ or $H_t$ are fix.
\end{cor}

\begin{proof}
We have already noted that if both $H_t$ and $S_t$ are mobile then $P_I$ and $C_I$ are fix, and since $P_I \vee C_I =S_t$ we
conclude that $S_t$ is fix, which contradicts the hypothesis.
\end{proof}

\begin{cor}\label{cor:tangox}
Let $\varphi_t$ be a quadratic flow in $\tang{\times}\cup \tang{O}\cup \lin$. Then either $S_t$ or $H_t$ are fix.
\end{cor}

\begin{proof}
Note that ${H_I}_t$ mobile implies ${C_I}_t$ mobile, whereas ${S_I}_t$ mobile implies $C_I$ fix.
\end{proof}

\begin{lem}
Let $\varphi_t$ be a quadratic flow in $\gen{\cdot}\cup \lin$. Then either $S_t$ or $H_t$ are mobile.
\end{lem}

\begin{proof}
Assume now that both $H_t$ and $S_t$ are fix. Without loss of
generality we can choose coordinates such that $H=(x_3=0)$ and $S=(q(x_0,x_1,x_2)=0)$, where $q$ is an homogeneous polynomial
of degree 2. Therefore if $\varphi_t=({\varphi_0}_t,{\varphi_1}_t, {\varphi_2}_t,{\varphi_3}_t)$ we have
$${\varphi_i}_t={\alpha_i}_t q + {\ell_i}_t x_3 \qquad (\ast)$$
with ${\ell_i}_t\in A_1(x_0,x_1,x_2,x_3)$ such that $\dim \langle {\ell_i}_t \rangle_t =3$ for every $t$ such that $\varphi_t$ is
generic and ${\alpha_i}_t\in \C$. As $\varphi_0=Id$ we have that
$${\alpha_i}_0 q + {\ell_i}_0 x_3= \ell x_i$$
for $i=0,1,2,3$ and $\ell\in A_1$. If ${\alpha_i}_0=0$ for $i=0,1,2,3$ then $\ell$ must be a constant multiple of $x_3$ and
${\ell_i}_0$ must be linearly independent. Therefore for $t$ small enough ${\ell_i}_t$ would be linearly independent and $\varphi_t$ is not generic,
 which is a contradiction. We conclude that at least one ${\alpha_i}_0\neq 0$, but then $\ell$ would be a constant multiple of $x_3$ and it is 
easy to see that the condition $(\ast)$ would lead to contradiction.
\end{proof}


\subsection{Generic flows}

\begin{thm}\label{thm:gen}
Let $\varphi_t$ be a quadratic flow in $\gen{\cdot}\cup \lin$. Then
\begin{enumerate}[\bf i)]
\item $H_t$ is mobile and $P_I$, $S$ and $C_I$ are fix.
\item $\varphi_t\in \gen{/\!/}\cup \lin$.
\item $\varphi_t$ preserves the family of hyperplanes through $C_I$, which contains $\{H_s\}$.
\item $\varphi_t$ preserves the family of hyperplanes through $P_I$, which does not contain the
family of hyperplanes through $C_I$.
\end{enumerate}
In particular, there are no quadratic flows in $\gen{O}\cup \gen{\times}\cup \lin$.
\end{thm}

\begin{proof}
Assume that $H_t$ is mobile, then $S_t$ and $P_t$ are fix and $P=P_I$ (by proposition \ref{cor:basicbir223} \ref{prop:4iv}). 
Moreover, all the hyperplanes through the point $P_I$ are sent by $\varphi_t$ to hyperplanes through the point $P$. 
As $\varphi_t$ and $\varphi_s$ commute we have
$$\varphi_t(\varphi_s(H_t))=\varphi_s(\varphi_t(H_t))=\varphi_s(P)=\varphi_s(P_I)=H_s.$$
As $H_t$ is mobile $H_t \neq H_s$ and we have the following possibilities:
\begin{itemize}
\item $H_t\subset S_s$ (in particular rank ${C_I}_t\leq 2$) and $\varphi_s(H_t)\subset {C_I}_s$.
In this case, as $H_t$ is mobile we must have $S_t$ mobile, which is a contradiction. 
\item $H_t \not\subset S_s$. Then $\varphi_s(H_t)$ is a hyperplane sent by $\varphi_t$ to the hyperplane
$H_s$. This implies that $\varphi_s(H_t)$ contains either $P_I$ or ${C_I}_t$. On the other hand
$\varphi_s(H_t)$ is a hyperplane if and only if $P_I\in H_t$ or ${C_I}_s\subset H_t$. The first option is 
impossible for we are in the generic case. If rank ${C_I}_s \geq 2$ then ${C_I}_s \subset
H_t$ implies $H_s=H_t$, which is also impossible for we are assuming that $H_t$ is mobile. Therefore 
rank ${C_I}_t=1$. Finally, as ${C_I}_s\subset H_t$ for every $t$ and $s$ and $H_t$ is mobile, we
conclude that ${C_I}_t$ is fix, and that rank $C_I=1$ and clearly $\varphi_t$ preserves the family of 
hyperplanes which contain $C_I$, which includes the $H_t$. Moreover $\varphi_t\in \gen{/\!/}\cup \lin$.
\end{itemize}

Assume now that $H_t$ is fix. Then $S_t$ is mobile and $C_I$ is fix. We have seen that a generic flow with $C_I$ fix verifies rank
$C_I=1$. Therefore $\varphi_t \in \gen{/\!/}\cup \lin$. Without loss of generality we can assume that
$$H=(x_1=0), \quad C_I=(x_0=x_1=0), \quad S_t=((x_0+a(t) x_1)^2=0).$$
Then if $\varphi=(\varphi_0,\varphi_1,\varphi_2,\varphi_3)$ we have
$${\varphi_i}_t={\ell_i}_t \cdot x_1+ {\alpha_i}_t \cdot (x_0+ a(t) x_1)^2$$
with ${\alpha_i}_t\in \C,\, {\ell_i}_t\in A_1(x_0,x_1,x_2,x_3)$ for $i=0,1,2,3$. As usual
we impose the condition $\varphi_0=\mathrm{Id}$. For ${\varphi_0}_0$ there are two possibilities:
\begin{itemize}
\item ${\alpha_0}_0=0$ and ${\varphi_0}_0=\alpha x_0 x_1$ with $\alpha\in \C^{\ast}$,
\item ${\alpha_0}\neq 0$ and ${\varphi_0}_0= \alpha x_0 (x_0+a(0) x_1)$ with $\alpha\in \C^{\ast}$.
\end{itemize}
In both cases we conclude that ${\alpha_i}_0=0$ for $i=1,2,3$ (because otherwise a term in $x_0^2$ would
appear). It is not difficult to see that then ${\alpha_0}_0=0$ and ${\ell_i}_0$ are linearly independent, which
is a contradiction with the hypothesis of $\varphi$ generic.

Let us finally prove \textbf{iv)}. We know that $C_I$, $S$ and $P_I$ are fix and that $H_t$ is mobile.
Without loss of generality we can assume that $P_I=[1,0,0,0]$, $S=(x_3^2=0)$, $C_I=(x_0=x_3=0)$ and
$H_t=( \ell_t(x_0,x_3)=a(t)x_0+b(t)x_3=0 )$. Then
$$\varphi_t[x_0,x_1,x_2,x_3]=[\alpha_t \cdot x_3^2+\ell_t \cdot {\ell_0}_t, \ell_t \cdot {\ell_1}_t, \ell_t \cdot {\ell_2}_t, \ell_t\cdot {\ell_3}_t]$$
with $\alpha_t\in \C^{\ast}$ and ${\ell_i}_t\in A_1(x_0,x_1,x_2,x_3)$ for $i=0,1,2,3$. Moreover, $({\ell_1}_t={\ell_2}_t={\ell_2}_t=0)=[1,0,0,0]$, which
implies that ${\ell_i}_t$ does not depend on $x_0$ for $i=1,2,3$. Therefore the family of hyperplanes through $P_I$ is
invariant by $\varphi_t$, i.e. if $L$ is a hyperplane such that $P_I\in L$ then $\varphi_t(L)$ is a hyperplane
such that $P_I\in L$ then $\varphi_t(L)$ is a hyperplane for every $t$ and $P_I\in \varphi_t(L)$ for every $t$. It is
clear that the only hyperplane $L$ containing $C_I$ and $P_I$ is $(x_3=0)$.
\end{proof}

\begin{rem}
In the example \ref{exam:genll} we have seen that there exist quadratic flows in $\gen{/\!/}\cup \lin$ such that neither the
hyperplanes containing $C_I$ nor the hyperplanes containing $P_I$ are fixed one to one.
\end{rem}

\newpage

\begin{thm}\label{thm:NFgen}
Let $\varphi_t$ be a generic quadratic flow. Then up to a linear conjugation $\varphi_t$ is of the following types:
\begin{enumerate}[\bf a)]
\item $\varphi_t[x_0,x_1,x_2,x_3]=[x_0x_3, (x_0 t+x_3)\cdot {\ell_1}_t, (x_0 t+x_3)\cdot {\ell_2}_t,
(x_0 t +x_3)\cdot x_3]$, or
 \item  $\varphi_t[x_0,x_1,x_2,x_3]=[x_0x_3+a(t)x_3^2, (a(t)x_0+x_3)\cdot {\ell_1}_t, (a(t)x_0+x_3)\cdot {\ell_2}_t,
(a(t)x_0+x_3)\cdot x_3]$, 
\end{enumerate}
where $a(t)=\frac{e^{\alpha t}-1}{e^{\alpha t}+1}$ and $\alpha\in \C^{\ast}$ and $\Psi_t[x_1,x_2,x_3]=[{\ell_1}_t,{\ell_2}_t,x_3]$ 
is one of the following linear flows in $\pp^2$:
\begin{enumerate}[\bf i)]
\item $\Psi_t[x_1,x_2,x_3]=[x_1,x_2+tx_3,x_3]$,
\item $\Psi_t[x_1,x_2,x_3]=[x_1 e^{\beta t} ,x_2+t x_3,x_3]$,
\item $\Psi_t[x_1,x_2,x_3]=[x_1+x_2(e^{\beta t}-1)+t x_3,x_2 e^{\beta t},x_3]$,
\item $\Psi_t[x_1,x_2,x_3]=[x_1 e^{\gamma_1 t}, x_2 e^{\gamma_2 t},x_3]$,
\item $\Psi_t[x_1,x_2,x_3]=[(x_1+t x_2) e^{\beta t},x_2 e^{\beta t},x_3]$,
\item $\Psi_t[x_1,x_2,x_3]=[x_1+t x_2 + \frac{t^2}{2} x_3 ,x_2+t x_3,x_3]$,
\end{enumerate}
with $\beta\in \C^{\ast}, \gamma_1,\gamma_2 \in \C$. Moreover, given a pair of flows in the previous list they are linearly conjugated 
if and only if they are both of type \textbf{a) iv)} or \textbf{b) iv)} with $\gamma_1=\gamma_2$ and the conjugation switches $x_1$ and $x_2$.  
\end{thm}

\begin{rem}
Up to a normalisation ot the time coordinate $t$ by an homotecy one can assume that $\alpha$ is 1 in case \textbf{a)} and that either
$\beta=1$ or $\gamma_1=1$ or $(\gamma_1,\gamma_2)=(0,1)$ in case \textbf{b)}. 
\end{rem}

\begin{proof}
By theorem $\ref{thm:gen}$ we know that $\varphi\in \gen{/\!/}$ with $P=P_I,C_I$ and $S$ fix.
We can therefore assume that $P_I=[1,0,0,0]$, $C_I=(x_0=x_3=0)$, $S=(x_3^2=0)$ and $H_t=(\ell_t(x_0,x_3)=a(t) x_0+
b(t) x_3=0)$, where $a(t), b(t)$ are holomorphic on $t$. Therefore, there exist $\alpha_t\in \C^{\ast}$, ${\ell_0}_t\in A_1(x_0,x_1,x_2,x_3)$ and ${\ell_1}_t,{\ell_2}_t,{\ell_3}_t\in A_1(x_1,x_2,x_3)$ such that 
$$\varphi_t[x_0,x_1,x_2,x_3]=[\alpha_t x_3^2+\ell_t \cdot {\ell_0}_t, \ell_t \cdot {\ell_1}_t,\ell_t \cdot {\ell_2}_t,
\ell_t \cdot {\ell_3}_t].$$
We must impose that 
\begin{align} \varphi_s(\varphi_t)(x)&= [\alpha_s\cdot \ell_t^2 \cdot {\ell_3}_t^2+\ell_s(\varphi_t)\cdot {\ell_0}_s(\varphi_t),
\{\ell_s(\varphi_t) \cdot \ell_t \cdot {\ell_i}_s({\ell_1}_t,{\ell_2}_t,{\ell_3}_t)\}_{i=1,2,3}]\label{tmess} \\
&= [\alpha_{t+s} x_3^2+\ell_{t+s} \cdot {\ell_0}_{t+s},\{ {\ell_i}_{t+s}\cdot \ell_{t+s} \}_{i=1,2,3}]. \nonumber
\end{align}
Note that 
\begin{equation}\label{lsfit}
\ell_s(\varphi_t)=a(s) \alpha_t x_3^2+ \ell_t (a(s) {\ell_0}_t+b(s) {\ell_3}_t). 
\end{equation}
As $\varphi_s\circ \varphi_t$ is quadratic the components of the expression \ref{tmess} must have a common
factor $\rho_{t,s}(x)$ of degree 2 (which might depend on $t$ and $s$). There are two possibilities for $\rho_{t,s}$
(up to a holomorphic function depending on $t$ and $s$ taking values in $\C^{\ast}$):
\begin{enumerate}[\bf (a)]
 \item $\rho_{t,s}(x)=\ell_s(\varphi_t)$,
 \item $\rho_{t,s}(x)$ is the product of $\ell_t$ and a factor $q_{t,s}(x)$ of $\ell_s(\varphi_t)$.
\end{enumerate}
In the first case we would have 
$$ \varphi_s(\varphi_t)(x)=\bigg[\alpha_s \frac{\ell_t^2\cdot{\ell_3}_t^2}{\ell_s(\varphi_t)},\big\{\ell_t \cdot {\ell_i}_s({\ell_1}_t,{\ell_2}_t,{\ell_3}_t)\big\}_{i=1,2,3} \bigg].$$
It follows that $\ell_{t+s}$ is a multiple of $\ell_t$ by an element of $\C^{\ast}$, which implies that $H_t$ is
fix and contradicts the hypothesis. Therefore, we can assume that $\rho_{t,s}(x)$ is a product of $\ell_t$ and a factor 
$q_{t,s}(x)$ of $\ell_s(\varphi_t)$ and
\begin{equation}\label{lmess2}
\varphi_s(\varphi_t)(x)=\bigg[\alpha_s \cdot \ell_t \cdot {\ell_3}_t^2+\frac{\ell_s(\varphi_t)\cdot{\ell_0}_s(\varphi_t)}{\ell_t},\big\{\ell_s(\varphi_t) \cdot {\ell_i}_s({\ell_1}_t,{\ell_2}_t,{\ell_3}_t)\big\}_{i=1,2,3} \bigg].
\end{equation}
As $q_{t,s}(x)$ must be a common factor of $\ell_s(\varphi_t)$ and $\ell_t\cdot {\ell_3}_t^2$ we conclude that
either $\ell_t$ or ${\ell_3}_t$ are factors of $\ell_s(\varphi_t)$. From \ref{lsfit} we derive that $\ell_t$ cannot be 
a factor of $\ell_s(\varphi_t)$ unless $a(t)=0$ (which contradicts $H_t$ mobile) or $\alpha_t=0$ (which would imply 
$\varphi_t$ linear). Thus $\ell_t$ is a factor of ${\ell_0}_s(\varphi_t)$ (see expression \ref{lmess2}) and ${\ell_3}_t$ is a factor 
of $\ell_s(\varphi_t)$. Expression \ref{lsfit} implies that if ${\ell_3}_t$ is a factor of $\ell_s(\varphi_t)$ then it is also 
a factor of $\alpha_t x_3^2+\ell_t \cdot {\ell_0}_t$. If 
$${\ell_0}_t=A_0(t) x_0+A_1(t) x_1+A_2(t) x_2+A_3(t) x_3$$
then 
$${\ell_0}_s(\varphi_t)=A_0(s) (\alpha_t x_3^2+\ell_t \cdot {\ell_0}_t)+A_1(s) \ell_t \cdot {\ell_1}_t+A_2(s) \ell_t \cdot {\ell_2}_t+A_3(s) \ell_t \cdot {\ell_3}_t.$$
From the fact that $\ell_t$ is a factor of ${\ell_0}_s(\varphi_t)$ we conclude that $A_0(t)=0$. We impose next that the quadric $\alpha_t x_3^2+\ell_t \cdot {\ell_0}_t$ on $x_1,x_2,x_3$ has rank 2 (for it admits ${\ell_3}_t$ as a factor). 
Imposing this condition and using its expression in terms of $a(t)$, $b(t)$, $\alpha_t$ and $A_i(t)$ one concludes that $A_1(t)=A_2(t)=0$. Thus ${\ell_0}_t=\beta_t x_3$ and 
\begin{equation}
\alpha_t x_3^2+\ell_t \cdot {\ell_0}_t=x_3 \cdot \big( (\alpha_t+\beta_t \cdot b(t))x_3+a(t) \cdot \beta_t x_0 \big) 
\end{equation}
As ${\ell_3}_t\in A_1(x_1,x_2,x_3)$ is a factor of $\alpha_t x_3^2+\ell_t \cdot {\ell_0}_t$ we have ${\ell_3}_t=\delta_t x_3$. Note that in particular in the first component of $\varphi_t$ appear only the monomials $x_0x_3$ and $x_3^2$. From now on we will use a more natural notation. Namely,
$$\varphi_t[x_0,x_1,x_2,x_3]=[A(t) x_0x_3+B(t)x_3^2, \ell_t \cdot {\ell_1}_t, \ell_t \cdot {\ell_2}_t, \ell_t \cdot {\delta}_t x_3]$$
with $A(t)$, $B(t)$ depending holomorphically on $t$. From $\varphi_0=Id$ we conclude that 
$$A(0)\neq 0, \quad B(0)=0, \quad a(0)=0, \quad b(0)\neq 0$$
and a normalization allows us to assume that $A(0)=b(0)=1$. Moreover, for small values of $t$ one has $A(t)\neq 0$, 
$b(t)\neq 0$. Dividing all the components of $\varphi_t$ by by $A(t)$ and modifying ${\ell_1}_t$, ${\ell_2}_t$ and
$\delta_t$ if necessary we can assume that $A(t)=b(t)=1$ for all $t$ so
\begin{equation}
\varphi_t[x_0,x_1,x_2,x_3]=[x_0x_3+B(t)x_3^2, (a(t)x_0+x_3){\ell_1}_t, (a(t)x_0+x_3){\ell_2}_t, (a(t)x_0+x_3)\delta_t x_3] 
\end{equation}
and $B(0)=a(0)=1$. Now, 
$$\ell_s(\varphi_t)=(a(s)+\delta_t a(t)) x_0x_3+(\delta_t+a(s)B(t))x_3^2$$
and
\begin{multline}
\varphi_s(\varphi_t)(x)=\bigg[(\delta_t+a(t)B(s)\delta_t^2)x_0x_3+(\delta_tB(t)+\delta_t^2B(s))x_3^2,\\ \frac{\ell_s(\varphi_t)}{x_3} {\ell_1}_s(\ell_{i_t}), \frac{\ell_s(\varphi_t)}{x_3} {\ell_2}_s(\ell_{i_t}), \frac{\ell_s(\varphi_t)}{x_3} \delta_t\delta_s x_3\bigg].
\end{multline}
As $\varphi_s(\varphi_t)(x)=\varphi_{t+s}(x)$ we have 
\begin{align*}
B(t+s)&= \frac{B(t)+\delta_t B(s)}{1+\delta_t a(t)B(s)}\\
a(t+s)&= \frac{a(s)+\delta_t a(t)}{\delta_t+a(s)B(t)}\\
\delta_{t+s}&= \delta_s \frac{\delta_t+a(s)B(t)}{1+a(t)B(s)\delta_t}
\end{align*}
and setting
\begin{align*}
f(t,s)&=\delta_t+ a(t)B(s)\delta_t^2\\
g(t,s)&=\delta_t+a(s)B(t)
\end{align*}
we have 
$$\varphi_s(\varphi_t)(x)=[f(t,s)(x_0x_3+B(t+s)x_3^2), g(t,s)\ell_{t+s}  {\ell_1}_s(\{\ell_i\}_t), g(t,s)\ell_{t+s}  {\ell_2}_s(\{\ell_i\}_t), g(t,s) \ell_{t+s}\delta_{t+s} x_3]. $$
In particular $\Psi_t[x_1,x_2,x_3]=[{\ell_1}_t(x_1,x_2,x_3), {\ell_2}_t(x_1,x_2,x_3), \delta_t x_3]$ is a linear flow. Using the 
classification of linear flows in $\pp^2$ (cf. \cite{CerDes}) one concludes that
$$\delta_t=e^{\delta t}$$
and that $\Psi_t$ belongs to the list given in the statement of the theorem. Therefore,
$$\frac{\ell_s(\varphi_t)}{x_3} \delta_t \delta_s x_3= g(t,s) \ell_{t+s} \delta_{t+s} x_3$$
and it follows that $f(t,s)=g(t,s)$. This equality is equivalent to 
$a(t)B(s)\delta_t^2=a(s)B(t)$. Thus, 
$$\frac{a(t)}{B(t)}\delta_t^2=\frac{a(s)}{B(s)},$$ which implies that $a(t)/B(t)$ is constant and that
$\delta_t=1$. Let $B(t)=c \cdot a(t)$ with $c\in \C$. We are left to solve the equation
\begin{equation}
a(t+s)=\frac{a(t)+a(s)}{1+c \cdot a(t) \cdot a(s)}.
\end{equation}
Assuming that $s$ is a constant and differentiating with respect to $t$ one obtains 
$$\frac{\partial}{\partial t} a(t,s)=a'(t) \frac{1-c \cdot a^2(s)}{1+c \cdot a(t) \cdot a(s)}.$$
Simmetry on $t$ and $s$ implies that
$$a'(t)\frac{1-c \cdot a^2(s)}{1+c \cdot a(t) \cdot a(s)}=a'(s) \frac{1-c \cdot a^2(t)}{1+c \cdot a(t) \cdot a(s)},$$
so 
$$\frac{a'(s)}{1- c \cdot a^2(s)}=\frac{a'(t)}{1- c \cdot a^2(t)}.$$
Again it follows that there is a constant $D\in \C$ such that
$$\frac{a'(t)}{1- c \cdot a^2(t)}=D.$$
If $c=0$ then as $a(0)=0$ one has $a(t)=Dt$. If $c\neq 0$ then one obtains 

$$a(t)=\tilde{c}^{-1} \frac{e^{\tilde{D}t}-1}{e^{\tilde{D}t}+1}, \qquad B(t)=\tilde{c} \frac{e^{\tilde{D}t}-1}{e^{\tilde{D}t}+1},$$
where $\tilde{c},\tilde{D}\in \C$ and $\tilde{c}^2=c$. 

Therefore either $a(t)= a t$ and $B(t)=0$ or $A(t)=a \frac{e^{\alpha t}-1}{e^{\alpha t}+1}.$ In both cases $a\neq 0$ for otherwise $\ell_t$ 
would be fix and we can assume $\alpha \neq 0$ for the same reason. 

Note that in case \textbf{a)} up to a good choice of $\rho_0^0$ we can assume that $a=1$. Moreover if $\varphi_t$ and $\widetilde{\varphi}_t$ 
were conjugated we would have (it suffices to write the condition $A\circ \widetilde{\varphi}_t = \varphi_t \circ A$ for the first coordinate):
$$ \rho_0^0 x_0 x_3 + \rho^0_3 (t x_0 + x_3)x_3= (\rho^0_0 x_0 + \rho^0_3 x_3) \rho_3^3 x_3$$
so $\rho^0_3=0$ and $\rho^3_3=1$. 

To end the proof we must just show that given two flows of the previous list with different choices of the parameters they cannot be conjugated 
unless both flows are in case \textbf{a) iv)} or \textbf{b) iv)} and the conjugation switches $x_1$ and $x_2$ or unless we are in case \textbf{b)} 
with $\alpha=\tilde{\alpha}$ and $A\neq \tilde{A}$ (which in particular shows that we can assume $a=1$). 

Let $\varphi_t$ and $\tilde{\varphi}_t$ be two flows of the previous list and $A=[\mu_0,\mu_1,\mu_2,\mu_3]$  with $\mu_i \in A_1(x_0,x_1,x_2,x_3)$ 
for $i=0,1,2,3$ a linear conjugation such that 

$$A \circ \tilde{\varphi}_t= \varphi_t \circ A.$$

As $P_I=[1,0,0,0]$, $C_I=(x_0=x_3=0)$ and $S=(x_3^2=0)$ are fix and common for both flows we conclude that $\mu_1(x), \mu_2(x) \in A_1 (x_1,x_2,x_3)$ 
and $\mu_3=\mu_3^3 x_3$. Moreover as $\ell_t=(a(t) x_0 + x_3=0)$ is conjugated to $\tilde{\ell}_t=(\tilde{a}(t) x_0  + x_3=0)$ one has 
$\mu_0(x)= \mu^0_0 x_0 + \mu^0_3 x_3$. In particular if we define $\tilde{A}(x_1,x_2,x_3):=[\mu_1,\mu_2,\mu_3]$ we obtain a conjugation of the 
correponding flows $\Psi_t$ and $\tilde{\Psi}_t$ in $\pp^2$, which implies that $\Psi_t=\tilde{\Psi}_t$ excepting in the case \textbf{iv)} with a 
switch of the coordinates $x_1$ and $x_2$.  

Therefore in case \textbf{b)} one has (again writing the condition $A\circ \tilde{\varphi}_t = \varphi_t \circ A$ for the first coordinate) 
$$\mu^0_0 (x_0x_3+\tilde{a}(t) x_3^2)+\mu^0_3(\tilde{a}(t)x_0+x_3)x_3=(\mu^0_0 x_0 + \mu^0_3 x_3) \mu^3_3 x_3+ a(t) (\mu^3_3)^2 x_3^2.$$
Equivalently, 
$$\mu_0^0 \mu_3^3 x_0x_3+ \big(\mu^0_3 \mu^3_3 + a(t) (\mu^3_3)^2 \big) x_3^2= (\mu^0_0+ \tilde{a}(t) \mu^0_3) x_0x_3 + (\tilde{a}(t) \mu^0_0 + \mu^0_3) x_3^2$$
which yields $\mu^0_3=0$ so
$$\mu^0_0 \mu^3_3 x_0 x_3+ a(t) (\mu^3_3)^2 x_3^2 = \mu^0_0 x_0 x_3 + \tilde{a}(t) \mu^0_0 x_3^2.$$ 
Using that $a(t)= a \frac{e^{\alpha t}-1}{e^{\alpha t}+1}$ and $\tilde{a}(t)=\tilde{a} \frac{e^{\tilde{\alpha} t}-1}{e^{\tilde{\alpha} t}+1}$ 
we conclude that $\mu^3_3=1$ and $\tilde{a}(t)= \mu^0_0 a(t)$ (with the exception already mentioned). 
\end{proof}

\begin{cor}\label{cor:gen}
Let $\varphi_t$ a quadratic generic flow. Then $\varphi_t$ is determined by a linear flow $\eta_t$ on the $\pp^2$ of the net of lines through $P_I$
and by a linear flow $\chi_t$ on the $\pp^1$ of the pencil of planes through the line $C_I$. Namely, for a point $P\in \pp^3$ we have
$$\varphi_t(P)=\eta_t(P \vee P_I) \cap \chi_t(P\vee C_I).$$
\end{cor}

\begin{proof}
With the coordinates of the previous theorem we have $P_I=[1,0,0,0]$ and $C_I=(x_0=x_3=0)$. If $L$ is a line through $P_I$ we define $\eta_t$ by
$$\eta_t(L)=P_I \vee \Psi_t (L\cap(x_0=0)). $$
As a plane $\pi$ through $C_I$ is given by an equation $\rho x_0+\sigma x_3=0$ we check easily that the strict transformation of $\pi$ by $\varphi_t$ is
$$\begin{array}{lll}
  (\rho-t \sigma) x_0 + \sigma x_3 =0 & \textrm{for} & \alpha=0 \\
  (\rho+a(-t) \sigma) x_0 + (a(-t) \rho + \sigma) x_3=0 & \textrm{for} & \alpha\in \C^{\ast}
\end{array}$$
with $a(t)=\frac{e^{\alpha t}-1}{e^{\alpha t}+1}$. Thus we define
$$\mu_t[\rho,\sigma]:=\left\{ \begin{array}{lll}
{[\rho-t\sigma,\sigma]} & \textrm{for} & \alpha=0 \\
{[\rho+a(-t)\sigma,a(-t)\rho+\sigma]} & \textrm{for} & \alpha \in \C^{\ast}
\end{array}\right. . $$
\end{proof}


\subsection{Flows in $\tang{O}\cup \tang{\times}\cup \lin$}

\begin{lem}\label{lem:tangox}
Let $\varphi_t$ be a quadratic flow in $\tang{\times}\cup \lin$ or $\tang{O}\cup \lin$. Then either $S_t$ or $H_t$ are mobile.
\end{lem}

\begin{proof}
Assume that both $S_t$ and $H_t$ are fix, then $L=H\cap S$ is a fix line and ${P_I}_t\in L$. Without loss of generality we
can choose coordinates such that $H=(x_3=0)$ and $S=(x_0=0)$, then ${P_I}_t=[0,1,f(t),0]$. Therefore
$$\varphi_t[x_0,x_1,x_2,x_3]=\phantom{aixo esta aqui per tirar equacio totalment a esquerra}$$
$$[x_3 \cdot {\ell_0}_t, a_t x_0 x_1 + r_t (x_0, x_3, f(t)x_1-x_2), b_t x_0 x_1 + s_t (x_0,x_3,f(t)x_1-x_2), x_3\cdot {\ell_3}_t]$$
where $r_t,s_t\in A_2 (x_0,x_3, f(t)x_1-x_2)$ (recall that $\varphi_t(H)={P_I}_{-t}=P_t$ and that the
tangent plane of $({\varphi_i}_t=0)$ at ${P_I}_t$ is $x_0=0$).
Then $\varphi_0=\mathrm{Id}$ implies that $x_3\cdot {\ell_0}_0=x_0 \cdot \ell$ with $\ell\in A_1 (x_0,x_1,x_2,x_3)$.
Therefore $\ell=\alpha \cdot x_3$ with $\alpha\in \C^{\ast}$. Now it is not difficult to reach a contradiction with
$$a_0 x_0x_1 + r_0 (x_0,x_3, f(t)x_1-x_2)= \alpha x_1 x_3.$$
\end{proof}

\begin{thm}\label{thm:tangox}
Let $\varphi_t$ be a quadratic flow in $\tang{O}\cup \lin$ or $\tang{\times}\cup \lin$. Then:
\begin{enumerate}[\bf i)]
\item $H_t$ and ${C_I}_t$ are mobile and $P_I$ and $S_I$ are fix.
\item ${L_I}_t=S \cap {H_t}$ is a fix line through $P_I$.
\item $\varphi_t$ preserves the family of hyperplanes through $L_I$, which includes  $\{H_s\}$.
\item $\varphi_t$ preserves the family of hyperplanes through $P_I$, which includes the family of
hyperplanes through $L_I$.
\end{enumerate}
\end{thm}

\begin{rem}
In this case as $P_I$ is contained in the line $L_I$ we cannot determine $\varphi_t$ using flows on the net of lines through $P_I$
and the pencil of planes through $L_I$ as in Corollary \ref{cor:gen}. The list in Theorem \ref{thm:NFtangllosc} shows that there are not other
choices of a pair point/line which allows to do so either.

\end{rem}

\begin{proof}
By corollary \ref{cor:tangox} and lemma \ref{lem:tangox} to show \textbf{i)} it is enough to prove that $H_t$ is mobile and $S$ fix. Note that if ${H_t}$ mobile and $S_I$ fix it is clear that ${C_I}_t$ must be mobile. Assume that $H$ is fix and ${S_I}_t$ mobile. Then $C_I$ is fix and
$${S_I}_{t+s}=\varphi_{t+s} (C_I)= \varphi_t \circ \varphi_s (C_I) = \varphi_t ({S_I}_s).$$
As ${S_I}_t$ are hyperplanes we conclude that either ${P_I}_t\in {S_I}_s$ for every $t,s$ or
$C_I \subset {S_I}_t$ for every $t$.

As $C_I$ is a fix plane conic of rank greater or equal than 2 the
second possibility contradicts $S_t$ mobile unless $C_I=L_1 \cup L_2$, $L_1, L_2$ lines and $L_1 \subset {S_I}_t$
for every $t$.

If ${P_I}_t\in {S_I}_s$ for every $t,s$ and rank $C_I=3$ it is not difficult to conclude that $P_I$ must be fix, which
contradicts $H_I$ fix.

If ${P_I}_t\in {S_I}_s$ for every $t,s$ and rank $C_I=2$, i.e. $C_I=L_1 \cup L_2$, we conclude that ${P_I}_t\in L_1$
and $L_1 \subset S_t$ for every $t$.

Assume then that $C_I=L_1 \cup L_2$, $L_1, L_2$ lines and $L_1 \subset {S_I}_t$ for every $t$. Without loss of generality
we can choose coordinates such that
$$H=(x_3=0), \quad L_1=(x_1=x_3=0), \quad L_2=(x_2=x_3=0).$$
Then ${P_I}_t=[f(t),0,1,0]$ and $S_t=(a(t) x_1+b(t)x_3=0)$, where $a(t), b(t)$ are germs of holomorphic functions on $t$. If $\varphi_t=({\varphi_0}_t,{\varphi_1}_t,{\varphi_2}_t,{\varphi_3}_t)$ using $\ind(\varphi_t)=L_1\cup L_2$ and that $({\varphi_i}_t=0)$
are quadrics tangent to $S_t$ at the point ${P_I}_t$ we obtain
\begin{multline*}
\varphi_t[x_0,x_1,x_2,x_3]=\\
[{\alpha_0}_t \cdot x_2 (a(t)x_1+b(t) x_3)+ x_3 \cdot {\ell_0}_t, x_3 \cdot {\ell_1}_t,
{\alpha_2}_t \cdot x_2 (a(t)x_1+b(t) x_3)+ x_3 \cdot {\ell_2}_t, x_3 \cdot {\ell_3}_t]
\end{multline*}
with ${\alpha_i}_t\in \C$ and ${\ell_i}_t \in A_1 (x_1,x_3, x_0-f(t)x_2)$. 

The condition $\varphi_0=\mathrm{Id}$ implies that (dividing by an element of $\C^{\ast}$ if
necessary)
$$\varphi_0=[x_0x_3,x_1x_3,x_2x_3,x_3^2]$$
so
\begin{eqnarray*}
x_0x_3&=&{\alpha_0}_0 x_2 (a(0)x_1+b(0) x_3)+x_3 {\ell_0}_0 \\
x_1 &=& {\ell_1}_0 \\
x_2x_3 &=& {\alpha_2}_0 x_2 (a(0)x_1+b(0) x_3)+ x_3 {\ell_2}_0 \\
x_3&=& {\ell_3}_0
\end{eqnarray*}
and ${\ell_i}_0\in A_1(x_0-f(0)x_2)$ for $i=1,2,3,4$. One concludes that
$${\ell_0}_0=\beta_0 (x_0-f(0) x_2), \, a(0)=0, \, {\alpha_0}_0b(0)-\beta_0f(0)=0,\, {\ell_2}_0=0,\, {\ell_1}_0=x_1, \, {\ell_3}_0=x_3, \, {\alpha_0}_0 b(0)=1.$$
We can therefore assume that $\ell_t=a(t)x_1+x_3$ and the previous conditions for $t=0$ translate into
$${\alpha_0}_0=f(0), \, {\ell_0}_0=x_0-f(0)x_2, {\ell_1}_0=x_1,\, {\alpha_2}_0=1, \,{\ell_2}_0=1, \, {\ell_3}_0=x_3.$$
On the other hand $\varphi_t(H)=\varphi_t(x_3=0)=P_t={P_I}_{-t}=[f(-t),0,1,0]$, which implies that
$${\alpha_0}_t=f(-t) {\alpha_2}_t.$$
Normalizing if necessary we have
\begin{equation}
\varphi_t(x)=[(f(-t)x_2\ell_t+x_3 {\ell_0}_t, x_3 {\ell_1}_t, x_2 \ell_t + x_3 {\ell_2}_t, x_3 {\ell_3}_t]
\end{equation}
and
\begin{equation}
{\ell_0}_0=x_0-x_2, \, {\ell_1}_0=x_1, \, {\ell_2}_0=0, \, {\ell_3}_0=x_3, \, f(0)=1.
\end{equation}
Therefore
\begin{multline}
\varphi_s(\varphi_t(x))=[f(-s) (x_2 \ell_t + x_3 {\ell_2}_t) \ell_s(\varphi_t)+x_3 {\ell_3}_t {\ell_0}_s(\varphi_t),\\
x_3 {\ell_3}_t {\ell_1}_s(\varphi_t), (x_2 \ell_t+x_3 {\ell_2}_t) \ell_s(\varphi_t)+x_3 {\ell_3}_t {\ell_2}_s(\varphi_t),
x_3 {\ell_3}_t {\ell_3}_s (\varphi_t)] 
\end{multline}
and
\begin{equation}
\ell_s\circ\varphi_t(x)=x_3(a(s) {\ell_1}_t + {\ell_3}_t).
\end{equation}
Dividing by $x_3$ we obtain
\begin{multline}
\varphi_s\circ\varphi_t(x)=\bigg[f(-s) (x_2 \ell_t+ x_3 {\ell_2}_t) \frac{\ell_s(\varphi_t)}{x_3}+{\ell_3}_t {\ell_0}_s(\varphi_t),\\ {\ell_3}_t {\ell_1}_s(\varphi_t), (x_2 \ell_t+x_3 {\ell_2}_t) \frac{\ell_s(\varphi_t)}{x_3}+{\ell_3}_t {\ell_2}_s(\varphi_t), {\ell_3}_t {\ell_3}_s(\varphi_t)].
\end{multline}
As $\varphi_s\circ\varphi_t=\varphi_{t+s}$, which is a quadratic flow, we must have a linear form dividing all the
components of $\varphi_s(\varphi_t)$. This form can be either \textbf{(1)} ${\ell_3}_t$ or \textbf{(2)} another linear
form $\tilde{\ell}_t$. 

Let us study case \textbf{(1)}. We see that ${\ell_3}_t$ divides either $\frac{\ell_s(\varphi_t)}{x_3}=a(s){\ell_1}_t+{\ell_3}_t$ or $x_2 \ell_t+x_3 {\ell_2}_t$. We can clearly exclude
the first possibility. On the other hand
$$x_2 \ell_t + x_3 {\ell_2}_t= x_2 (a(t) x_1+x_3)+x_3 {\ell_2}_t (x_1,x_3,x_0-f(t)x_2)$$
and taking into account that ${\ell_3}_t\in A_1(x_1,x_3, x_0-f(t)x_2)$ one can check that ${\ell_3}_t$ cannot
divide this term either. 

Assume now that we are in case \textbf{(2)}, i.e. there exists a form $\tilde{\ell}_t$ which is not a multiple of
${\ell_3}_t$ which divides all the components of $\varphi_s\circ\varphi_t$. In particular $\tilde{\ell}_t$ must divide 
${\ell_1}_s(\varphi_t)$, ${\ell_3}_s(\varphi_t)$ and $x_3 ({\ell_0}_s(\varphi_t)-f(-s){\ell_2}_s(\varphi_t))$. Note that if $\tilde{\ell}_t=x_3$ then ${\ell_1}_t,{\ell_3}_t\in A_1(x_1,x_3)$. But in this case $\frac{\ell_s(\varphi_t)}{x_3}\in A_1(x_1,x_3)$ and we easily reach a contradiction. Therefore $\tilde{\ell}_t$ divides also ${\ell_0}_s(\varphi_t)-f(-s){\ell_2}_s(\varphi_t)$. 

If ${\ell_1}_t=A^0_1(t) (x_0-f(t)x_2)+A^1_1(t) x_1+A^3_1(t) x_3$ then 

$${\ell_1}_s(\varphi_t)=A_1^0(s) \big((f(-t)-f(s))x_2 \ell_t+x_3 ({\ell_0}_t-f(s) {\ell_2}_t)\big) + A^1_1(s) x_3 {\ell_1}_t+A^3_1(s) x_3 {\ell_3}_t.$$

A direct computation shows that for ${\ell_1}_s(\varphi_t)$ to have rank 2 either
\begin{equation}
A^0_1(s) (f(-t)-f(s)) a(t)=0 \quad \textrm{and} \quad \tilde{\ell}_t=x_3
\end{equation}
or
\begin{equation}\label{eq:ici}
A^0_1(s) \big( (f(-t)-f(s))x_2+{\ell_0}_t -f(s) {\ell_2}_t \big)+ A^1_1(s) {\ell_1}_t + A^3_1(s) {\ell_3}_t=0 \quad \textrm{and} \quad \tilde{\ell}_t=x_1 \, \textrm{or} \, x_2.
\end{equation}
An analogous argument applies to ${\ell_3}_s(\varphi_t)$ and ${\ell_0}_s(\varphi_t)-f(-s){\ell_2}_s(\varphi_t)$. 
As we have already discussed the possibility $\tilde{\ell}_t=x_3$ we focus on \ref{eq:ici}. The linear forms
${\ell_1}_t$, ${\ell_3}_t$ and ${\ell_0}_t-f(-t){\ell_2}_t$ are linearly independent because
$$\ind(\varphi_t)=\{ {\ell_1}_t={\ell_3}_t={\ell_0}_t-f(-t){\ell_2}_t=x_2 \ell_t+x_3 {\ell_2}_t=0\}$$
and as $\varphi_t$ is not generic $\ind(\varphi_t)=L_1 \cup L_2$ for every $t\neq 0$. Now, on \ref{eq:ici} we take 
$s=-t$. Then
$$A^0_1(-t)({\ell_0}_t-f(-t){\ell_2}_t)+A^1_1(-t) {\ell_1}_t+ A^3_1(-t) {\ell_3}_t=0$$
and as ${\ell_1}_t$, ${\ell_3}_t$ and ${\ell_0}_t-f(-t){\ell_2}_t$ are linearly independent we conclude that
$$A^0_1(-t)=A^1_1(-t)=A^3_1(-t)=0$$
so ${\ell_1}_t=0$, which is a contradiction. We have thus proved that $H_t$ is mobile and that $P$ and $S$ are fix.

Therefore we can assume that $H_t$ and ${C_I}_t$ are mobile and $S=S_I$, $P=P_I$ are fix. Let us now show that the line ${L_I}_t:=S_I\cap H_t$ is fix. As ${L_I}_t$ is a line by the point $P_I$ contained in $H_t$ we have $\varphi_t({L_I}_t)=P_I$. Thus,
$$H_I=\varphi_t(P_I)=\varphi_s(\varphi_t({L_I}_s))=\varphi_t(\varphi_s({L_I}_t)).$$
It follows that $\varphi_t({L_I}_s)=P_I$, which implies that ${L_I}_s\subset H_t$ for any $t,s$. As
$H_t$ is mobile we have $L_I$ fix. 

We are left to prove statements \textbf{iii)} and \textbf{iv)}. We can assume that $P=[0,1,0,0]=P_I$, $L_I=(x_0=x_3=0)$, $S=(x_3=0)$ and $H_t=(\ell_t=a(t) x_0 + b(t) x_3=0)$, where $a(t)$ and $b(t)$ are germs of holomorphic functions. Then
$$\varphi_t[x_0,x_1,x_2,x_3]=[{\ell_0}_t\cdot \ell_t, \alpha_t x_1x_3+q_t(x_0,x_2,x_3), {\ell_2}_t\cdot \ell_t, {\ell_3}_t\cdot \ell_t]$$
with ${\ell_i}_t\in A_1(x_0,x_2,x_3)$ (recall that all the elements of the linear system $\Gamma_{\varphi_t}$ are quadrics tangent to the plane $S$ at the point $P$). In particular the statement in \textbf{iv)} is proved. 

As $\varphi_0=\mathrm{Id}$ we conclude (dividing by an element of $\C^{\ast}$ if necessary) that 
$$\alpha_0=1, \, q_0=0,\, {\ell_0}_0=x_0,\, {\ell_2}_0=x_2, \, {\ell_3}_0=x_3, \, {\ell_1}_0=x_3, \, a(0)=0, \, b(0)=1.$$
We can therefore assume that $b(t)=1$. Moreover,

\begin{multline}
\varphi_s(\varphi_t(x))=[\ell_t\cdot {\ell_0}_s ({\ell_i}_t) \ell_s ({\ell_i}_t), \\
\alpha_s \alpha_t x_1 x_3 {\ell_3}_t + \alpha_s q_t (x_0,x_2,x_3) {\ell_3}_t + \ell_t q_s ({\ell_i}_t), 
\ell_t {\ell_2}_s({\ell_i}_t)\ell_s({\ell_i}_t), \ell_t {\ell_3}_s({\ell_i}_t) \ell_s({\ell_i}_t)]
\end{multline}
where $\ell_t=a(t)x_0+x_3$ and 
$$\ell_s({\ell_i}_t)=a(t) {\ell_0}_t + {\ell_3}_t.$$

The components of $\varphi_s\circ\varphi_t$ must admit a common linear factor for $\varphi_{t+s}$ is quadratic and the factor must be either $\ell_t$ or $\ell_s({\ell_i}_t)$ (up to an element of $\C^{\ast}$). On the other hand from the coefficient of $x_1x_3$ one concludes that the factor must be ${\ell_3}_t$. Therefore, the possibilities are:
\begin{enumerate}[\bf (I)]
\item ${\ell_3}_t=\beta_t \cdot \ell_t$,
\item ${\ell_3}_t=\beta_t \cdot \ell_s({\ell_i}_t)$.
\end{enumerate}
The second one would imply that $H_t$ is fix so we can assume that \textbf{(I)} holds. Therefore,
$$\varphi_s(\varphi_t(x))=[{\ell_0}_s({\ell_i}_t)\ell_s({\ell_i}_t), \alpha_s\alpha_t\beta_t x_1x_3+\beta_t\alpha_s q_t(x_0,x_2,x_3)+q_s({\ell_i}_t),{\ell_2}_s({\ell_i}_t)\ell_s({\ell_i}_t), \beta_t \ell_s^2({\ell_i}_t) ].$$ 
In particular,
\begin{eqnarray}
\ell_s({\ell_i}_t)&=&f(t,s) \cot \ell_{t+s} \label{eq:ferst}\\
\alpha_t\cdot \alpha_s \cdot \beta_t &=& f^2(t,s) \cdot \alpha_{t+s}
\end{eqnarray}
where $f(t,s)$ is holomorphic on $t$ and $s$. 
From \ref{eq:ferst} one concludes that ${\ell_0}_t\in A_1(x_1,x_3)$, which proves \textbf{iii)}. 
\end{proof}

\begin{thm}\label{thm:NFtangox1}
Let $\varphi_t$ be a quadratic flow in $\tang{O}\cup\tang{\times}\cup \lin$. Then, up to linear conjugation $\varphi_t$ is of one of the following types:
\begin{enumerate}[\bf I)]
\item $\varphi_t[x_0,x_1,x_2,x_3]=[x_0(x_3+tx_0),e^{\alpha t} x_1x_3+q_t(x_0,x_2,x_3), e^{\beta t} x_2 (x_3+tx_0), (x_3+tx_0)^2]$ where:
\begin{itemize}
\item[\bf a)] $\alpha=\beta=0$,
$$q_t(x_0,x_2,x_3)=\frac{c_{33}}{3}t^3 x_0^2 + c_{22} t x_2^2 + c_{33} t x_3^2 + \frac{c_{23}}{2} t^2 x_0x_2 + c_{33} t^2 x_0x_3 + c_{23} t x_2x_3$$
and either $c_{22},c_{23},c_{33}\in \{0,1\}$ or $c_{22}=c_{33}=1$ and $c_{23}\in\C^{\ast}$. 
\item[\bf b)] $\alpha=0, \, \beta \neq 0$, 
$$q_t(x_0,x_2,x_3)=\frac{c_{33}}{3}t^3 x_0^2+c_{22} (1-e^{2\beta t})x_2^2+c_{33} t x_3^2+ (1-e^{\beta t}) c_{02} x_0x_2 + c_{33} t^2 x_0 x_3$$
and $c_{22}, c_{33}, c_{02}\in \C$.
\item[\bf c)] $\alpha \neq 0$, $\beta=0$,
$$q_t(x_0,x_2,x_3)=(e^{\alpha t}-1)(c_{00} x_0^2 + c_{22} x_2^2 + c_{02} x_0 x_2)$$
and $(c_{00},c_{22},c_{02})=(0,0,0), \, (1,0,0),\,(0,0,1)$ or $(c_{00},1,0)$ with $c_{00}\in \C^{\ast}$.
\item[\bf d)] $\alpha \neq 0$, $\beta\neq 0$, 
$$q_t(x_0,x_2,x_3)=c_{00} (e^{\alpha t}-1)x_0^2+ c_{22} (e^{\alpha t}-e^{2\beta t}) x_2^2+c_{02} (e^{\alpha t}-e^{\beta t}) x_0x_2$$
and $c_{00}, c_{22}, c_{02}\in \C$.
\end{itemize}
\item $\varphi_t[x_0,x_1,x_2,x_3]=[x_0(x_3+tx_0),e^{\alpha t} x_1x_3+q_t(x_0,x_2,x_3), (x_2+tx_3+\frac{t^2}{2}x_0) (x_3+tx_0), (x_3+tx_0)^2]$ where 
\begin{itemize}
\item[\bf e)] $\alpha =0$,
\begin{multline*}
q_t(x_0,x_2,x_3)= t (c_{22} x_2+c_{02} x_0 +c_{23} x_3) x_2+ \bigg( \frac{c_{02}}{6} t^3 +\frac{c_{23}}{8} t^4 +\frac{c_{22}}{20} t^5 \bigg) x_0^2\\+
\bigg(\frac{c_{23}}{2} t^2 + \frac{c_{22}}{3}t^3 \bigg)(x_3^2+x_0x_2)+ \bigg( \frac{c_{02}}{2} t^2 + \frac{c_{23}}{2} t^3 + \frac{c_{22}}{4} t^4\bigg) x_0x_3 
+ c_{22} t^2 x_2 x_3
\end{multline*}
and $(c_{02}, c_{22}, c_{23})=(0,1,c_{23}), \, (c_{02},0,0)$ or $(c_{02},0,1)$ with $c_{23},c_{02} \in \C$.
\item[\bf f)] $\alpha \neq 0$,
\begin{multline*}
q_t(x_0,x_2,x_3)=(e^{\alpha t}-1) (c_{00} x_0^2+c_{22} x_2^2 + c_{02} x_0x_2) -\bigg(\frac{c_{22}}{4} t^4 +\frac{c_{02}}{2} t^2 \bigg) x_0^2\\-
c_22 t^2 (x_3^2+x_0x_2) - 2 c_22 t x_2 x_3 - (c_{02} t + c_{22} t^3) x_0x_3 
\end{multline*}
and $(c_{00},c_{22},c_{02})=(0,0,1),\,(1,0,0)$ or $(c_{00},0,1)$ with $c_{00}\in \C$.
\end{itemize} 
\end{enumerate}
Moreover two such different flows $\varphi_t$ and $\widetilde{\varphi}_t$ are linearly conjugated if and only if we are 
in cases \textbf{b)} or \textbf{d)} and there exists $\lambda \in \C^{\ast}$ such that $\tilde{q}_t=\lambda \cdot q_t$,
\end{thm}

\begin{proof}
We will use the same notation as in theorem \ref{thm:tangox}. Recall that we had chosen coordinates such that
$P=P_I=[0,1,0,0]$, $L_I=(x_0=x_3=0)$, $S=(x_3=0)$ and $H_t=(a(t)x_0+x_3=0)$. Moreover, we had concluded that
$$\varphi_t[x_0,x_1,x_2,x_3]=[{\ell_0}_t \ell_t, \alpha_t x_1 x_3+q_t(x_0,x_2,x_3), {\ell_2}_t \ell_t, \beta_t \ell_t^2]$$
with ${\ell_0}_t\in A_1(x_0,x_3)$, ${\ell_2}_t\in A_1(x_0,x_2,x_3)$ such that 
$$\alpha_0=1, \, q_0=0,\, {\ell_0}_0=x_0, \, {\ell_2}_0=x_2, \, {\ell_3}_0=x_3, \, {\ell_1}_0=x_3,\, a(0)=0$$
and 
\begin{eqnarray}
\ell_s({\ell_i}_t)&=&f(t,s) \cdot \ell_{t+s}, \label{eq:first}\\
\alpha_t\cdot \alpha_s \cdot \beta_t &=& f^2(t,s) \cdot \alpha_{t+s}. \label{eq:second}
\end{eqnarray}
Moreover,
$$\varphi_s(\varphi_t(x))=[{\ell_0}_s({\ell_i}_t)\ell_s({\ell_i}_t), \alpha_s\alpha_t\beta_t x_1x_3+\beta_t\alpha_s q_t(x_0,x_2,x_3)+q_s({\ell_i}_t),{\ell_2}_s({\ell_i}_t) \ell_s({\ell_i}_t), \beta_t \ell_s^2({\ell_i}_t)]$$ 
and $\ell_t$ is a common factor of all the components of $\varphi_s \circ \varphi_t$. 

We denote now ${\ell_0}_t=A_0^0(t)x_0+A^3_0(t)x_3$. We can rewrite \ref{eq:first} as
$$(a(s)A^0_0(t)+a(t))x_0+(a(s)A^3_0(t)+1)x_3=f(t,s)(a(t+s)x_0+x_3)$$
so
\begin{eqnarray*}
a(s) A^0_0(t)+a(t) &=& f(t,s) a(t+s)\\
a(s) A^3_0(t)+1 &=& f(t,s).
\end{eqnarray*}
As $a(0)=0$ we have $f(t,0)=1$. On the other hand, as $\alpha_0=1$ one has
$$\alpha_t \beta_t = f^2(t,0) \alpha_t$$
so $\beta_t=1$. We can define $\psi_t[x_0,x_2,x_3]=[{\ell_0}_t,{\ell_2}_t,\ell_t]$, which is a linear flow. Moreover, 
$\ell_t=a(t)x_0+x_3$ and ${\ell_0}_t\in A_1(x_0,x_3)$. From the classification of linear flows on $\pp^2$ we derive that
either
$$a(t)=t, \, {\ell_0}_s=x_0, \, {\ell_2}_s=x_2 e^{\beta t}$$
or 
$$a(t)=t, \, {\ell_0}_s=x_0, \, {\ell_2}_s=x_2+tx_3+\frac{t^2}{2} x_0,$$
which in both cases yields $f(t,s)=1$ and $\alpha_t=e^{\alpha t}$ with $\alpha, \beta\in \C$. Thus, either
$$\varphi_t[x_0,x_1,x_2,x_3]=[x_0(x_3+tx_0),e^{\alpha t} x_1x_3+q_t(x_0,x_2,x_3), e^{\beta t} x_2 (x_3+tx_0), (x_3+tx_0)^2]$$
and
\begin{multline}
\varphi_s(\varphi_t)(x)=[x_0(x_3+(t+s)x_0), e^{\alpha(t+s)}x_1x_3+e^{\alpha s} q_t(x)+q_t({\ell_i}_t), \\
e^{\beta(t+s)}x_2(x_3+(t+s)x_0), (x_3+(t+s)x_0)^2]
\end{multline}
or 
$$\varphi_t[x_0,x_1,x_2,x_3]=\bigg[x_0(x_3+tx_0),e^{\alpha t} x_1x_3+q_t(x_0,x_2,x_3), \bigg(x_2+tx_3+\frac{t^2}{2}x_0\bigg) (x_3+tx_0), (x_3+tx_0)^2\bigg]$$
and 
\begin{multline}
\varphi_s(\varphi_t)(x)=\bigg[x_0(x_3+(t+s)x_0), e^{\alpha(t+s)}x_1x_3+e^{\alpha s} q_t(x)+q_t({\ell_i}_t), \\
\bigg(x_2+(t+s)x_3+\frac{(t+s)^2}{2}x_0\bigg)(x_3+(t+s)x_0), (x_3+(t+s)x_0)^2\bigg].
\end{multline}
We must impose now that either 
\begin{equation}\label{eq:A}
e^{\alpha s} q_t (x_0,x_2,x_3)+q_s(x_0, e^{\beta t}x_2, x_3+tx_0)=q_{t+s}(x_0,x_2,x_3)
\end{equation}
or
\begin{equation}\label{eq:B}
e^{\alpha s} q_t (x_0,x_2,x_3)+q_s(x_0, x_2+t x_3+\frac{t^2}{2} x_0, x_3+tx_0)=q_{t+s}(x_0,x_2,x_3).
\end{equation}
We will denote
$$q_t(x_0,x_2,x_3)=q_{00}(t) x_0^2+ q_{02}(t) x_0x_2+ q_{03}(t) x_0x_3+ q_{22}(t) x_2^2+ q_{23}(t) x_2x_3+q_{33}(t) x_3^2.$$
From \ref{eq:A} we obtain
\begin{eqnarray*}
q_{00}(t+s)&=&e^{\alpha s} q_{00}(t)+q_{00}(s)+t^2 q_{33}(s) + t q_{03}(s),\\
q_{02}(t+s)&=&e^{\alpha s} q_{02}(t)+e^{\beta t} (q_{02}(s)+tq_{23}(s)),\\
q_{03}(t+s)&=&e^{\alpha s} q_{03}(t) + q_{03}(s)+2t q_{33}(s),\\
q_{22}(t+s)&=&e^{\alpha s} q_{22}(t)+e^{2\beta t}q_{22}(s),\\
q_{23}(t+s)&=&e^{\alpha s} q_{23}(t) + e^{\beta t} q_{23}(s),\\
q_{33}(t+s)&=&e^{\alpha s} q_{33}(t)+q_{33}(s).
\end{eqnarray*}
Using the symmetry on $t,s$ on the previous equations one concludes that $q_t[x_0,x_2,x_3]$ verifies \ref{eq:A} if and only if it is of one forms:

\begin{itemize}
\item[\textbf{a)}] $(c_{00}t+\frac{c_{03}}{2} t^2+ \frac{c_{33}}{3} t^3) x_0^2 +c_{22} t x_2^2+ c_{33} t x_3^2+(c_{02}t+\frac{c_{23}}{2} t^2) x_0x_2+(c_{03} t+ c_{33} t^2) x_0x_3  + c_{23} t x_2x_3 ,$ for $\alpha=\beta=0$,
\item[\textbf{b)}] $(c_{00}t+\frac{c_{03}}{2} t^2+ \frac{c_{33}}{3} t^3) x_0^2 + c_{22} (1-e^{2 \beta t}) x_2^2 + c_{33} t x_3^2+((1-e^{\beta t}) c_{02}-c_{23}te^{\beta t}) x_0x_2+(c_{03} t + c_{33} t^2) x_0x_3  + c_{23} (1-e^{\beta t}) x_2x_3$, 
for $\alpha=0, \, \beta \neq 0$, or
\item[\textbf{c)}] $\big( c_{00} (e^{\alpha t} -1) - c_{03} t -c_{33} t^2 \big)x_0^2+ c_{22} (e^{\alpha t}- 1) x_2^2+ c_{33} (e^{\alpha t}-1) x_3^2 + \big( c_{02} (e^{\alpha t}-1) -c_{23} t \big)x_0x_2 + \big( c_{03}  (e^{\alpha t}-1) - 2 c_{33} t \big) x_0 x_3   + c_{23} (e^{\alpha t}-1)x_2 x_3$ for $\alpha \neq 0$ and $\beta=0$, 
\item[\textbf{d)}] $\big( c_{00} (e^{\alpha t} -1) - c_{03} t -c_{33} t^2 \big)x_0^2+ c_{22} (e^{\alpha t}- e^{2 \beta t}) x_2^2+ c_{33} (e^{\alpha t}-1) x_3^2 + \big( c_{02} (e^{\alpha t}-e^{\beta t}) -c_{23} t e^{\beta t} \big)x_0x_2 + \big( c_{03}  (e^{\alpha t}-1) - 2 c_{33} t \big) x_0 x_3   + c_{23} (e^{\alpha t}-e^{\beta t})x_2 x_3$  for $\alpha \neq 0$, $\beta \neq 0$, 
\end{itemize}
and $c_{ij}$ are constant. Note for instance that the last equation yields
$$e^{\alpha s} q_{33}(t)+q_{33}(s)=e^{\alpha t} q_{33}(s) + q_{33}(t)$$
which implies that $\frac{q_{33}(t)}{(e^{\alpha t}-1)}$ is constant for $\alpha \neq 0$.

A similar computation allows to conclude that $q_t$ verifies equation \ref{eq:B} if and only if $q_t(x_0,x_2,x_3)$ is 
\begin{itemize}
\item[\textbf{e)}] 
$ t q_{0} + \big(\frac{c_{03}}{2} t^2+ (c_{33}+\frac{c_{02}}{2}) \frac{t^3}{3} + \frac{c_{23}}{8} t^4 + \frac{c_{22}}{20} t^5\big)x_0^2 +  \big(\frac{c_{23}}{2} t^2 + \frac{c_{22}}{3} t^3 \big) x_3^2 + \big(\frac{c_{23}}{2} t^2 + \frac{c_{22}}{3} t^3  \big)x_0x_2 + \big((c_ {33}+\frac{c_{02}}{2})t^2 + \frac{c_{23}}{2} t^3+ \frac{c_{22}}{4} t^4 \big)x_0x_3 +c_{22} t^2 x_2 x_3$, for $\alpha=0$  
\item[\textbf{f)}]
$(e^{\alpha t}-1) q_0 - \big(\frac{c_{22}}{4}t^4 + \frac{c_{23}}{2} t^3 + (c_{33}+\frac{c_{02}}{2}) t^2 + c_{03} t \big) x_0^2- (tc_{23}+t^2 c_{22})(x_3^2+x_0x_2) - 2t c_{22} x_2x_3 - \big(tc_{02}+2tc_{33}+\frac{3 c_{23}}{2} t^2 + c_{22} t^3\big) x_0x_3 $, for $\alpha \neq 0$.
\end{itemize} 
where
$$q_0(x_0,x_2,x_3)=c_{00}x_0^2+c_{22}x_2^2+ c_{33}x_3^2+c_{02} x_0x_2+ c_{03} x_0x_3+c_{23} x_2x_3$$
and $c_{ij}$ are constant.

Let now $\varphi_t$ and $\tilde{\varphi}_t$ be two flows of one of the previous types: \textbf{a)},$\ldots$, \textbf{f)},  and $A=[\mu_0,\mu_1,\mu_2,\mu_3]$  with $\mu_i \in A_1(x_0,x_1,x_2,x_3)$ for $i=0,1,2,3$ a linear conjugation such that 

$$A \circ \varphi_t= \tilde{\varphi}_t \circ A. \qquad \qquad (\ast)$$

Using that $P_I=[0,1,0,0]$, $L_I=(x_0=x_3=0)$ and $S=(x_3^2=0)$ are fix and common for both flows, $H_t=(tx_0+x_3=0)$ and imposing $(\ast)$ on the last component we conclude that
$$\mu_0=x_0;\quad \mu_1=\mu_1^0 x_0+\mu_1^1 x_1 + \mu_1^2 x_2 + \mu_1^3 x_3; \quad \mu_2=\mu^0_2 x_0+\mu^2_2 x_2 + \mu^3_2 x_3;\quad \mu_3=x_3,$$
up to product by an element of $\C^{\ast}$.

If $\varphi_t$ and $\tilde{\varphi}_t$ are both of type \textbf{I)} then imposing now $(\ast)$ on the third component of the flows one obtains 

$$\mu^0_2 e^{\tilde{\beta}t} x_0 + \mu^2_2 e^{\tilde{\beta}t} x_2 + \mu^3_2 e^{\tilde{\beta}t} x_3 =\mu^0_2 x_0 +\mu_2^2 e^{\beta t} x_2 +\mu_2^3 (x_3+tx_0).$$

As $\mu^2_2\neq 0$ we conclude that $\beta=\tilde{\beta}$ and $\mu^3_2=0$. Moreover either $\beta$ or $\mu^0_2$ are zero and if $\mu^0_2=0$ then $\mu^2_2=1$. 

If both flows are of type \textbf{II)} analogously one obtains $\mu^3_2=0$ and $\mu_2^2=1$. 

Finally, it also implies that one flow of type \textbf{I)} and one flow of type \textbf{II)} cannot be conjugated. 

Imposing condition $(\ast)$ on the second component for two linearly conjugated flows of either type \textbf{I)} or two flows of type \textbf{II)} one concludes easily that in both cases $\alpha=\tilde{\alpha}$.
Moreover, writing explicitely this condition for the six cases \textbf{a)} to \textbf{f)} one obtains the list of the statement. Let us assume for instance that $\alpha=\beta=0$, one obtains:

\begin{multline*}
(\mu^0_1 x_0 + \mu^1_1 x_1 + \mu_1^2 x_2 + \mu_1^3 x_3) x_3 + 
(\tilde{c}_{00} t + \frac{\tilde{c}_{03}}{2} t^2 + \frac{\tilde{c}_{33}}{3} t^3) x_0^2+ (\tilde{c}_{02}t+ \frac{\tilde{c}_{23}}{2} t^2) x_0 (\mu^0_2 x_0 + \mu_2^2 x_2) + \\ \tilde{c}_{22} t (\mu_2^0 x_0 + \mu_2^2 x_2)^2+
\tilde{c}_{23} t (\mu^0_2 x_0 + \mu^2_2 x_2) x_3 + \tilde{c}_{33} t x_3^2+(\tilde{c}_{03} t + \tilde{c}_{33} t^2) x_0x_3=\\
\mu^0_1 x_0 (x_3+tx_0) + \mu_1^1 \big(x_1x_3+ (c_{00} t + \frac{ c_{03}}{2} t^2 + \frac{c_{33}}{3}t^3)x_0^2+
(c_{02} t + \frac{c_{23}}{2} t^2) x_0x_2 + c_{22} t x_2^2 + c_{23} t x_2x_3+ c_{33} t x_3^2 + \\ (c_{03} t + c_{33} t^2)x_0x_3\big) + \mu_1^2 x_2 (x_3+tx_0)+\mu_1^3(x_3+tx_0)^2,
\end{multline*}
which implies
\begin{eqnarray*}
\tilde{c}_{00} &=& \mu_1^1 c_{00}-\mu^0_2 \mu^1_1 (\mu_2^2)^{-1} c_{02} +\mu^0_2 (\mu^2_2)^{-2} \mu^1_1 c_{22} -\mu^0_2\mu^2_1 (\mu^2_2)^{-1}+\mu^0_1\\
\tilde{c}_{22} &=& (\mu_2^2)^{-2} \mu_1^1 c_{22} \\
\tilde{c}_{33} &=& \mu_1^1 c_{33} \\
\tilde{c}_{02} &=&  (\mu_2^2)^{-1} \mu_1^1 c_{02}-2 \mu^0_2 (\mu_2^2)^{-2} \mu_1^1 c_{22}+\mu_1^2 (\mu_2^2)^{-1}\\
\tilde{c}_{03} &=& \mu_1^1 c_{03} -\mu_2^0 (\mu_2^2)^{-1} \mu_1^1 c_{23}+2\mu_1^3\\
\tilde{c}_{23} &=& (\mu_2^2)^{-1} \mu_1^1 c_{23}.
\end{eqnarray*}
It is not difficult to see that up to a linear conjugation one can choose the parameters $c_{ij}$ as
in the statement. And analogous computation gives the rest of the flows.
\end{proof}

\begin{rem}
Computing $\ind(\varphi_t)$ it is not difficult to verify that the flows in \ref{thm:NFtangox1} belong indeed to $\tang{O}\cup\tang{\times}\cup\lin$  
\end{rem}

\begin{rem}
As in the generic case, up to a normalisation ot the time coordinate $t$ by an homotecy one can assume that $\alpha$ is either 0 or 1. 
\end{rem}


\subsection{Flows in $\tang{/\!/}\cup \osc\cup \lin$}

\begin{defn}
Let $\varphi_t$ be a quadratic flow. We say that $\varphi_t$ is a \emph{polynomial flow} if there is an invariant chart
$\C^3$ such that ${\varphi_t}_{|\C^3}:\C^3\rightarrow \C^3$ is polynomial for each $t$.
\end{defn}

The following result was proved by D. Cerveau and J. Deserti for quadratic flows in $\pp^2$. The proof
for $\pp^3$ is almost identical, we include it for the sake of clarity.

\begin{prop}
Let $\varphi_t$ be a quadratic flow such that $H_t$ is fix. Then $\varphi_t$ is a polynomial flow.
\end{prop}

\begin{proof}
We can assume that $\varphi_t$ is not linear and that $H=(x_3=0)$. Therefore for every $t$ such that $\varphi_t$ is 
quadratic we can write the flow as 
$$A_t (x_0x_3, x_1x_3-x_2^2, x_2x_3, x_3^2) B_t$$
or 
$$A_t (x_0x_3-x_1x_2, x_1x_2,x_2x_3,x_3^2) B_t$$
where $A_t,B_t\in \mathrm{PGL}_4(\C)$ and $B_t$ preserves the hyperplan $x_3=0$. As $\ind(\varphi_t)$ is contained in 
$H$, i.e. at the infinity of the affine chart $x_3=1$ for almost every $t$ and thus for every $t$, we have that $\varphi_t$ is polynomial.
\end{proof}

\begin{lem}\label{lem:previthmtanosc}
Let $\varphi_t$ a quadratic flow in $\osc \cup \lin$. Then $P_I$ is fix.
\end{lem}

\begin{proof}
Assume that $\varphi_t\in \tang{/\!/}\cup \osc \cup \lin$ and that ${P_I}_t$ is mobile. We will show that 
$\varphi_t\in \tang{/\!/}\cup \lin$. By lemma \ref{lem:PmobHfix} we have that $H_t$ is fix. Without loss of generality 
we can assume that $H=(x_3=0)$ and $P_t={P_I}_{-t}=[\alpha_t,\beta_t,\gamma_t,0]$ (because ${P_I}_t\in H$). Therefore 
$$\varphi_t(x)=[\alpha_t q_t + {\ell_0}_t x_3, \beta_t q_t + {\ell_1}_t x_3, \gamma_t q_t+{\ell_2}_t x_3, {\ell_3}_t x_3]$$
with ${\ell_i}_t\in A_1(x_0,x_1,x_2,x_3)$ for $i=0,1,2,3$, $q_t\in A_2(x_0,x_1,x_2)$, $\mathrm{rank}\, q_t \leq 2$ and 
$\jac(\varphi_t)=f(t) x_3^4$ with $f(t)\in \C^{\ast}$. Moreover, by the previous lemma $\varphi_t$ is a polynomial flow. 
In the affine coordinates $x=\frac{x_0}{x_3}$, $y=\frac{x_1}{x_3}$ and $z=\frac{x_2}{x_3}$, the flow $\varphi_t$ has the expression:
$$\widetilde{\varphi}_t(x)=\bigg( \frac{\alpha_t\frac{q_t}{x_3^2}+\frac{{\ell_0}_t}{x_3}}{\frac{{\ell_3}_t}{x_3}},
\frac{\beta_t \frac{q_t}{x_3^2}+\frac{{\ell_1}_t}{x_3}}{\frac{{\ell_3}_t}{x_3}},\frac{\gamma_t \frac{q_t}{x_3^2}+\frac{{\ell_2}_t}{x_3}}{\frac{{\ell_3}_t}{x_3}} \bigg).$$
As $\widetilde{\varphi}_t$ is a polynomial we must have ${\ell_3}_t=\delta_t x_3$ so (dividing by $\delta_t$ if 
necessary) we have ${\ell_3}_t=x_3$. Thus, in affine coordinates
$$\widetilde{\varphi}_t=(\alpha_t \tilde{q}_t +{\tilde\ell}_{0_t}, \beta_t \tilde{q}_t +{\tilde\ell}_{1_t}, 
\gamma_t \tilde{q}_t +{\tilde\ell}_{2_t}),$$
where ${\tilde\ell}_{i_t}=A^i_t x + B^i_t y + C^i_t z + D^i_t$ with $A^i_t,B^i_t,C^i_t,D^i_t\in \C$ for $i=0,1,2$ and $\tilde{q}_t$ 
is a homogeneous polynomial of degree 2 in $x,y,z$. On the other hand $\widetilde{\varphi}_s \circ \widetilde{\varphi}_t$ must still 
be a quadratic flow. Let $A_{\widetilde{q}_t}$ be the bilinear form associated to $\widetilde{q}_t$, then denoting 
$\widetilde{L}_t=({\tilde\ell}_{0_t},{\tilde\ell}_{1_t},{\tilde\ell}_{2_t})$ one has
$$\tilde{q}_s\big(\widetilde{\varphi}_t(x,y,z)\big)=(\tilde{q}_t P_t + \widetilde{L}_t)\cdot A_{\tilde{q}_s} \cdot
(\tilde{q}_t P_t + \widetilde{L}_t)^t= \tilde{q}_t^2 (P_t \cdot A_{\tilde{q}_s} \cdot P_t^T)+ 2 \tilde{q}_t (P_t \cdot A_{\tilde{q}_s} \cdot \widetilde{L}_t^T)+\widetilde{L}_t \cdot A_{\tilde{q}_s} \cdot \widetilde{L}^T_t.$$
As the two first terms have degrees 4 and 3 respectively they must be zero. We conclude that
\begin{eqnarray}
P_t \cdot A_{\tilde{q}_s} \cdot P_t^T &=& \tilde{q}_s(P_t)=0   \label{eqn:prima}\\
P_t \cdot A_{\tilde{q}_s} \cdot \tilde{L}_t^T&=&0 \label{eqn:seconda}
\end{eqnarray}
As $P_t$ is mobile condition \ref{eqn:prima} implies that $\tilde{q}_s$ has a fixed component. We can assume that $q_t= x_1 \cdot m_t$ 
where $m_t\in A_1 (x_0,x_1,x_2)$ and $\beta_t=0$, i.e. $P_t=(\alpha_t,0,\gamma_t)$. Thus
$$\tilde{\varphi}_t(x,y,z)=\big(\alpha_t \cdot y \cdot \widetilde{m}_t + {\tilde\ell}_{0_t}, {\tilde\ell}_{1_t}, 
\gamma_t \cdot y \cdot \widetilde{m}_t + {\tilde\ell}_{2_t}\big),$$
where $\widetilde{m}_t$ is the homogeneous polynomial of degree 1 on $x,y,z$ obtained from $m_t$. Now, condition \ref{eqn:seconda} 
says that either ${\tilde\ell}_{1_t}=0$ or $\frac{1}{2} \widetilde{m}_s(P_t)=P_t \cdot A_{\tilde{q}_s}=0$. In the first case 
$\widetilde{\varphi}_t$ wouldn't be an automorphism, which contradicts the hypothesis. In the second case, as $P_t$ is mobile, 
$\widetilde{m}_s$ must be fix and in fact a multiple of $y$. Thus, one can assume 
$$\widetilde{\varphi}_t(x,y,z)=(\alpha_t y^2+ {\tilde\ell}_{0_t}, {\tilde\ell}_{1_t}, \gamma_t y^2+ {\tilde\ell}_{2_t})$$ and 
$\varphi_t\in \tang{/\!/}\cup \lin$, which ends the proof.
\end{proof}

\begin{lem}\label{lem:previthmosctanbis}
Let $\varphi_t$ be a quadratic flow in $\tang{/\!/}\cup \osc\cup\lin$. Then $H_t$ is fix.
\end{lem}

\begin{proof}
As ${P_I}_t$ mobile implies that $H$ is fix it is enough to consider the case $P_I$ fix. Assume that $H_t=(\ell_t=0)$ is mobile 
and $P_I=[1,0,0,0]$. Therefore
$$\varphi_t(x)=[{\ell_0}_t \ell_t+q_t, {\ell_1}_t \ell_t, {\ell_2}_t \ell_t, {\ell_3}_t \ell_t],$$
with ${\ell_0}_t, {\ell_1}_t, {\ell_2}_t, {\ell_3}_t \in A_1(x_0,x_1,x_2,x_3)$. Note also that $P_I\in (\ell_t=0)$ and that the 
strict transform of a line through $P_I$ is a conic through $P_I$, so in fact ${\ell_0}_t, {\ell_1}_t, {\ell_2}_t, {\ell_3}_t \in A_1(x_1,x_2,x_3)$. 
Moreover $P_I\in \ind(\varphi_t)$ implies that $q_t$ has no term in $x_0^2$. Note that for $\varphi_t$ to be birational we must have 
${\ell_1}_t, {\ell_2}_t, {\ell_3}_t$ linearly independent and note also that we can assume that ${\ell_0}_t=\delta_t x_0$ (modifying $q_t$ if necessary). Then
$$\varphi_t[x_0,x_1,x_2,x_3]=[\delta_t x_0+\frac{q_t(x_1,x_2,x_3)}{\ell_t(x_1,x_2,x_3)}, {\ell_1}_t, {\ell_2}_t, {\ell_3}_t].$$
As $\varphi_0=Id$ we have $q_0=0$, $\delta_0=1$ and ${\ell_i}_0=x_i$ for $i=1,2,3$. We impose now the condition 
$\varphi_s \circ \varphi_t=\varphi_{t+s}$. We have 
$$\varphi_s \circ \varphi_t (x)= \bigg[\delta_s \delta_t x_0+\delta_s \frac{q_t(x)}{\ell_t(x)}+
\frac{q_s({\ell_i}_t)}{{\ell_s}({\ell_i}_t)},{\ell_1}_s({\ell_i}_t),{\ell_2}_s({\ell_i}_t),{\ell_3}_s({\ell_i}_t)\bigg].$$
In particular $\Psi_t[x_1,x_2,x_3]=[{\ell_1}_t,{\ell_2}_t,{\ell_3}_t]$ is a linear flow in $\pp^2$. 
Thus we can assume that $\delta_t \cdot \delta_s=\delta_{t+s}$ and 
$$\delta_s\frac{q_t(x_1,x_2,x_3)}{\ell_t(x_1,x_2,x_3)}+\frac{q_s({\ell_i}_t)}{{\ell_s}({\ell_i}_t)}=\frac{q_{t+s}(x)}{\ell_{t+s}(x)}$$
so
\begin{equation}
(\delta_s \cdot  q_t \cdot \ell_s({\ell_i}_t)+q_s({\ell_i}_t)\cdot \ell_t)\cdot \ell_{t+s} = q_{t+s} \cdot \ell_s({\ell_i}_t) \cdot \ell_t \label{eqn:terza}.
\end{equation}
As $H_t$ is mobile $\ell_t$ cannot divide $\ell_{t+s}$ so it divides the first factor of the left part of equation \ref{eqn:terza}, 
which implies that $\ell_t$ divides the product $q_t \cdot \ell_s({\ell_i}_t)$. If $\ell_t$ divides $q_t$ for all $t$ the flow is linear, 
which is a contradiction so we conclude that $\ell_t$ is a non-zero scalar multiple of $\ell_s({\ell_i}_t)$. A similar argument shows that 
$\ell_{t+s}$ is a non-zero scalar multiple of $\ell_s({\ell_i}_t)$. Again this would imply that $H_t$ is fix. 
\end{proof}

\begin{thm}\label{thm:NFtangllosc}
Let $\varphi_t$ be a quadratic flow in $\tang{/\!/}\cup \osc \cup \lin$. Then:
\begin{enumerate}[\bf i)]
\item $H$ is fix.
\item $\varphi_t$ is a polynomial flow.
\item If $\varphi_t\in \tang{/\!/}\cup \lin$ then $C_I$ is fix.
\item If $\varphi_t\in \osc\cup \lin$ then $P_I$ is fix. 
\item If $P_I$ is fix then $\varphi_t$ preserves the family of hyperplanes through $P_I$.
\item If $C_I$ is fix then $\varphi_t$ preserves the family of hyperplanes through a line $L$ contained in $C_I$.
\end{enumerate}
Moreover, up to linear conjugation, $\varphi_t$ belongs to the following list:
\begin{enumerate}[\bf I)]
 \item If ${P_I}_t$ is mobile and $\varphi_t\in \tang{/\!/}\cup \lin$.
\begin{enumerate}[\bf a)]
\item $\varphi_t(x)=[a (e^{2 \mu t}-e^{\alpha t}) x_1^2+ e^{\alpha t} x_0 x_3, e^{\mu t} x_1x_3, 
b(e^{2\mu t}-e^{\beta t}) x_1^2+e^{\beta t} x_2x_3, x_3^2]$,
\item $\varphi_t(x)=[a(e^{2 \mu t}-e^{\mu t}) x_1^2+ e^{\mu t} (x_0+tx_1)x_3, e^{\mu t} x_1x_3, 
b(e^{2\mu t}-e^{\beta t})x_1^2+e^{\beta t} x_2x_3, x_3^2]$,
\item $\varphi_t(x)=[a(e^{2 \mu t}-e^{\alpha t}) x_1^2+e^{\alpha t} x_0x_3, e^{\mu t} x_1x_3, 
b(e^{2 \mu t}-1)x_1^2+x_2x_3+tx_3^2, x_3^2]$,
\item $\varphi_t(x)=[a(e^{2\mu t}-e^{\mu t})x_1^2+e^{\mu t}(x_0+t x_1)x_3, e^{\mu t}x_1x_3,
b(e^{2 \mu t}-1) x_1^2+x_2x_3+tx_3^2, x_3^2]$.
\end{enumerate}
with $a,b\neq 0$.
\item If $P_I$ is fix then:
\begin{enumerate}[\bf 1)]
\item $\varphi_t(x)=[e^{\delta t} x_0x_3+ 2tB e^{\alpha t}x_1x_3+
2tC x_2x_3+ t^2 C x_3^2+ A(e^{2 \alpha t}-e^{\delta t})x_1^2+2B(e^{\alpha t}-e^{\delta t})x_1x_2+C(1-e^{\delta t})x_2^2, 
e^{\alpha t}x_1x_3, x_2x_3+tx_3^2, x_3^2]$ with $\delta \neq 0$, $\delta \neq \alpha$, $\delta\neq 2 \alpha$, \label{flow:ix3first}
\item $\varphi_t(x)=[x_0x_3+2t B e^{\alpha t}x_1x_3+
(bt+C t^2)x_2x_3+(\frac{b}{2}t^2+\frac{C}{3}t^3)x_3^2+A(e^{2 \alpha t}-1)x_1^2+2B(e^{\alpha t}-1)x_1x_2, 
e^{\alpha t} x_1x_3, x_2x_3+tx_3^2, x_3^2]$ with $\alpha \neq 0$, \label{flow:ix3second},
\item $\varphi_t(x)=[x_0x_3+(a t+B t^2)x_1x_3+(b t +C t^2)x_2x_3+(\frac{b}{2}t^2+\frac{C}{3}t^3)x_3^2+At x_1^2+2Bt x_1x_2+C t x_2^2, 
x_1x_3, x_2x_3+tx_3^2, x_3^2]$,\label{flow:ix3third}
\item $\varphi_t(x)=[e^{\delta t} x_0x_3+(a t+B t^2) e^{\delta t}x_1x_3+2 C tx_2x_3+C t^2 x_3^2+
A ( e^{2\delta t}-e^{\delta t}) x_1^2+2B t e^{\delta t} x_1x_2+C(1-e^{\delta t})x_2^2, e^{\delta t}x_1x_3,
x_2x_3+tx_3^2, x_3^2]$ with $\delta \neq 0$,\label{flow:ix3fourth}
\item $\varphi_t(x)=[e^{\delta t} x_0x_3+2tB e^{\frac{\delta t}{2}} x_1x_3+2 C t x_2x_3+ C t^2 x_3^2+
A t e^{\delta t} x_1^2+2B(e^{\frac{\delta t}{2}}-e^{\delta t})x_1x_2+C(1-e^{\delta t})x_2^2, \\ e^{\frac{\delta t}{2}}x_1x_3,
x_2x_3+tx_3^2, x_3^2]$ with $\delta \neq 0$, \label{flow:ix3fifth}
\item  $\varphi_t(x)=[e^{\delta t}x_0x_1+ A(e^{-\alpha t}-e^{\delta t})x_2^2+2(B(e^{-\alpha t}-e^{\delta t})+A t e^{-\alpha t}) x_2x_3+ 
\big( C(e^{-\alpha t}-e^{\delta t})+(At^2+2 t B)e^{-\alpha t} \big)x_3^2, e^{\alpha t} x_1^2, (x_2+tx_3)x_1,x_3x_1]$ with $\delta \neq 0$, \label{eq:ambx1}\label{flow:ix1first}
\item  $\varphi_t(x)=[e^{\delta t}x_0x_1+ A t e^{\delta t} x_2^2+(2 B t + A t^2)e^{\delta t} x_2x_3+ 
(C t + B t^2 + \frac{A}{3}t^3)e^{\delta t} x_3^2, e^{-\delta t} x_1^2, (x_2+tx_3)x_1,x_3x_1]$ with $\delta \neq 0$, \label{flow:ix1second}
\item  $\varphi_t(x)=[e^{\delta t}x_0x_1+ a t e^{\delta t}x_1^2+A(e^{-\delta t}-e^{\delta t})x_2^2+2(B(e^{-\delta t}-e^{\delta t})+
A t e^{-\delta t}) x_2x_3+ \big( C(e^{-\delta t}-e^{\delta t})+(At^2+2 t B)e^{-\delta t} \big)x_3^2, e^{\delta t} x_1^2, (x_2+tx_3)x_1,x_3x_1]$ 
with $\delta \neq 0$, \label{flow:ix1third}
\item  $\varphi_t(x)=[x_0x_1+ b t x_1x_2+\frac{b t^2}{2}x_1x_3+A(e^{-\alpha t }-1)x_2^2+2(B(e^{-\alpha t}-1)+A t e^{-\alpha t}) x_2x_3
+\big( C(e^{-\alpha t}-1)+(At^2+2 t B)e^{-\alpha t} \big)x_3^2, e^{\alpha t} x_1^2, (x_2+tx_3)x_1,x_3x_1]$ with $\alpha \neq 0$,
\label{eq:ambx1bis}\label{flow:ix1fourth}
\item  $\varphi_t(x)=[x_0x_1+ a t x_1^2+b t x_1x_2+\frac{b t^2}{2}x_1x_3+A t x_2^2+2\big(B t+\frac{A}{2}t^2\big) x_2x_3+
\big(C t+ B t^2 + \frac{A}{3} t^3\big)x_3^2, x_1^2, (x_2+tx_3)x_1,x_3x_1]$,
\label{flow:ix1fifth}
\item $\varphi_t(x)=[e^{\delta t} x_0x_3+ A (e^{2 \alpha t}-e^{\delta t})x_1^2+2 B (e^{(\alpha + \beta)t}-e^{\delta t}) x_1x_2+
C(e^{2 \beta t}-e^{\delta t})x_2^2, e^{\alpha t}x_1x_3, e^{\beta t} x_2x_3, x_3^2]$ with $\delta \neq 0$, 
$\delta \neq \alpha, \beta,\alpha + \beta, 2 \alpha, 2 \beta$, \label{flow:iix3first}
\item $\varphi_t(x)=[x_0x_3+ c t x_3^2+A (e^{2 \alpha t}-1)x_1^2+2 B (e^{(\alpha + \beta)t}-1) x_1x_2+C(e^{2 \beta t}-1)x_2^2, 
e^{\alpha t}x_1x_3, e^{\beta t} x_2x_3, x_3^2]$ with $\alpha \neq 0$ and $\beta \neq 0$, \label{flow:iix3second}\label{eq:exam14}
\item $\varphi_t(x)=[e^{\delta t} x_0x_3+ a t e^{\delta t} x_1x_3+A (e^{2 \delta t}-e^{\delta t})x_1^2+2 B (e^{(\delta + \beta)t}-e^{\delta t}) x_1x_2+
C(e^{2 \beta t}-e^{\delta t})x_2^2, e^{\delta t}x_1x_3, \\ e^{\beta t} x_2x_3, x_3^2]$ with $\delta \neq 0$, $\delta \neq \beta, 2 \beta$,
\label{flow:iix3third}
\item $\varphi_t(x)=[e^{2\alpha t} x_0x_3+ A t e^{2\alpha t} x_1^2+2 B (e^{(\alpha + \beta)t}-e^{2\alpha t}) x_1x_2+C(e^{2 \beta t}-e^{2\alpha t})x_2^2, 
e^{\alpha t}x_1x_3, e^{\beta t} x_2x_3, x_3^2]$ with $\alpha \neq 0$, $\beta \neq 0$, $\alpha \neq \beta$,\label{flow:iix3fourth}
\item $\varphi_t(x)=[e^{\delta t} x_0x_3+ a t e^{\delta t} x_1x_3+b t e^{\delta t}x_2x_3+A (e^{2 \delta t}-e^{\delta t})x_1^2+
2 B (e^{2\delta t}-e^{\delta t}) x_1x_2+C(e^{2 \delta t}-e^{\delta t})x_2^2, e^{\delta t}x_1x_3, e^{\delta t} x_2x_3, x_3^2]$ 
with $\delta \neq 0$,  \label{flow:iix3fifth}
\item $\varphi_t(x)=[e^{2\alpha t} (x_0x_3+ A t x_1^2+2 B t x_1x_2+C tx_2^2), e^{\alpha t}x_1x_3, e^{\alpha t} x_2x_3, x_3^2]$ 
with $\alpha \neq 0$, \label{flow:iix3sixth}
\item $\varphi_t(x)=[e^{(\alpha + \beta) t} x_0x_3+ A (e^{2 \alpha t}-e^{(\alpha + \beta) t})x_1^2+2 B t e^{(\alpha + \beta)t} x_1x_2+
C(e^{2 \beta t}-e^{(\alpha+\beta)t})x_2^2, e^{\alpha t}x_1x_3, e^{\beta t} x_2x_3, \\ x_3^2]$ with $\alpha\neq 0$, $\beta \neq 0$, $\alpha\neq \beta$, 
\label{flow:iix3seventh}
\item $\varphi_t(x)=[e^{\delta t} x_0x_3+ a t e^{\delta t} x_1x_3+A (e^{2 \delta t}-e^{\delta t})x_1^2+2 B t e^{\delta t} x_1x_2+C(1-e^{\delta t})x_2^2, 
e^{\delta t}x_1x_3, x_2x_3, x_3^2]$ with $\delta \neq 0$, \label{flow:iix3eight}
\item $\varphi_t(x)=[x_0x_3+ a t x_1x_3+c t x_3^2+A t x_1^2+2 B (e^{\beta t}-1) x_1x_2+C(e^{2 \beta t}-1)x_2^2, x_1x_3, e^{\beta t} x_2x_3, x_3^2]$
 with $\beta \neq 0$,\label{flow:iix3tenth}
\item $\varphi_t(x)=[x_0x_3+t (c x_3^2+\epsilon_1 x_1^2+\epsilon_2 x_2^2), x_1x_3,x_2x_3,x_3^2]$, with $\epsilon_i=0$ or 1 for $i=1,2$,\label{flow:iix3eleventh}
\item $\varphi_t(x)=[e^{\delta t}x_0x_3+2 A t x_1x_3+
\big(2 B te^{\alpha t}-2 A t \big) x_2x_3+ At^2 x_3^2+
A(1-e^{\delta t})x_1^2+2\big(B(e^{\alpha t}-e^{\delta t})-A(1-e^{\delta t})\big)x_1x_2+ 
\big(C(e^{2 \alpha t}-e^{\delta t})-2B(e^{\alpha t}-e^{\delta t})+A(1-e^{\delta t})\big)x_2^2, 
x_1x_3+(e^{\alpha t}-1)x_2x_3+tx_3^2, e^{\alpha t}x_2x_3,x_3^2]$ with $\delta \neq 0, \alpha,2 \alpha, \alpha\neq 0$,
\label{flow:iiix3first}
\item $\varphi_t(x)=[e^{\delta t}x_0x_3+2 A t x_1x_3+
\big( (B t^2+bt)e^{\delta t}-2 A t\big) x_2x_3+ At^2 x_3^2+A(1-e^{\delta t})x_1^2+2\big(B t e^{\delta t}-A(1-e^{\delta t})\big)x_1x_2+ 
\big(C(e^{2 \delta t}-e^{\delta t})-2B t e^{\delta t}+A(1-e^{\delta t})\big)x_2^2, 
x_1x_3+(e^{\delta t}-1)x_2x_3+tx_3^2, e^{\delta t}x_2x_3,x_3^2]$ with $\delta \neq 0$, \label{flow:iiix3second}
\item $\varphi_t(x)=[e^{2\alpha}x_0x_3+2 t A x_1x_3+
\big(2te^{\alpha t}B-2tA \big) x_2x_3+ At^2 x_3^2+
A(1-e^{2\alpha t})x_1^2+2\big(B(e^{\alpha t}-e^{2\alpha t})-A(1-e^{2\alpha t})\big)x_1x_2+ 
\big(C t e^{2 \alpha t}-2B(e^{\alpha t}-e^{2\alpha t})+A(1-e^{2\alpha t})\big)x_2^2, 
x_1x_3+(e^{\alpha t}-1)x_2x_3+tx_3^2, e^{\alpha t}x_2x_3,x_3^2]$ with $\alpha \neq 0$, \label{flow:iiix3third}
\item $\varphi_t(x)=[x_0x_3+(at+At^2)x_1x_3+\big(2 t e^{\alpha t} B- A t^2-at)x_2x_3+
\big(\frac{A}{3} t^3+\frac{a}{2} t^2\big)x_3^2+A t x_1^2+ 2 \big(B(e^{\alpha t}-1)-At \big)x_1x_2+
\big(C(e^{2\alpha t}-1)-2B(e^{\alpha t}-1)+At\big)x_2^2, x_1x_3+(e^{\alpha t}-1)x_2x_3+tx_3^2, 
e^{\alpha t} x_2x_3, x_3^2]$ with $\alpha \neq 0$, \label{flow:iiix3fourth}
\item $\varphi_t(x)=[e^{\delta t}x_0x_2-2A(e^{-\alpha t}-e^{\delta t}) x_1x_2+
A(e^{-\alpha t}-e^{\delta t})x_2^2+
\big(-2 A t e^{-\alpha t}+ 2 B (1-e^{-\alpha t}) \big)x_2x_3+
A(e^{-\alpha t}-e^{\delta t})x_1^2+2 \big(B(e^{-\alpha t}-e^{\delta t})+At e^{-\alpha t}\big)x_1x_3+
\big(C(e^{-\alpha t}-e^{\delta t})+At^2e^{-\alpha t}+2Bt e^{-\alpha t} \big)x_3^2, (x_1+(e^{\alpha t}-1)x_2+tx_3)x_2,
e^{\alpha t} x_2^2, x_2x_3]$ with $\delta \neq 0$, $\alpha\neq 0$, \label{flow:iiix2first}
\item $\varphi_t(x)=[x_0x_2+\big(a t-2A(e^{-\alpha t}-1) \big)x_1x_2+
\big(-a t+A(e^{-\alpha t}-1) \big)x_2^2+
\big(\frac{a}{2}t^2-2 A t e^{-\alpha t}+ 2(1-e^{-\alpha t})B \big)x_2x_3+
A(e^{-\alpha t}-e^{\delta t})x_1^2+2 \big(B(e^{-\alpha t}-1)+At e^{-\alpha t}\big)x_1x_3+
\big(C(e^{-\alpha t}-1)+At^2e^{-\alpha t}+2Bt e^{-\alpha t} \big)x_3^2, (x_1+(e^{\alpha t}-1)x_2+tx_3)x_2,
e^{\alpha t} x_2^2, x_2x_3]$ with $\alpha \neq 0$,\label{flow:iiix2second}
\item $\varphi_t(x)=[e^{\delta t}x_0x_2-2A t e^{\delta t}x_1x_2+
A t e^{\delta t} x_2^2-(2Bt+At^2) e^{\delta t} x_2x_3+
A t e^{\delta t} x_1^2+ \big(2B t + A t^2\big) e^{\delta t} x_1x_3+(C t + B t^2 + \frac{A}{3} t^3)e^{\delta t}
x_3^2, (x_1+(e^{-\delta t}-1)x_2+tx_3)x_2,e^{-\delta t} x_2^2, x_2x_3]$ with $\delta \neq 0$,\label{flow:iiix2third}
\item $\varphi_t(x)=[e^{\delta t}x_0x_2-2A(e^{-\delta t}-e^{\delta t}) x_1x_2+
\big(b t e^{\delta t}+A(e^{-\delta t}-e^{\delta t}) \big)x_2^2+
\big(-2 A t e^{-\delta t}+ 2 B (1-e^{-\delta t}) \big)x_2x_3+
A(e^{-\delta t}-e^{\delta t})x_1^2+2 \big(B(e^{-\delta t}-e^{\delta t})+At e^{-\delta t}\big)x_1x_3+
\big(C(e^{-\delta t}-e^{\delta t})+At^2e^{-\delta t}+2Bt e^{-\delta t} \big)x_3^2, (x_1+(e^{\delta t}-1)x_2+tx_3)x_2,
e^{\delta t} x_2^2, x_2x_3]$ with $\delta \neq 0$,\label{flow:iiix2fourth}
\item $\varphi_t(x)=[e^{\delta t}x_0x_3+ A(e^{2\alpha t}-e^{\delta t})x_1^2+2\big(B(e^{2 \alpha t}-e^{\delta t})+
A t e^{2\alpha t} \big)x_1x_2+ \big(C(e^{2\alpha t}-e^{\delta t})+(A t^2+2 B t)e^{2 \alpha t} \big)x_2^2,e^{\alpha t} (x_1+t x_2)x_3, 
e^{\alpha t} x_2x_3, x_3^2]$ with $\delta \neq 0$,\label{flow:ivx3first}
\item $\varphi_t(x)=[x_0x_3+c t x^3_3+A(e^{2\alpha t}-1)x_1^2+2\big(B(e^{2 \alpha t}-1)+A t e^{2\alpha t} \big)x_1x_2+ 
\big(C(e^{2\alpha t}-1)+(A t^2+2 B t)e^{2 \alpha t} \big)x_2^2,e^{\alpha t} (x_1+t x_2)x_3, e^{\alpha t} x_2x_3, x_3^2]$ 
with $\alpha \neq 0$, \label{flow:ivx3second}
\item $\varphi_t(x)=[e^{\delta t}x_0x_3+ at e^{\delta t} x_1 x_3 + \frac{a}{2} t^2 e^{\delta t} x_2x_3+A(e^{2\delta t}-e^{\delta t})x_1^2+
2\big(B(e^{2 \delta t}-e^{\delta t})+A t e^{2\delta t} \big)x_1x_2+ 
\big(C(e^{2\delta t}-e^{\delta t})+(A t^2+2 B t)e^{2 \delta t} \big)x_2^2,e^{\delta t} (x_1+t x_2)x_3, e^{\delta t} x_2x_3, x_3^2]$ with 
$\delta \neq 0$,\label{flow:ivx3third}
\item $\varphi_t(x)=[e^{2\alpha t}x_0x_3+A t e^{2\alpha t} x_1^2+(2B t +A t^2) e^{2 \alpha t} x_1x_2+ \big( Ct+B t^2+ 
\frac{A}{3} t^3) e^{2\alpha t} \big)x_2^2,e^{\alpha t} (x_1+t x_2)x_3, e^{\alpha t} x_2x_3, \\ x_3^2]$ with $\delta \neq 0$.\label{flow:ivx3fourth}
\item $\varphi_t(x)=[e^{\delta t} x_0x_2+  2 A t e^{\alpha t} x_1x_2+
A t^2 e^{\alpha t} x^2_2+2 t B x_2 x_3+ A(e^{\alpha t}-e^{\delta t})x_1^2+2 B (1-e^{\delta t})x_1 x_3+ C(e^{-\alpha t}-e^{\delta t})x_3^2, 
e^{\alpha t} (x_1+tx_2)x_2, e^{\alpha t} x_2^2, x_2x_3]$ with $\delta \neq 0$,\label{flow:ivx2first}
\item $\varphi_t(x)=[ x_0x_2+ 2 A t e^{\alpha t} x_1x_2+
 A t^2 e^{\alpha t} x^2_2+(ct+B t^2) x_2x_3+ A(e^{\alpha t}-1)x_1^2+2 B t x_1 x_3+ C(e^{-\alpha t}-1)x_3^2, e^{\alpha t} (x_1+tx_2)x_2, 
e^{\alpha t} x_2^2, x_2x_3]$ with $\alpha \neq 0$, \label{flow:ivx2second}
\item $\varphi_t(x)=[e^{\delta t} x_0x_2+  (at +A t^2) e^{\delta t}x_1x_2+ (\frac{a}{2}t^2+ \frac{A}{3} t^3)e^{\delta t} x^2_2+
2 t B x_2 x_3+ A t e^{\delta t} x_1^2+2 B (1-e^{\delta t})x_1 x_3+ C(e^{-\delta t}-e^{\delta t})x_3^2, e^{\delta t} (x_1+tx_2)x_2, 
e^{\delta t} x_2^2, x_2x_3]$ with $\delta \neq 0$,\label{flow:ivx2third}
\item $\varphi_t(x)=[e^{\delta t} x_0x_2+ 2A t e^{-\delta t} x_1x_2+
A t^2 e^{-\delta t} x^2_2+ 2 B t x_2 x_3+ A(e^{-\delta t}-e^{\delta t})x_1^2+2 B (1-e^{\delta t})x_1 x_3+ C t e^{\delta t}x_3^2, 
e^{-\delta t} (x_1+tx_2)x_2, e^{-\delta t} x_2^2, x_2x_3]$ with $\delta \neq 0$.\label{flow:ivx2fourth}
\item $\varphi_t(x)=[e^{\delta t} x_0x_3+\big(At^2+2tB\big)x_1x_3+\big(At^3+3Bt^2+2Ct\big)x_2x_3+\big(\frac{A}{4}t^4 + B t^3+ C t^2 \big) x^2_3 + 
A(1-e^{\delta t})x_1^2+ 2\big(B(1-e^{\delta t})+A t \big)x_1x_2 + \big(C(1-e^{\delta t})+2 B t + A t^2\big) x_2^2 , 
\big(x_1+tx_2+\frac{t^2}{2} x_3 \big)x_3, (x_2+tx_3)x_3, x_3^2]$ with $\delta \neq 0$.\label{flow:vx3first}
\end{enumerate}
\end{enumerate}
where all the parameters are constant complex numbers. Moreover, two such flows $\varphi_t^1$ and $\varphi_t^2$ are linearly conjugated if 
and only if they are of the same type with the same $\delta, \alpha, \beta$ and there exists $\lambda \in \C^{\ast}$ such that 
$(A_1,B_1,C_1,a_1,b_1,c_1)=\lambda (A_2,B_2,C_2,a_2,b_2,c_2)$ unless $\varphi_t^1$ and $\varphi_t^2$ are both in the cases \ref{flow:iix3first} or 
\ref{flow:iix3fourth} and the linear conjugation switches $x_1$ and $x_2$.
\end{thm}

\begin{rem}
As in the previous cases up to a normalization ot the time coordinate $t$ we can assume that: in case {\textbf I)} either $\alpha, \beta$ o $\mu$ are 1
(if not vanishing) and in case {\textbf II)} that either $\delta=1$ or we have $\delta=0$ and $\alpha$ or $\beta$ are 1 (if not vanishing). 
\end{rem}

\begin{exam}
Note that in case \ref{eq:exam14} for $A=B=C=1$ we have 
${C_I}_t=\big\{ \big( (e^{\alpha t}-1)x_1+(e^{\beta t}-1)x_2\big)\cdot \big( (e^{\alpha t}+1)x_1+(e^{\beta t}+1)x_2\big)=x_3=0\big\}$. Therefore 
$\varphi_t \in \osc$ for all $t\neq 0$ and there is no line $L\subset {C_I}_t$ which is fix. However $\varphi_t$ preserves the families of 
hyperplanes through $x_1=x_3=0$ or $x_2=x_3=0$.
\end{exam}

Indeed we have:

\begin{cor}
Let $\varphi_t$ be a quadratic flow in $\tang{/\!/}\cup \osc\cup \lin$. Then there is (at least) a line $L_I$ such that $\varphi_t$ 
preserves the family of hyperplanes through $L_I$. 
\end{cor}

\begin{proof}
It is enough to take $L_I=\{x_1=x_3=0\}$ for flows in case \textbf{I)} and for flows \ref{eq:ambx1} and \ref{eq:ambx1bis} in case 
\textbf{II)} and $L_I=\{x_2=x_3=0\}$ for the rest of flows.
\end{proof}

\begin{rem}
Even in the case when $P_I$ is fix one cannot obtain a result in the spirit of Corollary \ref{cor:gen} for $P_I\in L_I$ for every possible 
choice of $L_I$ as in the previous Corollary.
\end{rem}

\begin{proof}(of theorem \ref{thm:NFtangllosc})
We have already seen that $H$ is fix and that $\varphi_t$ is a polynomial flow.
We will start considering the case $P_I$ mobile. We resume the notation and computations of lemma \ref{lem:previthmtanosc}. 
As $\widetilde{\varphi}_0=\mathrm{Id}$ we have $\alpha_0=\gamma_0=0$ and $\tilde{\ell}_{0_0}=x$, $\tilde{\ell}_{1_0}=y$ and 
$\tilde{\ell}_{2_0}=z$. Let us impose now $\widetilde{\varphi}_{t+s}=\widetilde{\varphi}_s \circ \widetilde{\varphi}_t$. 
Using again that $P_t$ is mobile one concludes that $\tilde{\ell}_{1_t}=\mu_t \cdot y$ (note that otherwise the second component 
of the composition $\widetilde{\varphi}_s \circ \widetilde{\varphi}_t$ would have a quadratic term) and, as $\varphi_t$ is a flow,
 $\mu_t=e^{\mu t}$. The first component of $\widetilde{\varphi}_s \circ \widetilde{\varphi}_t$ is
\begin{multline*}
\alpha_s e^{2 \mu t} y^2 + A^0_s \alpha_t y^2+ A^0_s \tilde{\ell}_{0_t}+ B^0_s e^{\mu t} \tilde{\ell}_{1_t}+C^0_s \gamma_t y^2 + C^0_s \tilde{\ell}_{2_t}+D^0_s=\\ (\alpha_s e^{2 \mu t} + A^0_s \alpha_t + C^0_s \gamma_t) y^2 + \tilde{\ell}_{0_s}(\tilde{\ell}_{i_t}),
\end{multline*}
which yields 
\begin{equation}\label{eq:PImobil3}
\alpha_{t+s}=e^{2 \mu t} \alpha_s+ A^0_s \alpha_t + C^0_s \gamma_t
\end{equation}
and $\tilde{\ell}_{0_s}(\tilde{\ell}_{i_t})=\tilde{\ell}_{0_{t+s}}$. Analogously, for the third component one obtains 
\begin{equation}\label{eq:PImobil4}
\gamma_{t+s}=e^{2 \mu t} \gamma_s+ A^2_s \alpha_t + C^2_s \gamma_t
\end{equation}
and $\tilde{\ell}_{2_s}(\tilde{\ell}_{i_t})=\tilde{\ell}_{2_{t+s}}$. In particular $\widetilde{\Psi}_t=(\tilde{\ell}_{0_t}, e^{\mu t} y, \tilde{\ell}_{2_t})$ 
is a linear flow in $\C^3$. A direct computation shows that there are the following possibilities for $\widetilde{\Psi}_t$:
\begin{enumerate}[\bf (i)]
\item $\widetilde{\Psi}_t=(e^{\alpha t} x, e^{\mu t} y, e^{\beta t} z)$,
\item $\widetilde{\Psi}_t=(e^{\mu t}x+t e^{\mu t} y, e^{\mu t} y, e^{\beta t} z)$,
\item $\widetilde{\Psi}_t=(e^{\alpha t} x, e^{\mu t} y, z+t)$,
\item $\widetilde{\Psi}_t=(e^{\mu t} x+ t e^{\mu t} y, e^{\mu t} y, z+t)$.
\end{enumerate}
Therefore one can assume $C^0_t=A^2_t=0$ and equations \ref{eq:PImobil3} and \ref{eq:PImobil4} yield 
$$\alpha_t=a (e^{2 \mu t}-A^0_t), \qquad \gamma_t=b (e^{2\mu t}-C^2_t)$$
and note that $A^0_t$ is equal to $e^{\alpha t}$ or $e^{\mu t}$ and $\gamma_t$ is equal to $e^{\beta t}$ or $1$. 
We can assume $A^0_t=e^{\alpha t}$ and $\gamma_t=e^{\beta t}$, we obtain then the list of flows in \textbf{I)} and the statement in \textbf{vii)}.

\medskip

Let us tackle now the case $P_I$ fix. Note that we can use the same arguments as in lemma \ref{lem:previthmosctanbis} 
to conclude that one can assume $P_I=[1,0,0,0]$ and 
$$\varphi_t[x_0,x_1,x_2,x_3]=\big[\delta_t x_0 + \frac{q_t}{\ell_t}, {\ell_1}_t, {\ell_2}_t, {\ell_3}_t \big]$$
with ${\ell_1}_t, {\ell_2}_t, {\ell_3}_t, \ell_t\in A_1(x_1,x_2,x_3)$, $q_t\in A_2(x_1,x_2,x_3)$, $q_0=0$, $\delta_t=e^{\delta t}$
and $\ell_{i_0}=x_i$ for $i=1,2,3$. Moreover as $\Psi_t(x)=[{\ell_1}_t,{\ell_2}_t,{\ell_3}_t]$ is a linear flow on $\pp^2$ it belongs to the following list: 
\begin{enumerate}[\bf (i)]
\item $\Psi_t(x)=[e^{\alpha t} x_1, x_2+t x_3, x_3]$,
\item $\Psi_t(x)=[e^{\alpha t} x_1, e^{\beta t} x_2, x_3]$,
\item $\Psi_t(x)=[x_1+(e^{\alpha t}-1) x_2 +t x_3, e^{\alpha t} x_2, x_3]$,
\item $\Psi_t(x)=[e^{\alpha t} (x_1+t x_2), e^{\alpha t} x_2, x_3]$,
\item $\Psi_t(x)=[x_1+tx_2 + \frac{t^2}{2} x_3, x_2+tx_3, x_3]$.
\end{enumerate}
The statements in \textbf{v)} and \textbf{vi)} follow immediately. 
The condition $\varphi_{t+s}=\varphi_s\circ \varphi_t$ translates into 
$$\ell({\ell_i}_t)\cdot q_{t+s}= e^{\delta s}\cdot q_t \cdot \ell({\ell_i}_t) + \ell \cdot q_s({\ell_i}_t)$$
or equivalently 
\begin{equation}\label{eq:LaEquacio}
\ell({\ell_i}_t)\cdot ( q_{t+s}-e^{\delta s} q_t)= \ell \cdot q_s ({\ell_i}_t).
\end{equation}
If $(\ell({\ell_i}_t))=0)$ were mobile then $q_{t+s}-e^{\delta s} q_t$ has $\ell$ as a factor, which implies (taking $t=0$) 
that $q_s$ has $\ell$ as a fixed factor. However this would yield $\varphi_t$ linear for every $t$, which is a contradiction.
Therefore
\begin{equation}\label{eq:LaEqA}
\ell({\ell_i}_t)=\lambda(t) \ell  
\end{equation}
and
\begin{equation}\label{eq:LaEqB}
\lambda(t) (q_{t+s}-e^{\delta s} q_t)=q_s({\ell_i}_t).
\end{equation}
When we impose the condition \ref{eq:LaEqA} we conclude that we have the following possibilities:
\begin{enumerate}[\bf 1)]
\item $\ell=x_3$, $\lambda(t)=1$ and $\Psi_t$ is either of the previous flows,
\item $\ell=x_2$ and $\Psi_t$ is of types \textbf{iii)} or \textbf{iv)} and $\lambda(t)=e^{\alpha t}$,
\item  or $\ell=x_1$, $\lambda(t)=e^{\alpha t}$ and $\Psi_t$ is of type \textbf{i)}.
\end{enumerate}
Indeed, notice that in case \textbf{ii)} up to a change of coordinates one can assume that $\ell=x_3$, which is the reason why 
we have excluded the other possibilities of the previous list. 
We proceed now to determine $q_t$ in each case. 
\begin{enumerate}[\bf 1)]
\item $\ell=x_3$, $\lambda(t)=1$ and $\Psi_t$ of the following type:
\begin{enumerate}[\bf i)]
 \item We can assume $q_t(x)=x_3 v_t + \tilde{q}_t$ where $\tilde{q}_t\in A_2(x_1,x_2)$ and
$v_t\in A_1(x_1,x_2,x_3)$. Equation \ref{eq:LaEqB} is rewritten as 
\begin{equation}\label{eq:here}
x_3(v_{t+s}-e^{\delta s} v_t)+\tilde{q}_{t+s}-e^{\delta_s} \tilde{q}_t=x_3 v_s({\ell_i}_t)+\tilde{q}_s({\ell_i}_t).
\end{equation}
We denote $\tilde{q}_t=A(t) x_1^2+2 B(t) x_1 x_2 + C(t) x_2^2$ and $v_t=a(t) x_1 + b(t) x_2+ c(t) x_3$. Then equation
 \ref{eq:here} is equivalent to
\begin{eqnarray*}
A(t+s)-e^{\delta s} A(t) &=& e^{2 \alpha t} A(s) \\
B(t+s)-e^{\delta s} B(t) &=& e^{\alpha t} B(s) \\
C(t+s)-e^{\delta s} C(t) &=& C(s) \\
a(t+s)-e^{\delta s} a(t) &=& e^{\alpha t} a(s)+2te^{\alpha t} B(s) \\
b(t+s)-e^{\delta s} b(t) &=& b(s)+2tC(s) \\
c(t+s)-e^{\delta s} c(t) &=& c(s) + t^2 C(s) + t b(s)
\end{eqnarray*}
yielding solutions
\begin{eqnarray*}
A(t) &=& A (e^{2 \alpha t}-e^{\delta t}) \quad (\mathrm{or} \quad A t e^{\delta t} \quad \mathrm{for} \quad \delta=2 \alpha) \\
B(t) &=& B (e^{\alpha t}-e^{\delta t})  \quad (\mathrm{or} \quad Bt e^{\delta t} \quad \mathrm{for} \quad \delta=\alpha)  \\
C(t) &=& C (1-e^{\delta t}) \quad \mathrm{(or} \quad Ct \quad \mathrm{for} \quad \delta=0)  \\
a(t) &=& a (e^{\alpha t}-e^{\delta t})+2 t e^{\alpha t} B \quad (\mathrm{or} \quad at e^{\delta t}+Bt^2 e^{\delta t} \quad \mathrm{for} \quad \delta=\alpha)  \\
b(t) &=& b(1-e^{\delta t})+2 t C \quad \mathrm{(or} \quad b t+C t^2 \quad \mathrm{for} \quad \delta=0) \\
c(t) &=& c(1-e^{\delta t})+t b + t^2 C \quad \mathrm{(or} \quad c t+ \frac{b}{2} t^2 + C t^3 \quad \mathrm{for} \quad \delta=0) 
\end{eqnarray*}
and the flows
\begin{itemize}
\item[$\bullet$] $\varphi_t(x)=[e^{\delta t} x_0x_3+\big(a(e^{\alpha t}-e^{\delta t})+2tB e^{\alpha t} \big)x_1x_3+
\big(b(1-e^{\delta t})+2tC \big)x_2x_3+\big(c(1-e^{\delta t})+t^2 C+tb \big)x_3^2+
A(e^{2 \alpha t}-e^{\delta t})x_1^2+2B(e^{\alpha t}-e^{\delta t})x_1x_2+C(1-e^{\delta t})x_2^2, e^{\alpha t}x_1x_3,
x_2x_3+tx_3^2, x_3^2]$ with $\delta \neq 0$, $\delta \neq \alpha$, $\delta\neq 2 \alpha$.
\item[$\bullet$] $\varphi_t(x)=[x_0x_3+(a(e^{\alpha t}-1)+2t B e^{\alpha t})x_1x_3+
(bt+C t^2)x_2x_3+(ct+\frac{b}{2}t^2+\frac{C}{3}t^3)x_3^2+A(e^{2 \alpha t}-1)x_1^2+2B(e^{\alpha t}-1)x_1x_2+Ct x_2^2, 
e^{\alpha t} x_1x_3, x_2x_3+tx_3^2, x_3^2]$ with $\alpha \neq 0$, 
\item[$\bullet$] $\varphi_t(x)=[x_0x_3+(a t+B t^2)x_1x_3+(b t +C t^2)x_2x_3+(ct+ \frac{b}{2}t^2+\frac{C}{3}t^3)x_3^2+
At x_1^2+2Bt x_1x_2+Ctx_2^2, x_1x_3, x_2x_3+tx_3^2, x_3^2]$,
\item[$\bullet$] $\varphi_t(x)=[e^{\delta t} x_0x_3+(a t+B t^2) e^{\delta t}x_1x_3+
\big(b(1-e^{\delta t})+2tC \big)x_2x_3+\big(c(1-e^{\delta t})+t^2 C+tb \big)x_3^2+
A (e^{2\delta t}- e^{\delta t}) x_1^2+2B t e^{\delta t} x_1x_2+C(1-e^{\delta t})x_2^2, e^{\delta t}x_1x_3,
x_2x_3+tx_3^2, x_3^2]$ with $\delta \neq 0$,
\item[$\bullet$] $\varphi_t(x)=[e^{\delta t} x_0x_3+\big(a(e^{\frac{\delta t}{2}}-e^{\delta t})+2tB e^{\frac{\delta t}{2}}\big)x_1x_3+
\big(b(1-e^{\delta t})+2tC \big)x_2x_3+\big(c(1-e^{\delta t})+t^2 C+tb \big)x_3^2+
A t e^{\delta t} x_1^2+2B(e^{\frac{\delta t}{2}}-e^{\delta t})x_1x_2+C(1-e^{\delta t})x_2^2, e^{\frac{\delta t}{2}}x_1x_3,
x_2x_3+tx_3^2, x_3^2]$ with $\delta \neq 0$.
\end{itemize}
\end{enumerate}
\end{enumerate}

For the rest of the cases see the analogous computations in the Appendix A. 

\medskip

We will discuss now when two flows of the previous list are linearly conjugated. Let $\varphi_t^1·$ and $\varphi_t^2$ be two flows of the 
previous lists such that there is a linear conjugation $A=[\mu_0,\mu_1,\mu_2,\mu_3]$ with $\mu_i\in A_1(x_0,x_1,x_2,x_3)$ such that
$$\varphi_t^2= A^{-1} \circ \varphi_t^1 \circ A.$$
It is not difficult to convince oneself that in order to be linearly conjugated $\varphi_t^1$ and $\varphi_t^2$ must be both 
of type \textbf{I)} or \textbf{II)}. 

Assume that they are of type \textbf{I)}. As $H=(x_3=0)$ is fix we can assume that $\mu_3=x_3$ and that $A$ induces a linear conjugation 
between the induced polynomial flows $\widetilde{\varphi}_t^1$ and $\widetilde{\varphi}_t^2$ in $\C^3$. One can check that the only 
possibilities of conjugation are between flows of the same type (\textbf{a)}, \textbf{b)}, \textbf{c)} or \textbf{d)}) with 
$\alpha_1=\alpha_2$ and $\beta_1=\beta_2$. In this case the only conjugations are those which allow to assume that the pair $(a,b)$ is of 
the form $(a,1)$ with $a\in \C^{\ast}$. 

We assume now that $\varphi^1_t, \varphi^2_t$ are two flows of type \textbf{II)} and $A=[\mu_0,\mu_1,\mu_2,\mu_3]$ with 
$\mu_i\in A_1(x_0,x_1,x_2,x_3)$ such that 
$$\varphi_t^2=A^{-1}\circ \varphi_t^1 \circ A.$$
As $P_I=[1,0,0,0]$ and the planes $H^1$ and $H^2$ corresponding to the flows $\varphi_t^1$ and $\varphi_t^2$ respectively are of the
form $H^1=(x_{i_1}=0)$ or $H^2=(x_{i_2}=0)$, we have $\mu_i\in A_i(x_1,x_2,x_3)$ for $i=1,2,3$ and $\mu_{i_2}=\mu^{i_1}_{i_2} x_{i_1}$. 
Moreover we can choose $\mu_{i_2}^{i_1}=1$. This implies that the linear flows in $\Psi_t^1$ , $\Psi_t^2$ in $\pp^2$ corresponding to 
$\varphi_t^1, \varphi_t^2$ (same notation as before) are linearly conjugated by $\tilde{A}=[\mu_1,\mu_2,\mu_3]$. One concludes 
that $\tilde{A}=[x_1,x_2,x_3]$ and $\Psi_t^1=\Psi_t^2$. In particular $i_1=i_2$ so $H^1=H^2$. We denote now
$$\mu_0(x)=\mu_0^0 x_0 + \mu_0^1 x_1 + \mu_0^2 x_2 + \mu_0^3 x_3$$
and impose $A\circ \varphi_t^2=\varphi_t^1 \circ A$ on the first component for each of the previous cases. 
Let us make the computations for the first case (which corresponds to the flow 1):

\begin{multline*}
\mu_0^0 \big( e^{\delta_2 t} x_0x_3 + (a_2 (e^{\alpha_2 t}-e^{\delta_2 t})+2t B_2 e^{\alpha_2 t})x_1x_3+
(b_2(1-e^{\delta_2 t})+2 t C_2) x_2 x_3 +(c_2 (1-e^{\delta_2 t})+t^2 C_2 + t b_2) x_3^2+\\A_2 (e^{2 \alpha_2 t}-e^{\delta_2 t})x_1^2+ 
2 B_2 (e^{\alpha_2 t}-e^{\delta_2 t}) x_1x_2+C_2 (1-e^{\delta_2 t}) x_2^2 \big)+
\mu_0^1 e^{\alpha_2 t} x_1x_3+ \mu_0^2 x_2x_3 + \mu_0^2 t x_3^2 + \mu_0^3 x_3^2= \\
e^{\delta_1 t} (\mu_0^0 x_0 + \mu_0^1 x_1 + \mu_0^2 x_2 + \mu_0^3) x_3 +\big(a_1(e^{\alpha_1 t}-e^{\delta_1 t}) + 2t B_1 e^{\alpha_1 t}\big)x_1x_3 + 
\big(b_1(1-e^{\delta_1 t}) +2t C_1\big) x_2x_3+ \\ \big(  c_1(1-e^{\delta_1 t})+t^2 C_1 + t b_1 \big) x_3^2 + 
A_1 (e^{2 \alpha_1 t}-e^{\delta_1 t}) x_1^2+ 2 B_1 (e^{\alpha_1 t}-e^{\delta_1 t})x_1x_2 + C_1(1-e^{\delta_1 t}) x_2^2.
\end{multline*}

If follows that $\delta_1=\delta_2$ and

\begin{eqnarray*}
A_1 &=& \mu_0^0 A_2 \\
B_1 &=& \mu_0^0 B_2 \\
C_1 &=& \mu_0^0 C_2 \\
a_1 &=& \mu_0^0 a_2 + \mu_0^1 \\
b_1 &=& \mu_0^0 b_2 + \mu_0^2 \\
c_1 &=& \mu_0^0 c_2 + \mu_0^3.
\end{eqnarray*}

As $\mu_0^0\in \C^{\ast}$ and $\mu_0^1,\mu_0^2,\mu_0^3\in \C$ we can assume that
$a_1=b_1=c_1=0$ and $A_1,B_1,C_1$ are defined modulo a $\C^{\ast}$ action. Therefore we obtain the flow in \ref{flow:ix3first} 
and we conclude that two flows of this type are linearly conjugated if and only if there exists $\mu_0^0\in \C^{\ast}$ such that  
$(A_1,B_1,C_1)=\mu^0_0 (A_2,B_2,C_2)$. Analogous calculations yield the flows \ref{flow:ix3second}, \ref{flow:ix3third}, \ref{flow:ix3fourth} 
and \ref{flow:ix3fifth}.

Analogously we obtain the following flows for each of the cases: \textbf{1) ii)} from \ref{flow:iix3first} to \ref{flow:iix3eleventh}, 
\textbf{1) iii)} from \ref{flow:iiix3first}  to \ref{flow:iiix3fourth},
\textbf{1) iv)} from \ref{flow:ivx3first} to \ref{flow:ivx3fourth}, \textbf{1) v)} \ref{flow:vx3first}, \textbf{2) iii)} from 
\ref{flow:iiix2first} to \ref{flow:iiix2fourth}, \textbf{ 2) iv)} from \ref{flow:ivx2first} to
\ref{flow:ivx2fourth} and \textbf{3) iii)} from \ref{flow:ix1first} to \ref{flow:ix1fifth}. Note that in some cases due to simmetry some
extra considerations are necessary.

\end{proof}

\subsection{Conclusions}

We can sum up the results of this section in the following results:

\begin{thm}
Let $\varphi_t$ be a quadratic flow in $\mathsf{\mathbf{Bir}}^3$. Then one of the following possibilities hold:
\begin{enumerate}[\bf a)]
\item $\varphi_t\in \gen{/\!/}\cup \lin$, $P_I$, $S$, $C_I$ are fix and $H_t$ is mobile,
\item $\varphi_t \in \tang{O}\cup \tang{\times}\cup \lin$, $P_I$, $S$ are fix and $H_t$, ${C_I}_t$ are mobile,
\item $\varphi_t \in \tang{/\!/}\cup \lin$, $H$, $C_I$ are fix and ${P_I}_t$ can be either fix or mobile,
\item $\varphi_t \in \osc \cup \lin$, $H$, $P_I$ are fix and $C_I$ can be either fix or mobile.
\end{enumerate}
Moreover, if $H$ is fix then $\varphi_t$ is a polynomial flow, i.e. ${\varphi_t}_{|\pp^3\backslash H}: \pp^3\backslash H \rightarrow \pp^3 \backslash H$ is
polynomial for each $t$. In particular there are no flows in $\gen{O}\cup \gen{\times}\cup \lin$.
\end{thm}

\begin{thm}
Let $\varphi_t$ be a quadratic flow in $\mathsf{\mathbf{Bir}}^3$. Then:
\begin{enumerate}[\bf i)]
\item There exists a line $L$ such that $\varphi_t$ preserves the family of hyperplanes through $L$. Moreover in cases {\bf a)} and {\bf c)} 
we can choose $L=C_I$.
\item If $P_I$ is fix then $\varphi_t$ preserves the family of hyperplanes through $P_I$ (in particular the family of lines through $P_I$).
\end{enumerate}
\end{thm}

\begin{cor}
Let $\varphi_t$ a quadratic generic flow in $\mathsf{\mathbf{Bir}}^3$. Then $\varphi_t$ is determined by a linear flow $\eta_t$ on the $\pp^2$ of the net of lines 
through $P_I$ and a linear flow $\chi_t$ on the $\pp^1$ of the pencil of planes through the line $C_I$. Namely
$$\varphi_t(P)=\eta_t(P\vee P_I) \cap \chi_t(P\vee C_I).$$ 
\end{cor}

\begin{thm}
Let $\varphi_t$ be a quadratic flow in $\mathsf{\mathbf{Bir}}^3$. Then, up to linear conjugation, $\varphi_t$ is in one (and only one) of the lists of theorems 
\ref{thm:NFgen} (generic case), \ref{thm:NFtangox1} (non-generic with conic $C_I$ of multiplicity 1 at $P_I$) and 
\ref{thm:NFtangllosc} (non-generic with conic $C_I$ of multiplicity 2 at $P_I$).  
\end{thm}

\begin{rem}
As in $\pp^2$ using the classification one verifies that if $\varphi_t$ is a germ of quadratic flow then $\varphi_t$ is defined for every $t\in \C$. 
\end{rem}

\appendix

\section{The computations of Theorem \ref{thm:NFtangox1}}

We resume here the computations where we left them, using the same notation.
\medskip
\begin{enumerate}[\bf 1)]
\item $\ell=x_3$, $\lambda(t)=1$ and $\Psi_t$ of the following type:
\medskip

\begin{itemize}
\item[\bf ii)] One obtains the equations
\begin{eqnarray*}
A(t+s)-e^{\delta s} A(t) &=& e^{2 \alpha t} A(s) \\
B(t+s)-e^{\delta s} B(t) &=& e^{\alpha t} e^{\beta t} B(s) \\
C(t+s)-e^{\delta s} C(t) &=& e^{2 \beta t} C(s) \\
a(t+s)-e^{\delta s} a(t) &=& e^{\alpha t} a(s) \\
b(t+s)-e^{\delta s} b(t) &=& e^{\beta t} b(s) \ \\
c(t+s)-e^{\delta s} c(t) &=& c(s) 
\end{eqnarray*}
and the solutions
\begin{eqnarray*}
A(t) &=& A (e^{2 \alpha t}-e^{\delta t}) \quad \mathrm{(or} \quad At e^{\delta t} \quad \mathrm{for} \quad \delta=2\alpha) \\
B(t) &=& B (e^{(\alpha+\beta) t}-e^{\delta t})  \quad \mathrm{(or} \quad Bt e^{\delta t} \quad \mathrm{for} \quad \delta=\alpha+\beta)  \\
C(t) &=& C (e^{2 \beta t}-e^{\delta t}) \quad \mathrm{(or} \quad Ct e^{\delta t} \quad \mathrm{for} \quad \delta=2\beta)  \\
a(t) &=& a (e^{\alpha t}-e^{\delta t}) \quad \mathrm{(or} \quad a t e^{\delta t} \quad \mathrm{for} \quad \delta=\alpha)  \\
b(t) &=& b(e^{\beta t}-e^{\delta t}) \quad \mathrm{(or} \quad b t e^{\delta t} \quad \mathrm{for} \quad \delta=\beta) \\
c(t) &=& c(1-e^{\delta t}) \quad \mathrm{(or} \quad c t \quad \mathrm{for} \quad \delta=0) 
\end{eqnarray*}
Up to a linear conjugation switching $x_1$ and $x_2$ we obtain the following list of flows from the previous solutions:
\medskip
\begin{itemize}
\item[$\bullet$] $\varphi_t(x)=[e^{\delta t} x_0x_3+ a (e^{\alpha t}-e^{\delta t})x_1x_3+b(e^{\beta t}-e^{\delta t})x_2x_3+c(1-e^{\delta t})x_3^2+A (e^{2 \alpha t}-e^{\delta t})x_1^2+2 B (e^{(\alpha + \beta)t}-e^{\delta t}) x_1x_2+C(e^{2 \beta t}-e^{\delta t})x_2^2, e^{\alpha t}x_1x_3, e^{\beta t} x_2x_3, x_3^2]$ with $\delta \neq 0$, $\delta \neq \alpha,
\beta,\alpha + \beta, 2 \alpha, 2 \beta$,
\item[$\bullet$] $\varphi_t(x)=[x_0x_3+ a (e^{\alpha t}-1)x_1x_3+b(e^{\beta t}-1)x_2x_3+c t x_3^2+A (e^{2 \alpha t}-1)x_1^2+2 B (e^{(\alpha + \beta)t}-1) x_1x_2+C(e^{2 \beta t}-1)x_2^2, e^{\alpha t}x_1x_3, e^{\beta t} x_2x_3, x_3^2]$ with $\alpha \neq 0$ and $\beta \neq 0$,
\item[$\bullet$] $\varphi_t(x)=[e^{\delta t} x_0x_3+ a t e^{\delta t} x_1x_3+b(e^{\beta t}-e^{\delta t})x_2x_3+c(1-e^{\delta t})x_3^2+A (e^{2 \delta t}-e^{\delta t})x_1^2+2 B (e^{(\delta + \beta)t}-e^{\delta t}) x_1x_2+C(e^{2 \beta t}-e^{\delta t})x_2^2, e^{\delta t}x_1x_3, e^{\beta t} x_2x_3, x_3^2]$ with $\delta \neq 0$, $\delta=\alpha$, $\delta \neq \beta, 2 \beta$,
\item[$\bullet$] $\varphi_t(x)=[e^{2\alpha t} x_0x_3+ a (e^{\alpha t}-e^{2\alpha t})x_1x_3+b(e^{\beta t}-e^{2\alpha t})x_2x_3+c(1-e^{2\alpha t})x_3^2+A t e^{2\alpha t} x_1^2+2 B (e^{(\alpha + \beta)t}-e^{2\alpha t}) x_1x_2+C(e^{2 \beta t}-e^{2\alpha t})x_2^2, e^{\alpha t}x_1x_3, e^{\beta t} x_2x_3, x_3^2]$ with $\alpha \neq 0$, $\beta \neq 0$, $\alpha \neq \beta$,
\item[$\bullet$] $\varphi_t(x)=[e^{\delta t} x_0x_3+ a t e^{\delta t} x_1x_3+b t e^{\delta t}x_2x_3+c(1-e^{\delta t})x_3^2+A (e^{2 \delta t}-e^{\delta t})x_1^2+2 B (e^{2\delta t}-e^{\delta t}) x_1x_2+C(e^{2 \delta t}-e^{\delta t})x_2^2, e^{\delta t}x_1x_3, e^{\delta t} x_2x_3, x_3^2]$ with $\delta \neq 0$, 
\item[$\bullet$] $\varphi_t(x)=[e^{2\alpha t} x_0x_3+ a (e^{\alpha t}-e^{2\alpha t})x_1x_3+b(e^{\alpha t}-e^{2\alpha t})x_2x_3+c(1-e^{2\alpha t})x_3^2+A t e^{2 \alpha t} x_1^2+2 B t e^{2\alpha t} x_1x_2+Cte^{2 \alpha t}x_2^2, e^{\alpha t}x_1x_3, e^{\alpha t} x_2x_3, x_3^2]$ with $\alpha \neq 0$, 
\item[$\bullet$] $\varphi_t(x)=[e^{(\alpha + \beta) t} x_0x_3+ a (e^{\alpha t}-e^{(\alpha +\beta) t})x_1x_3+b(e^{\beta t}-
e^{(\alpha + \beta)t})x_2x_3+c(1-e^{(\alpha+\beta) t})x_3^2+A (e^{2 \alpha t}-e^{(\alpha + \beta) t})x_1^2+2 B t e^{(\alpha + \beta)t} x_1x_2+C(e^{2 \beta t}-e^{(\alpha+\beta)t})x_2^2, e^{\alpha t}x_1x_3, e^{\beta t} x_2x_3, x_3^2]$ with $\alpha\neq 0$, $\beta \neq 0$, $\alpha\neq \beta$,
\item[$\bullet$] $\varphi_t(x)=[e^{\delta t} x_0x_3+ a t e^{\delta t} x_1x_3+b (1-e^{\delta t}) x_2x_3+c(1-e^{\delta t})x_3^2+A (e^{2 \delta t}-e^{\delta t})x_1^2+2 B t e^{\delta t} x_1x_2+C(1-e^{\delta t})x_2^2, e^{\delta t}x_1x_3, x_2x_3, x_3^2]$ with $\delta \neq 0$,
\item[$\bullet$] $\varphi_t(x)=[x_0x_3+ a t x_1x_3+b(e^{\beta t}-1)x_2x_3+c t x_3^2+A t x_1^2+2 B (e^{\beta t}-1) x_1x_2+C(e^{2 \beta t}-1)x_2^2, x_1x_3,\\ e^{\beta t} x_2x_3, x_3^2]$ with $\beta \neq 0$,
\item[$\bullet$] $\varphi_t(x)=[x_0x_3+t \big((ax_1+bx_2+cx_3)x_3+(A x_1^2+2 B x_1 x_2+ C x_2^2)\big), x_1x_3,x_2x_3,x_3^2]$, 
\end{itemize}
\medskip
\item[\bf iii)] Again with the same notation we obtain:
\begin{eqnarray*}
A(t+s)-e^{\delta s} A(t) &=& A(s) \\
B(t+s)-e^{\delta s} B(t) &=& e^{\alpha t} B(s) +A(s) (e^{\alpha t}-1)\\
C(t+s)-e^{\delta s} C(t) &=& C(s) e^{2 \alpha t}+(e^{\alpha t}-1)^2 A(s)+ 2 B(s)e^{\alpha t}(e^{\alpha t}-1) \\
a(t+s)-e^{\delta s} a(t) &=& a(s)+ 2 t A(s) \\
b(t+s)-e^{\delta s} b(t) &=& e^{\alpha t} b(s) +(e^{\alpha t}-1)a(s)+2 t (e^{\alpha t}-1) A(s) + 2 t e^{\alpha t} B(s) \\
c(t+s)-e^{\delta s} c(t) &=& c(s) + t^2 A(s) + t a(s)
\end{eqnarray*}
yielding solutions
\begin{eqnarray*}
A(t) &=& A (1-e^{\delta t}) \quad \mathrm{(or} \quad At \quad \mathrm{for} \quad \delta=0) \\
B(t) &=& B (e^{\alpha t}-e^{\delta t})-A(1-e^{\delta t}) 
\\ & & \quad \mathrm{(or} \quad B(e^{\delta t}-1)-A t \quad \mathrm{for} \quad \delta=0,\, \alpha\neq 0,
\\ & & \quad \mathrm{or} \quad B t e^{\delta t} - A(1-e^{\delta t}) \quad \mathrm{for} \quad \alpha=\delta\neq 0),
\\
C(t) &=& C (e^{2 \alpha t}-e^{\delta t})-2 B (e^{\alpha t}-e^{\delta t})+A(1-e^{\delta t})  
\\
& &(\mathrm{or} \quad C(e^{2 \alpha t}-1)-2 B(e^{\alpha t}-1)+A t \quad \mathrm{for} \quad \delta=0, \, \alpha\neq 0,  \\
& &\quad \mathrm{or} \quad C(e^{2 \delta t}-e^{\delta t})-2 B t e^{\delta t} +A (1-e^{\delta t}) \quad \mathrm{for} \quad \delta=\alpha\neq 0,  
\\
& &\quad \mathrm{or} \quad C t e^{2 \alpha t}-2 B(e^{\alpha t}-e^{2\alpha t})+A (1-e^{2\alpha t})\quad \mathrm{for} \quad \delta=2\alpha, \, \alpha\neq 0),  
\\
a(t) &=& a (1-e^{\delta t})+2 t A \quad \mathrm{(or} \quad a t + A t^2 \quad \mathrm{for} \quad \delta=0)  \\
b(t) &=& b(e^{\alpha t}-e^{\delta t})+2 t e^{\alpha t} B -a(1-e^{\delta t})-2 t A \\
&& \quad \mathrm{(or} \quad b (e^{\alpha t}-1)+ 2 t e^{\alpha t} B- A t^2- at \quad \mathrm{for} \quad \delta=0,\,\alpha\neq 0, \\
&& \quad \mathrm{or} \quad (B t^2+b t +a) e^{\delta t} - 2 A t - a \quad \mathrm{for} \quad \alpha=\delta\neq 0), \\
c(t) &=& c(1-e^{\delta t})+a t + A t^2 \quad \mathrm{(or} \quad c t+ \frac{a}{2} t^2 + \frac{A}{3} t^3 \quad \mathrm{for} \quad \delta=0).
\end{eqnarray*}
Note that we can assume $\alpha \neq 0$ (otherwise we obtain one of the flows in \textbf{i)} and that the equations obtained for 
$A(t)+B(t)$,$C(t)+2B(t)+A(t)$ and $a(t)+b(t)$ are easy to solve. We obtain the flows:
\medskip

\begin{itemize}
\item[$\bullet$] $\varphi_t(x)=[e^{\delta t}x_0x_3+\big(a(1-e^{\delta t})+2 t A \big) x_1x_3+
\big(b(e^{\alpha t}-e^{\delta t})+2te^{\alpha t}B-a(1-e^{\delta t})-2tA \big) x_2x_3+ \big(c(1-e^{\delta t})+at+At^2 \big)x_3^2+A(1-e^{\delta t})x_1^2+2\big(B(e^{\alpha t}-e^{\delta t})-A(1-e^{\delta t})\big)x_1x_2+ 
\big(C(e^{2 \alpha t}-e^{\delta t})-2B(e^{\alpha t}-e^{\delta t})+A(1-e^{\delta t})\big)x_2^2, 
x_1x_3+(e^{\alpha t}-1)x_2x_3+tx_3^2, e^{\alpha t}x_2x_3,x_3^2]$ with $\delta \neq 0, \alpha,2 \alpha, \alpha\neq 0$,
\item[$\bullet$] $\varphi_t(x)=[e^{\delta t}x_0x_3+\big(a(1-e^{\delta t})+2 t A \big) x_1x_3+
\big( (B t^2+bt+a)e^{\delta t}-2 A t -a\big) x_2x_3+ \big(c(1-e^{\delta t})+at+At^2 \big)x_3^2+
A(1-e^{\delta t})x_1^2+2\big(B t e^{\delta t}-A(1-e^{\delta t})\big)x_1x_2+ 
\big(C(e^{2 \delta t}-e^{\delta t})-2B t e^{\delta t}+A(1-e^{\delta t})\big)x_2^2, 
x_1x_3+(e^{\delta t}-1)x_2x_3+tx_3^2, e^{\delta t}x_2x_3,x_3^2]$ with $\delta \neq 0$,
\item[$\bullet$] $\varphi_t(x)=[e^{2\alpha t}x_0x_3+\big(a(1-e^{2\alpha t})+2 t A \big) x_1x_3+
\big(b(e^{\alpha t}-e^{2\alpha t})+2te^{\alpha t}B-a(1-e^{2\alpha t})-2tA \big) x_2x_3+ \big(c(1-e^{2\alpha t})+at+At^2 \big)x_3^2+A(1-e^{2\alpha t})x_1^2+2\big(B(e^{\alpha t}-e^{2\alpha t})-A(1-e^{2\alpha t})\big)x_1x_2+ 
\big(C t e^{2 \alpha t}-2B(e^{\alpha t}-e^{2\alpha t})+A(1-e^{2\alpha t})\big)x_2^2, 
x_1x_3+(e^{\alpha t}-1)x_2x_3+tx_3^2, e^{\alpha t}x_2x_3,x_3^2]$ with $\alpha \neq 0$,
\item[$\bullet$] $\varphi_t(x)=[x_0x_3+(at+At^2)x_1x_3+\big(b(e^{\alpha t}-1)+2 B t e^{\alpha t} - A t^2-at)x_2x_3+
\big(\frac{A}{3} t^3+\frac{a}{2} t^2+c t\big)x_3^2+A t x_1^2+ 2 \big(B(e^{\alpha t}-1)-At \big)x_1x_2+
\big(C(e^{2\alpha t}-1)-2B(e^{\alpha t}-1)+At\big)x_2^2, x_1x_3+(e^{\alpha t}-1)x_2x_3+tx_3^2, e^{\alpha t} x_2x_3, x_3^2]$ with $\alpha \neq 0$,
\end{itemize}
\medskip
\item[\bf iv)] One has the equations:
\begin{eqnarray*}
A(t+s)-e^{\delta s} A(t) &=& e^{2 \alpha t} A(s) \\
B(t+s)-e^{\delta s} B(t) &=& e^{2 \alpha t} B(s) +t e^{2 \alpha t} A(s)\\
C(t+s)-e^{\delta s} C(t) &=& e^{2 \alpha t} (C(s)+2 t B(s) + t^2 A(s) ) \\
a(t+s)-e^{\delta s} a(t) &=& e^{\alpha t} a(s) \\
b(t+s)-e^{\delta s} b(t) &=& e^{\alpha t} b(s) + t e^{\alpha t} a(s) \\
c(t+s)-e^{\delta s} c(t) &=& c(s) 
\end{eqnarray*}
yielding solutions
\begin{eqnarray*}
A(t) &=& A (e^{2 \alpha t}-e^{\delta t}) \quad (\mathrm{or} \quad A t e^{\delta t}\quad \mathrm{for}\quad \delta=2\alpha)  \\
B(t) &=& B (e^{2 \alpha t}-e^{\delta t})+ A t e^{2 \alpha t}  \quad (\mathrm{or} \quad ( B t+ \frac{A}{2} t^2) e^{\delta t} \quad \mathrm{for}\quad \delta=2\alpha)  \\
C(t) &=& C (e^{2 \alpha t}-e^{\delta t})+ (A t^2+2 t B)e^{2 \alpha t} \quad (\mathrm{or} \quad (Ct+B t^2+ \frac{A}{3} t^3) e^{\delta t} \quad \delta=2\alpha)   \\
a(t) &=& a (e^{\alpha t}-e^{\delta t}) \quad (\mathrm{or} \quad a t e^{\delta t}\quad \mathrm{for}\quad \delta=\alpha)  \\
b(t) &=& b(e^{\alpha t}-e^{\delta t})+ at e^{\alpha t} \quad (\mathrm{or} \quad (b t +\frac{a}{2} t^2 )e^{\delta t}\quad \mathrm{for}\quad \delta=\alpha)  \\
c(t) &=& c(1-e^{\delta t}) \quad \mathrm{(or} \quad c t \quad \mathrm{for} \quad \delta=0).
\end{eqnarray*}
As before we can assume $\alpha \neq 0$. We obtain the flows:
\medskip
\begin{itemize}
\item[$\bullet$] $\varphi_t(x)=[e^{\delta t}x_0x_3+ a(e^{\alpha t}-e^{\delta t})x_1x_3+\big(b(e^{\alpha t}-e^{\delta t})+a t e^{\delta t}\big)x_2x_3+c(1-e^{\delta t})x^2_3+A(e^{2\alpha t}-e^{\delta t})x_1^2+2\big(B(e^{2 \alpha t}-e^{\delta t})+A t e^{2\alpha t} \big)x_1x_2+ \big(C(e^{2\alpha t}-e^{\delta t})+(A t^2+2 B t)e^{2 \alpha t} \big)x_2^2,e^{\alpha t} (x_1+t x_2)x_3, e^{\alpha t} x_2x_3, x_3^2]$ with $\delta \neq 0$,
\item[$\bullet$] $\varphi_t(x)=[x_0x_3+ a(e^{\alpha t}-1)x_1x_3+\big(b(e^{\alpha t}-1)+a t \big)x_2x_3+c t x^3_3+A(e^{2\alpha t}-1)x_1^2+2\big(B(e^{2 \alpha t}-1)+A t e^{2\alpha t} \big)x_1x_2+ \big(C(e^{2\alpha t}-1)+(A t^2+2 B t)e^{2 \alpha t} \big)x_2^2,e^{\alpha t} (x_1+t x_2)x_3, e^{\alpha t} x_2x_3, x_3^2]]$ with $\alpha \neq 0$, 
\item[$\bullet$] $\varphi_t(x)=[e^{\delta t}x_0x_3+ a t e^{\delta t} x_1x_3+ \big(b t+ a \frac{t^2}{2}\big) e^{\delta t} x_2x_3+c(1-e^{\delta t})x^2_3+A(e^{2\delta t}-e^{\delta t})x_1^2+2\big(B(e^{2 \delta t}-e^{\delta t})+A t e^{2\delta t} \big)x_1x_2+ \big(C(e^{2\delta t}-e^{\delta t})+(A t^2+2 B t)e^{2 \delta t} \big)x_2^2,e^{\delta t} (x_1+t x_2)x_3, e^{\delta t} x_2x_3, x_3^2]$ with $\delta \neq 0$,
\item[$\bullet$] $\varphi_t(x)=[e^{2\alpha t}x_0x_3+ a(e^{\alpha t}-e^{2\alpha t})x_1x_3+\big(b(e^{\alpha t}-e^{2\alpha t})+a t e^{2\alpha t}\big)x_2x_3+c(1-e^{2\alpha t})x^2_3+A t e^{2\alpha t} x_1^2+2\big(B t +A \frac{t^2}{2}\big) e^{2 \alpha t} x_1x_2+ \big( Ct+B t^2+ A \frac{t^3}{3}) e^{2\alpha t} \big)x_2^2,e^{\alpha t} (x_1+t x_2)x_3, e^{\alpha t} x_2x_3, x_3^2]$ with $\delta \neq 0$.
\end{itemize}
\medskip
\item[\bf v)] Finally one has:
\begin{eqnarray*}
A(t+s)-e^{\delta s} A(t) &=& A(s) \\
B(t+s)-e^{\delta s} B(t) &=& B(s) + t A(s) \\
C(t+s)-e^{\delta s} C(t) &=& C(s) + 2 t B(s)+ t^2 A (s) \\
a(t+s)-e^{\delta s} a(t) &=& a(s)+ t^2 A(s) + 2 t B(s) \\
b(t+s)-e^{\delta s} b(t) &=& b(s) + t a(s) + t^3 A(s) + 3 t^2 B(s) + 2 t C(s) \\
c(t+s)-e^{\delta s} c(t) &=& c(s) + t b(s) +\frac{t^2}{2} a(s) + \frac{t^4}{4} A(s) + t^3 B(s) + t^2 C(s).
\end{eqnarray*}
One verifies that if $\delta =0$ the only solutions correspond to linear flows, therefore we can assume $\delta \neq 0$ and we obtain the following solutions:
\begin{eqnarray*}
A(t) &=& A (1-e^{\delta t}) \\
B(t) &=& B (1-e^{\delta t})+ A t \\
C(t) &=& C (1-e^{\delta t})- 2 t B + A t^2 \\
a(t) &=& a (1-e^{\delta t})+ A t^2+ 2 t B  \\
b(t) &=& b (1-e^{\delta t}) + a t+ A t^3 + 3 B t^2+ 2 C t\\
c(t) &=& c(1-e^{\delta t})+b t + \big(\frac{a}{2}+C\big) t^2 + B t^3 + \frac{A}{4} t^4,
\end{eqnarray*}
which yield the flow 
\begin{multline*}\varphi_t(x)=[e^{\delta t} x_0x_3+\big(a(1-e^{\delta t})+At^2+2tB\big)x_1x_2+
\big(b(1-e^{\delta t})+at+At^3+3Bt^2+2Ct\big)x_2x_3+\\ \big(c(1-e^{\delta t})+ b t + \frac{a}{2} t^2 + \frac{A}{4} t^4 + B t^3+ C t^2 \big) x^2_3 + 
A(1-e^{\delta t})x_1^2+ 2\big(B(1-e^{\delta t})+A t \big)x_1x_2 + \\ \big(C(1-e^{\delta t})+2 B t + A t^2\big) x_2^3 , 
\big(x_1+tx_2+\frac{t^2}{2} x_3 \big)x_3, (x_2+tx_3)x_3, x_3^2]
\end{multline*}
with $\delta \neq 0$.
\medskip
\end{itemize}
\medskip
\item  $\ell=x_2$ and $\Psi_t$ of the following type:
\medskip
\begin{itemize}
 \item[\bf iii)] Then $\lambda(t)=e^{\alpha t}$ and we can assume $\alpha \neq 0$. Then
$$q_t(x)=x_2 v_t + \tilde{q}_t(x)$$
where $\tilde{q}_t\in A_2(x_2,x_3)$ and $v_t\in A_1(x_1,x_2,x_3)$. Then equation \ref{eq:LaEqB} is written as
\begin{equation}\label{eq:hier}
 e^{\alpha t}\big( (v_{t+s} -e^{\delta s} v_t)x_2 + (\tilde{q}_{t+s} -e^{\delta s} \tilde{q}_t) \big) =e^{\alpha t} x_2 v_s({\ell_i}_t) + \tilde{q}_s({\ell_i}_t).
\end{equation}
We denote $\tilde{q}_t=A(t) x_1^2+ 2 B(t) x_1x_3+ C(t) x_3^2$ and $v_t(x)=a(t) x_1+ b(t) x_2 + c(t) x_3$. Then
equation \ref{eq:hier} is written as 
\begin{eqnarray*}
A(t+s)-e^{\delta s} A(t) &=& e^{-\alpha t} A(s) \\
B(t+s)-e^{\delta s} B(t) &=& e^{-\alpha t}(B(s) + t A(s)) \\
C(t+s)-e^{\delta s} C(t) &=& e^{-\alpha t} (C(s) + 2 t B(s)+ t^2 A (s)) \\
a(t+s)-e^{\delta s} a(t) &=& a(s)+ 2 (1-e^{-\alpha t}) A(s) \\
b(t+s)-e^{\delta s} b(t) &=& e^{\alpha t} b(s) + (e^{\alpha t}-1)a(s) +(e^{\alpha t}-1)(1-e^{-\alpha t}) A(s) \\
c(t+s)-e^{\delta s} c(t) &=& c(s) + t a(s) + 2 t (1-e^{-\alpha t}) A(s) + 2(1-e^{-\alpha t}) B(s),
\end{eqnarray*}
yielding (solving the equations corresponding to $a(t)+2A(t)$, $a(t)+b(t)+A(t)$ and $c(t)+2B(t)$:
\begin{eqnarray*}
A(t) &=& A (e^{-\alpha t}-e^{\delta t}) \quad (\mathrm{or} \quad A t e^{\delta t} \quad \delta=-\alpha)\\
B(t) &=& B (e^{-\alpha t}-e^{\delta t})+ A t e^{-\alpha t} \quad (\mathrm{or} \quad ( B t + \frac{A}{2} t^2) e^{\delta t} \quad \delta=-\alpha) \\
C(t) &=& C (e^{-\alpha t}-e^{\delta t})+ 2 B t e^{-\alpha t} + A t^2 e^{-\alpha t} 
\quad (\mathrm{or} \quad (C t +B t^2 + \frac{A}{3} t^3) e^{\delta t} \quad \delta=-\alpha)
\\
a(t) &=& a (1-e^{\delta t}) - 2 A (e^{-\alpha t}- e^{\delta t})\\
&&  \quad \mathrm{(or} \quad at- 2 A(e^{-\alpha t}-1) \quad \mathrm{for} \quad \delta=0,\\
&&  \quad \mathrm{(or} \quad a (1-e^{\delta t}) - 2 A t e^{\delta t} \quad \mathrm{for} \quad \delta=-\alpha) \\
b(t) &=& b (e^{\alpha t}-e^{\delta t}) - a (1-e^{\delta t}) + A(e^{-\alpha t}-e^{\delta t})\\
&& \quad \mathrm{(or} \quad b (e^{\alpha t}-1) - a t + A(e^{-\alpha t}-1) \quad \mathrm{for} \quad \delta=0, \\
&& \quad \mathrm{or} \quad b (e^{-\delta t}-e^{\delta t}) - a (1-e^{\delta t}) + A t e^{\delta t}  \quad \mathrm{for} \quad \delta=-\alpha, \\
&& \quad \mathrm{or} \quad b t e^{\delta t} - a (1-e^{\delta t}) + A t e^{\delta t}  \quad \mathrm{for} \quad \delta=\alpha) \\
c(t) &=& c(1-e^{\delta t})+a t - 2 A t e^{-\alpha t}+2 B (1-e^{-\alpha t}) \\
&& \quad \mathrm{(or} \quad c t + \frac{a}{2} t^2 - 2 A t e^{-\alpha t}+2 B (1-e^{-\alpha t}) \quad \mathrm{for} \quad \delta=0,\\
&& \quad \mathrm{or} \quad c (1-e^{\delta t}) + at - (2 B t + A t^2)e^{\delta t} \quad \mathrm{for} \quad \delta=-\alpha),
\end{eqnarray*}
which yields flows 
\begin{itemize}
\item[$\bullet$] $\varphi_t(x)=[e^{\delta t}x_0x_2+\big(a(1-e^{\delta t})-2A(e^{-\alpha t}-e^{\delta t}) \big)x_1x_2+
\big(b(e^{\alpha t}-e^{\delta t})-a(1-e^{\delta t})+A(e^{-\alpha t}-e^{\delta t}) \big)x_2^2+
\big(c(1-e^{\delta t})+a t -2 A t e^{-\alpha t}+ 2 B (1-e^{-\alpha t}) \big)x_2x_3+
A(e^{-\alpha t}-e^{\delta t})x_1^2+2 \big(B(e^{-\alpha t}-e^{\delta t})+At e^{-\alpha t}\big)x_1x_3+
\big(C(e^{-\alpha t}-e^{\delta t})+At^2e^{-\alpha t}+2Bt e^{-\alpha t} \big)x_3^2, (x_1+(e^{\alpha t}-1)x_2+tx_3)x_2,
e^{\alpha t} x_2^2, x_2x_3]$ with $\delta \neq 0$, $\alpha\neq 0$,
\item[$\bullet$] $\varphi_t(x)=[x_0x_2+\big(a t-2A(e^{-\alpha t}-1) \big)x_1x_2+
\big(b(e^{\alpha t}-1)-a t+A(e^{-\alpha t}-1) \big)x_2^2+
\big(c t+\frac{a}{2}t^2-2 t A e^{-\alpha t}+ 2(1-e^{-\alpha t})B \big)x_2x_3+
A(e^{-\alpha t}-e^{\delta t})x_1^2+2 \big(B(e^{-\alpha t}-1)+At e^{-\alpha t}\big)x_1x_3+
\big(C(e^{-\alpha t}-1)+At^2e^{-\alpha t}+2Bt e^{-\alpha t} \big)x_3^2, (x_1+(e^{\alpha t}-1)x_2+tx_3)x_2,
e^{\alpha t} x_2^2, x_2x_3]$ with $\alpha \neq 0$,
\item[$\bullet$] $\varphi_t(x)=[e^{\delta t}x_0x_2-2A t e^{\delta t} x_1x_2+
\big(b (e^{-\delta t}-e^{\delta t})-a(1-e^{\delta t})+A t e^{\delta t} \big)x_2^2+
\big(c(1-e^{\delta t})+a t -2 A t e^{-\delta t}+ (2Bt+A t^2)e^{\delta t} \big)x_2x_3+
A t e^{\delta t} x_1^2+ \big(2B t + A t^2\big) e^{\delta t} x_1x_3+(C t + B t^2 + \frac{A}{3} t^3)e^{\delta t}
x_3^2, (x_1+(e^{\delta t}-1)x_2+tx_3)x_2,
e^{\delta t} x_2^2, x_2x_3]$ with $\delta \neq 0$,
\item[$\bullet$] $\varphi_t(x)=[e^{\delta t}x_0x_2+\big(a(1-e^{\delta t})-2A(e^{-\delta t}-e^{\delta t}) \big)x_1x_2+
\big(b t e^{\delta t}-a(1-e^{\delta t})+A(e^{-\delta t}-e^{\delta t}) \big)x_2^2+
\big(c(1-e^{\delta t})+a t -2 A t e^{-\delta t}+ 2 B (1-e^{-\delta t}) \big)x_2x_3+
A(e^{-\delta t}-e^{\delta t})x_1^2+2 \big(B(e^{-\delta t}-e^{\delta t})+At e^{-\delta t}\big)x_1x_3+
\big(C(e^{-\delta t}-e^{\delta t})+At^2e^{-\delta t}+2Bt e^{-\delta t} \big)x_3^2, (x_1+(e^{\delta t}-1)x_2+tx_3)x_2,
e^{\delta t} x_2^2, x_2x_3]$ with $\delta \neq 0$.
\end{itemize}
\medskip
\item[\bf iv)] One can assume that $\alpha\neq 0$ for otherwise we are in case \textbf{i)} for $\ell=x_3$.
We obtain the equations:
\begin{eqnarray*}
A(t+s)-e^{\delta s} A(t) &=& e^{\alpha t} A(s) \\
B(t+s)-e^{\delta s} B(t) &=& B(s)  \\
C(t+s)-e^{\delta s} C(t) &=& e^{-\alpha t} C(s) \\
a(t+s)-e^{\delta s} a(t) &=& e^{\alpha t} a(s)+ 2 t e^{\alpha t} A(s) \\
b(t+s)-e^{\delta s} b(t) &=& e^{\alpha t} b(s) + t e^{\alpha t} a(s) + t^2 e^{\alpha t} A(s) \\
c(t+s)-e^{\delta s} c(t) &=& c(s) + 2 t B(s).
\end{eqnarray*}
Therefore:
\begin{eqnarray*}
A(t) &=& A (e^{\alpha t}-e^{\delta t}) \quad \mathrm{(or} \quad A t e^{\delta t} \quad \mathrm{for} \quad \delta=\alpha) \\
B(t) &=& B (1-e^{\delta t}) \quad \mathrm{(or} \quad B t \quad \mathrm{for} \quad \delta=0)\\
C(t) &=& C (e^{-\alpha t}-e^{\delta t}) \quad \mathrm{(or} \quad C t e^{\delta t} \quad \mathrm{for} \quad \delta=-\alpha) \\
a(t) &=& a (e^{\alpha t}-e^{\delta t}) + 2 A t e^{\alpha t} \quad \mathrm{(or} \quad (at+A t^2) e^{\delta t} \quad \mathrm{for} \quad \delta=\alpha)\\
b(t) &=& b (e^{\alpha t}-e^{\delta t}) + a t e^{\alpha t} + A t^2 e^{\alpha t} \quad \mathrm{(or} \quad (b t + \frac{a}{2} t^2 + 
\frac{A}{3} t^3) e^{\delta t} \quad \mathrm{for} \quad \delta=\alpha) \\
c(t) &=& c(1-e^{\delta t})+ 2 t B  \quad \mathrm{(or} \quad 
c t + B t^2 \quad \mathrm{for} \quad \delta=0),
\end{eqnarray*} 
which yield the flows 
\begin{itemize}
\item[$\bullet$] $\varphi_t(x)=[e^{\delta t} x_0x_2+  \big(a(e^{\alpha t}-e^{\delta t})+2 A t e^{\alpha t} \big)x_1x_2+
\big(b(e^{\alpha t}-e^{\delta t})+a t e^{\alpha t}+ A t^2 e^{\alpha t}\big) x^2_2+\big(c(1-e^{\delta t})+2 t B \big)x_2 x_3+ A(e^{\alpha t}-e^{\delta t})x_1^2+2 B (1-e^{\delta t})x_1 x_3+ C(e^{-\alpha t}-e^{\delta t})x_3^2, e^{\alpha t} (x_1+tx_2)x_2, e^{\alpha t} x_2^2, x_2x_3]$ with $\delta \neq 0$,
\item[$\bullet$] $\varphi_t(x)=[ x_0x_2+ \big(a(e^{\alpha t}-1)+2 A t e^{\alpha t} \big)x_1x_2+
\big(b(e^{\alpha t}-1)+a t e^{\alpha t}+ A t^2 e^{\alpha t}\big) x^2_2+(ct+B t^2) x_2x_3+ A(e^{\alpha t}-1)x_1^2+2 B t x_1 x_3+ C(e^{-\alpha t}-1)x_3^2, e^{\alpha t} (x_1+tx_2)x_2, e^{\alpha t} x_2^2, x_2x_3]$ with $\alpha \neq 0$, 
\item[$\bullet$] $\varphi_t(x)=[e^{\delta t} x_0x_2+  (a t + A t^2) e^{\delta t}  x_1x_2+
\big( bt+\frac{a}{2} t^2+ \frac{A}{3} t^3\big)e^{\delta t} x^2_2+\big(c(1-e^{\delta t})+2 t B \big)x_2 x_3+ A t e^{\delta t} x_1^2+2 B (1-e^{\delta t})x_1 x_3+ C(e^{-\delta t}-e^{\delta t})x_3^2, e^{\delta t} (x_1+tx_2)x_2, e^{\delta t} x_2^2, x_2x_3]$ with $\delta \neq 0$,
\item[$\bullet$] $\varphi_t(x)=[e^{\delta t} x_0x_2+  \big(a(e^{-\delta t}-e^{\delta t})+2 A t e^{-\delta t} \big)x_1x_2+
\big(b(e^{-\delta t}-e^{\delta t})+a t e^{-\delta t}+ A t^2 e^{-\delta t}\big) x^2_2+\big(c(1-e^{\delta t})+2 t B \big)x_2 x_3+ A(e^{-\delta t}-e^{\delta t})x_1^2+2 B (1-e^{\delta t})x_1 x_3+ C t e^{\delta t}x_3^2, e^{-\delta t} (x_1+tx_2)x_2, e^{-\delta t} x_2^2, x_2x_3]$ with $\delta \neq 0$.
\end{itemize}
\end{itemize}
\medskip
\item $\ell=x_1$ and $\Psi_t$ of type:
\begin{itemize}
\medskip
 \item[\bf i)] Then $\lambda(t)=e^{\alpha t}$ and 
$$q_t(x)=x_1 v_t + \tilde{q}_t(x)$$
where $\tilde{q}_t\in A_2(x_2,x_3)$ and $v_t\in A_1(x_1,x_2,x_3)$. Equation \ref{eq:LaEqB} is written as
\begin{equation}
x_1(v_{t+s} -e^{\delta s} v_t) + (\tilde{q}_{t+s} -e^{\delta s} \tilde{q}_t)=x_1 v_s({\ell_i}_t) + e^{-\alpha t} \tilde{q}_s({\ell_i}_t).
\end{equation}
We denote $\tilde{q}_t=A(t) x_2^2+ 2 B(t) x_2x_3+ C(t) x_3^2$ and $v_t(x)=a(t) x_1+ b(t) x_2 + c(t) x_3$. Then
equation \ref{eq:hier} is written as 
\begin{eqnarray*}
A(t+s)-e^{\delta s} A(t) &=& e^{-\alpha t} A(s) \\
B(t+s)-e^{\delta s} B(t) &=& e^{-\alpha t}(B(s) + t A(s)) \\
C(t+s)-e^{\delta s} C(t) &=& e^{-\alpha t} (C(s) + 2 t B(s)+ t^2 A (s)) \\
a(t+s)-e^{\delta s} a(t) &=& e^{\alpha t} a(s) \\
b(t+s)-e^{\delta s} b(t) &=& b(s) \\
c(t+s)-e^{\delta s} c(t) &=& c(s) + t b(s).
\end{eqnarray*}
with solutions:
\begin{eqnarray*}
A(t) &=& A (e^{-\alpha t}-e^{\delta t}) \quad \mathrm{(or} \quad A t e^{\delta t} \quad \mathrm{for} \quad \delta=-\alpha)  \\
B(t) &=& B (e^{-\alpha t}-e^{\delta t})+ A t e^{-\alpha t} \quad \mathrm{(or} \quad (B t+ \frac{A}{2} t^2)e^{\delta t} \quad \mathrm{for} \quad \delta=-\alpha)  \\
C(t) &=& C (e^{-\alpha t}-e^{\delta t})+ 2 B t e^{-\alpha t} + A t^2 e^{-\alpha t} \quad \mathrm{(or} \quad 
(C t + B t^2 + \frac{A}{3} t^3 )e^{\delta t} \quad \mathrm{for} \quad \delta=-\alpha)  \\
a(t) &=& a (e^{\alpha t }-e^{\delta t})  \quad \mathrm{(or} \quad ate^{\delta t} \quad \mathrm{for} \quad \alpha=\delta)  \\
b(t) &=& b (1-e^{\delta t})  \quad \mathrm{(or} \quad 
b t \quad \mathrm{for} \quad \delta=0) \\
c(t) &=& c(1-e^{\delta t})+b t \quad \mathrm{(or} \quad 
c t +\frac{b}{2} t^2 \quad \mathrm{for} \quad \delta=0),
\end{eqnarray*}
which yield the flows
\begin{itemize}
\item[$\bullet$]  $\varphi_t(x)=[e^{\delta t}x_0x_1+ a(e^{\alpha t}-e^{\delta t})x_1^2+b(1-e^{\delta t})x_1x_2+(c(1-e^{\delta t})+tb)x_1x_3+
A(e^{-\alpha t}-e^{\delta t})x_2^2+2(B(e^{-\alpha t}-e^{\delta t})+A t e^{-\alpha t}) x_2x_3+ \big( C(e^{-\alpha t}-e^{\delta t})+(At^2+2 t B)e^{-\alpha t} \big)x_3^2, 
e^{\alpha t} x_1^2, (x_2+tx_3)x_1,x_3x_1]$ with $\delta \neq 0$, 
\item[$\bullet$]  $\varphi_t(x)=[e^{\delta t}x_0x_1+ a(e^{-\delta t}-e^{\delta t})x_1^2+b(1-e^{\delta t})x_1x_2+(c(1-e^{\delta t})+tb)x_1x_3
+A t e^{\delta t} x_2^2+2(B t + \frac{A}{2} t^2)e^{\delta t} x_2x_3+ 
(C t + B t^2 + \frac{A}{3} t^3)e^{\delta t} x_3^2, e^{-\delta t} x_1^2, (x_2+tx_3)x_1,x_3x_1]$ with $\delta \neq 0$, 
\item[$\bullet$]  $\varphi_t(x)=[e^{\delta t}x_0x_1+ a t e^{\delta t}x_1^2+b(1-e^{\delta t})x_1x_2+(c(1-e^{\delta t})+tb)x_1x_3+
A(e^{-\delta t}-e^{\delta t})x_2^2+2(B(e^{-\delta t}-e^{\delta t})+A t e^{\delta t}) x_2x_3+ 
\big( C(e^{-\delta t}-e^{\delta t})+(At^2+2 t B)e^{-\delta t} \big)x_3^2, e^{\delta t} x_1^2, (x_2+tx_3)x_1,x_3x_1]$ with $\delta \neq 0$, 
\item[$\bullet$]  $\varphi_t(x)=[x_0x_1+ a(e^{\alpha t}-1)x_1^2+b t x_1x_2+(c t+\frac{b}{2}t^2)x_1x_3+A(e^{-\alpha t }-1)x_2^2+
2(B(e^{-\alpha t}-1)+A t e^{-\alpha t}) x_2x_3
+\big( C(e^{-\alpha t}-1)+(At^2+2 t B)e^{-\alpha t} \big)x_3^2, e^{\alpha t} x_1^2, (x_2+tx_3)x_1,x_3x_1]$ with $\alpha \neq 0$,
\item[$\bullet$]  $\varphi_t(x)=[x_0x_1+ a t x_1^2+b t x_1x_2+(c t+\frac{b t^2}{2})x_1x_3+A t x_2^2+
\big(2B t+A t^2\big) x_2x_3+\big(C t+ B t^2 + \frac{A}{3} t^3\big)x_3^2, e^{\alpha t} x_1^2, (x_2+tx_3)x_1,x_3x_1]$.
\end{itemize}
\end{itemize}
\end{enumerate}

\bibliographystyle{amsplain}
\bibliography{cremona}

\providecommand{\bysame}{\leavevmode\hbox to3em{\hrulefill}\thinspace}
\providecommand{\MR}{\relax\ifhmode\unskip\space\fi MR }
\providecommand{\MRhref}[2]{%
  \href{http://www.ams.org/mathscinet-getitem?mr=#1}{#2}
}
\providecommand{\href}[2]{#2}
\begin{thebibliography}{10}

\bibitem{Albe}
M.~Alberich-Carrami{\~n}ana, \emph{Geometry of the plane {C}remona maps},
  Lecture Notes in Mathematics, vol. 1769, Springer-Verlag, Berlin, 2002.

\bibitem{BayBeau}
L.~Bayle and A.~Beauville, \emph{Birational involutions of {${\bf P}^2$}},
  Asian J. Math. \textbf{4} (2000), no.~1, 11--17, Kodaira's issue.

\bibitem{Bert}
E.~Bertini, \emph{Richerche sulle trasformazioni univoche involutorie nel
  piano}, Annali di Mat. \textbf{8} (1877), 244--286.

\bibitem{BPV}
J.~Blanc, I.~Pan, and T.~Vust, \emph{Sur un th\'eor\`eme de {C}astelnuovo},
  Bull. Braz. Math. Soc. \textbf{39}, no.~1.

\bibitem{CanFav}
S.~Cantat and C.~Favre, \emph{Sym\'etries birationnelles des surfaces
  feuillet\'ees}, J. Reine Angew. Math. \textbf{561} (2003), 199--235.

\bibitem{Cast}
G.~Castelnuovo, \emph{Sulle trasformazioni cremoniane del piano, che ammettono
  una curva fiera}, Rend. Accad. Lincei (1892).

\bibitem{CerDes}
D.~Cerveau and J.~D\'eserti, \emph{Transformations birationnelles de petit
  degr\'e}, arxiv:0811.2325[math.AG].

\bibitem{Dema}
M.~Demazure, \emph{Sous-groupes alg\'ebriques de rang maximum du groupe de
  {C}remona}, Ann. Sci. \'Ecole Norm. Sup. (4) \textbf{3} (1970), 507--588.

\bibitem{Des1}
J.~D\'eserti, \emph{Groupe de cremona et dynamique complexe: une approche de la
  conjecture de {Z}immer}, Int. Math. Res. Not. \textbf{Art. ID 71701.} (2006).

\bibitem{DolgIsk}
I.V. Dolgachev and V.~A. Iskovskikh, \emph{Groupe de cremona et dynamique
  complexe: une approche de la conjecture de {Z}immer}, Int. Math. Res. Not.
  \textbf{Art. ID 71701.} (2006).

\bibitem{EnrFan}
F.~Enriques and G.~Fano, \emph{Sui gruppi di transformazioni cremoniane dell
  spazio}, Annali di Mat\'ematica pura ed applicata, s. 2a \textbf{15} (1897),
  221--271.

\bibitem{Fan}
G.~Fano, \emph{I gruppi di {J}onqui\`eres generalizzati}, Atti della R. Acc. di
  Torino \textbf{33} (1898), 59--98.

\bibitem{God}
L.~Godeaux, \emph{Les transformations birationnelles du plan}, M\'emorial des
  {S}ciences {M}ath\'ematiques, vol. 122, Gauthier-Villars, Paris, 1953.

\bibitem{Hudson}
H.~Hudson, \emph{Cremona transformations in plane and space}, Cambridge
  University Press, London, 1927.

\bibitem{Isk}
V.~A. Iskovskih, \emph{Proof of a theorem on relations in the two-dimensional
  cremona group}, Uspekhi Mat. Nauk. \textbf{40} (1985), no.~5(245), 255--256.

\bibitem{Kan}
S.~Kantor, \emph{Theorie der endlichen gruppen von eindeutigen transformationen
  in der ebene}, Mayer \& Mullen, 1895.

\bibitem{Pan}
I.~Pan, \emph{Une remarque sur la g\'en\'eration du groupe de {C}remona}, Bol.
  Soc. Brasil. Mat. (N.S.) \textbf{30} (1999), no.~1, 95--98.

\bibitem{PRV}
I.~Pan, F.~Ronga, and T.~Vust, \emph{Transformations birationnelles
  quadratiques de l'espace projectif complexe \`a trois dimensions}, Ann. Inst.
  Fourier (Grenoble) \textbf{51} (2001), no.~5, 1153--1187.

\bibitem{Ume}
H.~Umemura, \emph{Maximal algebraic subgroups of the {C}remona group of three
  variables. {I}mprimitive algebraic subgroups of exceptional type}, Nagoya
  Math. J. \textbf{87} (1982), 59--78.

\bibitem{Ume1}
\bysame, \emph{On the maximal connected algebraic subgroups of the {C}remona
  group. {I}}, Nagoya Math. J. \textbf{88} (1982), 213--246.

\bibitem{Ume2}
\bysame, \emph{On the maximal connected algebraic subgroups of the {C}remona
  group. {II}}, Algebraic groups and related topics (Kyoto/Nagoya, 1983), Adv.
  Stud. Pure Math., vol.~6, North-Holland, Amsterdam, 1985, pp.~349--436.

\end{thebibliography}
\end{document}